\documentclass{amsart}
\usepackage{amsxtra, microtype}
\usepackage{stmaryrd}
\usepackage{mathrsfs}
\usepackage{amssymb,amsmath,stmaryrd,latexsym,wasysym}
\usepackage[all]{xy}
\usepackage{hyperref}
\usepackage[margin=1.45in]{geometry}
\usepackage{tikz}
\usetikzlibrary{decorations.pathmorphing}
\usetikzlibrary{arrows}
\usepackage{cancel}
\usepackage{appendix}
\usepackage{xcolor}
\usepackage{xspace}

\newtheorem{theorem}{Theorem}[section] 
\newtheorem{lemma}[theorem]{Lemma}     
\newtheorem{corollary}[theorem]{Corollary}
\newtheorem{proposition}[theorem]{Proposition}
\newtheorem{remark}[theorem]{Remark}
\newtheorem{definition}[theorem]{Definition}
\newtheorem{example}[theorem]{Example}
\newcommand{\lie}[1]{\mathfrak{#1}}

\newcommand{\duer}{^*_}
\DeclareMathOperator{\ann}{ann}

\newcommand{\R}{\mathbb{R}} 
 
\newcommand{\inv}{^{-1}}
\newcommand{\N}{\mathbb{N}} 

\newcommand{\mx}{\mathfrak{X}} 
\newcommand{\dr}{\mathbf{d}}
\newcommand{\ldr}[1]{{{\pounds}}_{#1}}
\newcommand{\ip}[1]{{\mathbf{i}}_{#1}}
\newcommand{\an}[1]{\arrowvert_{#1}} 
\newcommand{\llb}{\llbracket} 
\newcommand{\rrb}{\rrbracket}
\newcommand{\Beta}{\boldsymbol{\beta}}

\DeclareMathOperator{\pr}{pr}
\DeclareMathOperator{\Hom}{Hom}

\DeclareMathOperator{\Id}{Id}

\DeclareMathOperator{\der}{Der}
\newcommand{\nsp}[2]  {%
\langle\mspace{-6.8mu}%
\langle\mspace{-6.8mu}%
\langle\mspace{-6.8mu}%
\langle\mspace{-6.8mu}%
\langle\mspace{-6.8mu}%
\langle\mspace{-6.8mu}%
\langle{#1,\,}{#2}%
\rangle%
\mspace{-6.8mu}\rangle%
\mspace{-6.8mu}\rangle%
\mspace{-6.8mu}\rangle%
\mspace{-6.8mu}\rangle%
\mspace{-6.8mu}\rangle%
\mspace{-6.8mu}\rangle}

\allowdisplaybreaks
\title{The geometrisation of $\N$-manifolds of degree
  $2$.}

\author{M. Jotz Lean}

\begin{document}
\maketitle
\begin{abstract}
  This paper describes an equivalence of the canonical category of
  $\N$-manifolds of degree $2$ with a category of involutive double
  vector bundles. More precisely, we show how involutive double vector
  bundles are in duality with double vector bundles endowed with a
  linear metric.  We describe then how special sections of the metric
  double vector bundle that is dual to a given involutive double
  vector bundle are the generators of a graded manifold of degree $2$
  over the double base.

  We discuss how split Poisson $\N$-manifolds of degree $2$ are
  equivalent to \emph{self-dual representations up to homotopy} and
  so, following Gracia-Saz and Mehta, to linear splittings of a
  certain class of VB-algebroids. In other words, the equivalence of
  categories above induces an equivalence between so called
  \emph{Poisson involutive double vector bundles}, which are the dual
  objects to metric VB-algebroids, and Poisson $\N$-manifolds of degree $2$.\\

\end{abstract}

\tableofcontents

\section{Introduction}

Graded manifolds are omnipresent objects in the field of Poisson
geometry.  In this context, they were first considered in the early
00's in Voronov's study of Lie bialgebroids \cite{Voronov02} and in
Roytenberg's supergeometric approach to Courant algebroids
\cite{Roytenberg02}, see also \cite{Severa05}. One of our goals in
this series of papers is to make more precise the $\N$-geometric
approach to Courant algebroids, VB-algebroids and LA-courant
algebroids \cite{Li-Bland12}.

An $\N$-graded manifold over a smooth manifold $M$ is a sheaf of
$\N$-graded, graded commutative, associative, unital
$C^\infty(M)$-algebras over $M$, that is locally freely generated by
finitely many elements of strictly positive degree.  $\N$-manifolds of
degree $1$, and geometric stuctures thereon, are fully understood as
the exterior algebras of sections of smooth vector bundles. Such a
geometric understanding of $\N$-graded manifold of higher degrees does
not exist in the literature.  In the degree $2$ case, graded
manifolds with a homological vector field (Lie $2$-algebroids) were
linked to VB-Courant algebroids by Li-Bland in his thesis
\cite{Li-Bland12}, building up on the correspondence of Courant
algebroids with symplectic Lie 2-algebroids \cite{Roytenberg02,
  Severa05}: the cotangent space of a Lie $2$-algebroid becomes a
symplectic Lie $2$-algebroid with its canonical symplectic
structure. The linear properties of the cotangent space constuction
translate to an additional linear structure on the obtained Courant
algebroid.

While this nice idea is very simple to understand, the obtained
correspondence lacks very much of the concreteness of the
well-understood correspondence of Lie 1-algebroids with Lie
algebroids, and it gives little geometric insight on the meaning of
the generators of the graded algebra in terms of the double vector
bundle underlying the VB-Courant algebroid. Furthermore, morphisms are
not studied in this correspondence.

This paper remedies to this by geometrising $\N$-manifolds of degree
$2$ via a certain class of double vector bundles -- that was already
defined and considered by Pradines in his first work on double vector
bundles and nonholonomic jets \cite{Pradines77} -- and explaining the
full and rich picture behind Li-Bland's observations.  Our main result
is an explicit equivalence between the category of degree $2$
$\N$-manifolds with a category of double vector bundles with identical
sides and an involution exchanging the sides and restricting to minus
the identity on the core. Such double vector bundles were called
`symmetric double vector bundles with inverse symmetry'' by Pradines
\cite{Pradines77}.  For simplicity, we call them here
\emph{involutive double vector bundles}. We prove that the dual
objects are double vector bundles endowed with a linear metric. 

In particular, split $\N$-manifolds of degree $2$ are equivalent in
this manner to \emph{involutive} splittings of involutive double
vector bundles -- ``symmetric charts'' in \cite{Pradines77} -- or
equivalently to \emph{Lagrangian} splittings of the dual metric double vector
bundles.  Our approach is a classical one; an extension of the
construction of vector bundles over a manifold $M$ from free and
locally finitely generated sheaves of $C^\infty(M)$-modules, using the
double vector bundle charts in \cite{Pradines77}.

Let us stress out here that while the equivalence of metric double
vector bundles with $[2]$-manifolds can be seen as a special case of
Li-Bland's correspondence \cite{Li-Bland12} of Lie $2$-algebroids with
VB-Courant algebroids -- namely the one of a trivial homological
vector field versus a trivial Courant bracket -- this example is
neither explored in \cite{Li-Bland12} nor completely straightfoward to
deduce from the proof of Li-Bland's correspondence. We believe that
this example, and in particular our new interpretation of metric
double vector bundles as the dual objects to involutive double vector
bundles, is in fact of significant importance since it could lead to
geometrisations of $\mathbb N$-manifolds of higher degrees via
relatively simple `classical' geometric objects.

From our main theorem follow many enlightening results on geometric
structures on degree $2$ $\N$-manifolds and on their counterparts on
metric double vector bundles: in sequels of this paper \cite{Jotz16a,Jotz15c} we recover for instance in a
constructive manner Li-Bland's equivalences
between the category of Lie $2$-algebroids and the category of
VB-Courant algebroids, and between the category of Poisson Lie
$2$-algebroids and the category of LA-Courant algebroids
\cite{Li-Bland12}, providing along the way a more handy
definition of the latter objects than the one already existing
\cite{Li-Bland12}. Most importantly, our geometric approach to graded
manifolds of degree $2$ allows
us to describe several new classes of examples of those objects, and later
to explain the links between $2$-representations and Lie
$2$-algebroids, and between the different notions of doubles
associated to Lie bialgebroids; namely the cotangent double and the
bicrossproduct Lie algebroids \cite{Jotz16a}.  
\medskip

In the second part of this paper we focus on Poisson structures of
degree $-2$ on $\N$-manifolds of degree $2$ (\emph{Poisson
  $\N$-manifolds of degree $2$}).  We show that their splittings are
the same as \emph{self-dual $2$-term representations up to
  homotopy}. We deduce from the equivalence of $2$-term
representations up to homotopy with linear splittings of VB-algebroids
\cite{GrMe10a} that the equivalence of categories in our main theorem
induces an equivalence of Poisson $\N$-manifolds of degree $2$ with
\emph{metric VB-algebroids}, and we explain the dual picture of
Poisson involutive double vector bundles. In particular, we find that
symplectic $[2]$-manifolds are equivalent in this manner to cotangent doubles of
metric vector bundles, a particular class of involutive double vector
bundles, together with, up to a sign, the Poisson structure induced by
the canonical symplectic form.

\medskip

 Note that Grabowski, Grabowska and
Bruce propose in \cite{BrGrGr14} an alternative geometric
characterisation of $\N$-manifolds via double graded bundles, which 
simplifies to our description in the degree $2$ case.
Note also that in his PhD thesis \cite{delCarpio-Marek15}, Fernando
del Carpio-Marek finds independently, mostly through different
methods, results that are similar to some of ours. Since his results
will not be published, we summarise where appropriate the main lines
of his approach and we bridge some of his results to ours.

\subsection*{Outline, main results and applications}
This paper is organised as follows.  

\textbf{Section \ref{sect:b+d}} starts by recalling how vector bundle
morphisms are equivalent to morphisms of the sheaves of sections of
the dual bundles.  We then discuss in detail the necessary background
on double vector bundles and their dualisations and splittings. We
recall how linear splittings of VB-algebroids induce $2$-term
representations up to homotopy \cite{GrMe10a}.

\textbf{Section \ref{sec:metDVB}} recalls the definition of
$\N$-manifolds and the equivalence of $\N$-manifolds of degree $1$
with vector bundles.  We define metric double vector bundles and their
Lagrangian splittings, the existence of which we prove. We define the
dual objects, the involutive double vector bundles, and we describe
their morphisms. We
construct an equivalence of this category with the category of
$\N$-manifolds of degree $2$.

Finally we discuss the geometric meaning of the generators of the
graded algebras in terms of functions on the corresponding involutive
double vector bundles: they can be understood as functions on the
involutive double vector bundles that are polynomial in their sides,
and on which the pullback under the involution is just multiplication
by $-1$.

\textbf{Section \ref{sec:Poisson}} studies Poisson structures of
degree $-2$ on $\N$-manifolds of degree $2$.  We show how a Poisson
structure of degree $-2$ on a split $\N$-manifold of degree $2$ is
equivalent to a $2$-term representation up to homotopy that is dual to
itself. Then we give the geometrisation of Poisson $\N$-manifolds of
degree $2$; namely linear Lie algebroids structures on metric double
vector bundles, that are compatible with the metric, or equivalently,
double linear Poisson structures on involutive double vector bundles,
such that the involution is anti-Poisson.  We prove that the
equivalence of categories established in the previous section induces
an equivalence of the category of Poisson $\N$-manifolds of degree $2$
with the category of Poisson involutive double vector
bundles. Finally, we discuss some examples of Poisson $\N$-manifolds
of degree $2$, and of the corresponding metric VB-algebroids and
Poisson involutive double vector bundles. We discuss in detail
symplectic $\N$-manifolds of degree $2$, and show how they correspond
to symplectic cotangent doubles of metric vector bundles.

\subsection*{Acknowledgements}

The author warmly thanks Henrique Bursztyn, Fernando del Carpio Marek,
David Li-Bland, Rajan Mehta, Dmitry Roytenberg, Arkady Vaintrob, Alan
Weinstein and Chenchang Zhu for interesting conversations or
comments. The author specially thanks Malte Heuer for his careful
reading, and  Rohan Jotz Lean for many very
useful comments on earlier versions of this work.

\subsection*{Notation and conventions}

We write $p_M\colon TM\to M$, $q_E\colon E\to M$ for vector bundle
projections and $\pi_A\colon D\to A$ and $\pi_B\colon D\to B$ for the
two ``top'' vector bundle projections of a double vector bundle.  For
a vector bundle $Q\to M$ we often identify without further mentioning
the vector bundle $(Q^*)^*$ with $Q$ via the canonical isomorphism. We
write $\langle\cdot\,,\cdot\rangle$ or $\langle\cdot\,,\cdot\rangle_M$
for the canonical pairing of a vector bundle with its dual;
i.e.~$\langle a_m,\alpha_m\rangle=\alpha_m(a_m)$ for $a_m\in A$ and
$\alpha_m\in A^*$. We use many different pairings; in general,
which pairing is used is clear from its arguments.  Given a
section $\varepsilon$ of $E^*$, we write $\ell_\varepsilon\colon E\to
\R$ for the linear function associated to it, i.e. the function
defined by $e_m\mapsto \langle \varepsilon(m), e_m\rangle$ for all
$e_m\in E$.

We assume all manifolds to be connected.
Let $M$ be a smooth manifold. We denote by $\mx(M)$ and
$\Omega^1(M)$ the sheaves of smooth sections of the tangent and
the cotangent bundle, respectively. For an arbitrary vector bundle
$E\to M$, the sheaf of sections of $E$ is written
$\Gamma(E)$.  Let $f\colon M\to N$ be a smooth map between two smooth
manifolds $M$ and $N$.  Then two vector fields $X\in\mx_U(M)$ and
$Y\in\mx_V(N)$ are said to be \textbf{$f$-related} if $Tf\circ X=Y\circ
f$ on $U\cap f\inv(V)$.  We write then $X\sim_f Y$.

We write ``$[n]$-manifold'' for ``$\N$-manifold of degree $n$''. 
We write
``$2$-representations'' for ``$2$-term representations up to homotopy''.

\section{Background and definitions on double vector bundles and VB-algebroids}
\label{sect:b+d}
We collect in this section background on vector bundles and their
morphisms, on double vector bundles and their linear splittings, on
VB-algebroids and their encoding by $2$-representations. Further
references will be given throughout the text.

\subsection{Vector bundles and morphisms}\label{usual_VB_morphisms}
Let $A\to M$ and $B\to N$ be vector bundles and $\omega\colon A\to B$
a morphism of vector bundles over a smooth map $\omega_0\colon M\to
N$.  First we introduce a few notations. We say that $a\in\Gamma_U(A)$
is $\omega$-related to $b\in\Gamma_V(B)$ if
\[ \omega\circ a=b\circ\omega_0
\] 
on $U\cap \omega_0\inv(V)$. We write then $a\sim_\omega b$.  We write
$\omega_0^*B\to M$ for the pullback of $B$ under $\omega_0$; for $m\in
M$, elements of $(\omega_0^*B)(m)$ are pairs $(m,b_{\omega_0(m)})$
with $b_{\omega_0(m)}\in B(\omega_0(m))$.

 The dual of a morphism
$\omega\colon A\to B$ over $\omega_0\colon M\to N$ is in general not a morphism
of vector bundles, but a morphism $\omega^\star$ of
modules over $\omega_0^*\colon C^\infty(N)\to C^\infty(M)$:
\begin{equation}\label{dual_of_VB_map}
\omega^\star\colon \Gamma(B^*)\to \Gamma(A^*),
\qquad  \omega^\star(\beta)(m)=\omega_m^*\beta_{\omega_0(m)}
\end{equation}
for all $\beta\in\Gamma(B^*)$ and $m\in M$. 
We prove the following lemma in Appendix \ref{morphisms_VB_sheaves}.
\begin{lemma}\label{bundlemap_eq_to_morphism}
  The map $\cdot^\star$, that sends a morphism of vector bundles
  $\omega\colon A\to B$ over $\omega_0\colon M\to N$ to the morphism
  $\omega^\star\colon\Gamma(B^*)\to\Gamma(A^*)$ of modules over
  $\omega_0^*\colon C^\infty(N)\to C^\infty(M)$, is a bijection.
\end{lemma}

Note finally that for $\beta\in\Gamma_N(B^*)$ and $\ell_\beta\in C^{\infty, \rm{lin}}(B)$, we have 
\begin{equation}\label{pullback_linear}
\ell_{\omega^\star\beta}=\omega^*\ell_\beta\in C^{\infty, \rm{lin}}(A).
\end{equation}

\subsection{Double vector bundles}
We briefly recall the definitions of double vector bundles, of their
\emph{linear} and \emph{core} sections, and of their \emph{linear
  splittings} and \emph{lifts}. We refer to
\cite{Pradines77,Mackenzie05,GrMe10a} for more details.

A \textbf{double vector bundle} $(D;A,B;M)$ is a smooth manifold $D$
with two vector bundle structures over the total spaces $A$ and $B$ of
two vector bundles with base $M$, such that the square
\begin{equation*}
\begin{xy}
\xymatrix{
D \ar[r]^{\pi_B}\ar[d]_{\pi_A}& B\ar[d]^{q_B}\\
A\ar[r]_{q_A} & M}
\end{xy}
\end{equation*}
of vector bundle projections is commutative and the following 
conditions are satisfied:
\begin{enumerate}
\item $\pi_B$ is a vector bundle morphism over $q_A$;
\item $+_B\colon D\times_B D \rightarrow D$ is a vector bundle
morphism over $+\colon A\times_M A \rightarrow A$, where $+_B$ is
the addition map for the vector bundle $D\rightarrow B$, and 
\item the scalar multiplication $\R\times D\to D$ in the bundle $D\to
  B$ is a vector bundle morphism over the scalar multiplication
  $\R\times A\to A$.
\end{enumerate}
(Note that the notation $(D;A,B;M)$ is omissive: all structure maps
are of course part of the data of a double vector bundle.)

The corresponding statements to (1)--(3) for the operations in the
bundle $D\to A$ follow.  Note that the condition that each addition in
$D$ is a morphism with respect to the other is exactly
\begin{equation}\label{add_add} (d_1+_Ad_2)+_B(d_3+_Ad_4)=(d_1+_Bd_3)+_A(d_2+_Bd_4)
\end{equation}
for $d_1,d_2,d_3,d_4\in D$ with $\pi_A(d_1)=\pi_A(d_2)$,
$\pi_A(d_3)=\pi_A(d_4)$ and $\pi_B(d_1)=\pi_B(d_3)$,
$\pi_B(d_2)=\pi_B(d_4)$.

Given a double vector bundle $(D; A, B; M)$, the vector bundles $A$
and $B$ are called the \textbf{side bundles}. The \textbf{core} $C$ of
a double vector bundle is the intersection of the kernels of $\pi_A$
and of $\pi_B$. By \eqref{add_add}, adding over $A$ or over $B$
elements of the core yields the same result, and $C$ gets a natural vector
bundle structure over $M$, the projection of which we call $q_C\colon
C \rightarrow M$. The inclusion $C \hookrightarrow D$ is usually
denoted by $C_m \ni c \longmapsto \overline{c} \in \pi_A^{-1}(0^A_m)
\cap \pi_B^{-1}(0^B_m)$.

Given a double vector bundle $(D;A,B;M)$, the space of sections
$\Gamma_B(D)$ is generated as a $C^{\infty}(B)$-module by two
distinguished classes of sections (see \cite{Mackenzie11}), the
\emph{linear} and the \emph{core sections} which we now describe.  For
a smooth section $c\colon M \rightarrow C$, the corresponding
\textbf{core section} $c^\dagger\colon B \rightarrow D$ is defined as
\begin{equation}\label{core_section}
c^\dagger(b_m) = 0^D_{b_m} +_A \overline{c(m)}, \,\, m \in M, \, b_m \in B_m.
\end{equation}
We denote the corresponding core section $A\to D$ by $c^\dagger$ also,
relying on the argument to distinguish between them. The space of core
sections of $D$ over $B$ is written $\Gamma_B^c(D)$.  A section
$\xi\in \Gamma_B(D)$ is called \textbf{linear} if $\xi\colon B
\rightarrow D$ is a bundle morphism from $B \rightarrow M$ to $D
\rightarrow A$ over a section $a\in\Gamma(A)$.  The space of linear
sections of $D$ over $B$ is denoted by $\Gamma^\ell_B(D)$. A section
$\psi\in \Gamma(B^*\otimes C)$ defines a linear section
$\tilde{\psi}\colon B\to D$ over the zero section $0^A\colon M\to A$
by
$\widetilde{\psi}(b_m) = 0^D_{b_m}+_A \overline{\psi(b_m)}$
for all $b_m\in B$.  We call $\widetilde{\psi}$ a \textbf{core-linear
  section}.

\begin{example}\label{trivial_dvb}
  Let $A, \, B, \, C$ be vector bundles over $M$ and consider
  $D=A\times_M B \times_M C$. With the vector bundle structures
  $D=q^{!}_A(B\oplus C) \to A$ and $D=q_B^{!}(A\oplus C) \to B$, one
  finds that $(D; A, B; M)$ is a double vector bundle called the
  \textit{decomposed double vector bundle with sides $A$ and $B$ and
    core $C$}. The core sections are given by
$c^\dagger\colon b_m \mapsto (0^A_m, b_m, c(m))$, where $m \in M$,  $b_m \in
B_m$, $c \in \Gamma(C)$,
and similarly for $c^\dagger\colon A\to D$.  The space of linear
sections $\Gamma^\ell_B(D)$ is naturally identified with
$\Gamma(A)\oplus \Gamma(B^*\otimes C)$ via
$$
(a, \psi): b_m \mapsto (a(m), b_m, \psi(b_m)), \text{ where } \psi \in
\Gamma(B^*\otimes C), \, a\in \Gamma(A).
$$

In particular, the fibered product $A\times_M B$ is a double vector
bundle over the sides $A$ and $B$ and its core is the trivial bundle
over $M$.
\end{example}
\begin{definition}
  Let $(D_1; A_1, B_1; M_1)$ and $(D_2; A_2, B_2; M_2)$ be two double
  vector bundles. A double vector bundle morphism $(\Psi; \psi_A,
  \psi_B; \psi_0)$ from $D_1$ to $D_2$ is a commutative cube
{\small\begin{equation*}
  \begin{xy}
    \xymatrix{D_1\ar[rrr]^{\Psi}\ar[rd]^{\pi_{A_1}}\ar[dd]_{\pi_{B_1}}& && D_2\ar[rd]^{\pi_{A_2}}\ar[dd]
      &\\
& A_1\ar[dd]\ar[rrr]^{\psi_A}&&&A_2\ar[dd]^{q_{A_2}}\\
      B_1\ar[rd]_{q_{B_1}}\ar[rrr]^{\psi_B}&&&B_2\ar[rd]^{q_{B_2}}&\\
&M_1\ar[rrr]^{\psi_0}&&&M_2}
\end{xy}
\end{equation*}}
where all the faces are vector bundle morphisms. 
\end{definition}

Given a double vector bundle morphism $(\Psi; \psi_A, \psi_B;
\psi_0)$, its restriction to the core bundles induces a vector bundle
morphism $\psi_c\colon C_1 \to C_2$:
\[\Psi(\overline{\tau})=\overline{\psi_c(\tau)}
\] 
for all $\tau\in C_1$.  $\psi_C$ is called the \textbf{core morphism} of
$\Psi$.

Note that given $c\in\Gamma(C_2)$, we have
$\Psi^\star(c^\dagger)=(\psi_c^\star(c))^\dagger$ for
$c^\dagger\in \Gamma^c_{A_2}(D_2)$ or
$c^\dagger\in \Gamma^c_{B_2}(D_2)$. If $\chi\in\Gamma_{A_2}^l(D_2)$ is
linear over $b\in\Gamma(B_2)$, then
$\Psi^\star(\chi)\in\Gamma_{A_1}^l(D_1)$ is linear over
$\psi_B^\star(b)$. Similarly, if $\chi\in\Gamma_{B_2}^l(D_2)$ is
linear over $a\in\Gamma(A_2)$, then
$\Psi^\star(\chi)\in\Gamma_{B_1}^l(D_1)$ is linear over
$\psi_A^\star(a)$.  Furthermore, we have
$\Psi^\star(q_{A_2}^*f\cdot\chi)=q_{A_1}^*(\psi_0^*f)\cdot\Psi^\star(\chi)$
for all $f\in C^\infty(M_2)$ and $\chi \in \Gamma_{A_2}^l(D_2)$.

\subsubsection{Linear splittings and lifts}
\label{subsub:lsl}
A \textbf{linear splitting} of $(D; A, B; M)$ is an injective morphism
of double vector bundles $\Sigma\colon A\times_M B\hookrightarrow D$
over the identity on the sides $A$ and $B$.  That every double vector
bundle admits local linear splittings was proved by \cite{GrRo09}  (see
 also \cite{delCarpio-Marek15} for a more elementary proof).  Local
linear splittings are equivalent to double vector bundle
charts. Pradines originally defined double vector bundles as
topological spaces with an atlas of double vector bundle charts
\cite{Pradines74a} (see Definition \ref{double_atlas}). Using a
partition of unity, he proved that (provided the double base is a
smooth manifold) this implies the existence of a global double
splitting \cite{Pradines77}. Hence, any double vector bundle in the
sense of our definition admits a (global) linear splitting.

Note that a linear splitting of $D$ is equivalent to a
\textbf{decomposition} of $D$, i.e.~an isomorphism $\mathbb I\colon
A\times_MB\times_MC\to D$ of double vector bundles over the identities
on the sides and inducing the identity on the core. Given a linear
splitting $\Sigma$, the corresponding decomposition $\mathbb I$ is
given by $\mathbb I(a_m,b_m,c_m)=\Sigma(a_m,b_m)+_B(\tilde
0_{\vphantom{1}_{b_m}} +_A \overline{c_m})$.  Given a decomposition
$\mathbb I$, the corresponding linear splitting $\Sigma$ is given by
$\Sigma(a_m,b_m)=\mathbb I(a_m,b_m,0^C_m)$.

\medskip

A linear splitting $\Sigma$ of $D$ is also equivalent to a splitting
$\sigma_A$ of the short exact sequence of $C^\infty(M)$-modules
\begin{equation}\label{fat_seq_gamma}
0 \longrightarrow \Gamma(B^*\otimes C) \hookrightarrow \Gamma^\ell_B(D) 
\longrightarrow \Gamma(A) \longrightarrow 0,
\end{equation}
where the first map sends $\phi\in\Gamma(B^*\otimes C)$ to
$\tilde\phi\in\Gamma^\ell_B(D)$ and the third map is the map that
sends a linear section $(\xi,a)$ to its base section $a\in\Gamma(A)$.
The splitting $\sigma_A$ is called a \textbf{horizontal lift}. Given
$\Sigma$, the horizontal lift $\sigma_A\colon \Gamma(A)\to
\Gamma_B^\ell(D)$ is given by $\sigma_A(a)(b_m)=\Sigma(a(m), b_m)$ for
all $a\in\Gamma(A)$ and $b_m\in B$.  By the symmetry of a linear
splitting, we find that a lift $\sigma_A\colon
\Gamma(A)\to\Gamma_B^\ell(D)$ is equivalent to a lift $\sigma_B\colon
\Gamma(B)\to \Gamma_A^\ell(D)$.  Given a lift
$\sigma_A\colon\Gamma(A)\to\Gamma^\ell_B(D)$, the corresponding lift
$\sigma_B\colon\Gamma(B)\to\Gamma^\ell_A(D)$ is given by
$\sigma_B(b)(a(m))=\sigma_A(a)(b(m))$ for all $a\in\Gamma(A)$,
$b\in\Gamma(B)$.

Note that two linear splittings $\Sigma^1,\Sigma^2\colon
A\times_MB\to D$ differ by a section $\phi_{12}$ of $A^*\otimes
B^*\otimes C\simeq \operatorname{Hom}(A,B^*\otimes C)\simeq
\operatorname{Hom}(B,A^*\otimes C)$ in the following sense.  For each
$a\in\Gamma(A)$ the difference $\sigma_A^1(a)-_B\sigma_A^2(a)$ of
horizontal lifts is the core-linear section defined by
$\phi_{12}(a)\in\Gamma(B^*\otimes C)$. By symmetry,
$\sigma_B^1(b)-_A\sigma_B^2(b)=\widetilde{\phi_{12}(b)}$ for each
$b\in\Gamma(B)$.


\subsubsection{The tangent double of a vector bundle}\label{tangent_double}
Let $q_E\colon E\to M$ be a vector bundle.  Then the tangent bundle
$TE$ has two vector bundle structures; one as the tangent bundle of
the manifold $E$, and the second as a vector bundle over $TM$. The
structure maps of $TE\to TM$ are the derivatives of the structure maps
of $E\to M$. The space $TE$ is a double vector bundle with core bundle
$E \to M$. \begin{equation*}
\begin{xy}
\xymatrix{
TE \ar[d]_{Tq_E}\ar[r]^{p_E}& E\ar[d]^{q_E}\\
 TM\ar[r]_{p_M}& M}
\end{xy}
\end{equation*}
The map $\bar{}\,\colon E\to p_E^{-1}(0^E)\cap (Tq_E)^{-1}(0^{TM})$
sends $e_m\in E_m$ to $\bar
e_m=\left.\frac{d}{dt}\right\an{t=0}te_m\in T_{0^E_m}E$.  The core
vector field corresponding to $e \in \Gamma(E)$ is the vertical lift
$e^{\uparrow}\colon E \to TE$, i.e.~the vector field with flow
$\phi\colon E\times \R\to E$, $\phi(e'_m, t)=e'_m+te(m)$. An element
of $\Gamma^\ell_E(TE)=\mx^\ell(E)$ is called a \textbf{linear vector
  field}. Since its flow is a flow of vector bundle morphisms, a
linear vector field sends linear functions to linear functions and
pullbacks to pullbacks. It is well-known (see e.g.~\cite{Mackenzie05})
that a linear vector field $\xi\in\mx^l(E)$ covering $X\in\mx(M)$ is
equivalent to a derivation $\delta_\xi^*\colon \Gamma(E^*) \to \Gamma(E^*)$
over $X\in \mx(M)$, and hence to the dual derivation
$\delta_\xi\colon\Gamma(E)\to\Gamma(E)$ over $X\in \mx(M)$. The precise
correspondence is given by
\begin{equation}\label{ableitungen}
\xi(\ell_{\varepsilon}) 
= \ell_{\delta_\xi^*(\varepsilon)} \,\,\,\, \text{ and }  \,\,\, \xi(q_E^*f)= q_E^*(X(f))
\end{equation}
for all $\varepsilon\in\Gamma(E^*)$ and $f\in C^\infty(M)$.  
We write $\widehat \delta$ for the linear vector field
in $\mx^l(E)$ corresponding in this manner to a derivation $\delta$ of
$\Gamma(E)$.  The choice of a linear splitting $\Sigma$ for $(TE; TM,
E; M)$ is equivalent to the choice of a connection on $E$: Since a
linear splitting gives us a linear vector field
$\sigma_{TM}(X)\in\mx^l(E)$ for each $X\in \mx(M)$, we can define
$\nabla\colon \mx(M)\times\Gamma(E)\to \Gamma(E)$ by
$\sigma_{TM}(X)=\widehat{\nabla_X}$ for all $X\in\mx(M)$. Conversely,
a connection $\nabla\colon \mx(M)\times\Gamma(E)\to\Gamma(E)$ defines
a lift $\sigma_{TM}^\nabla$ for $(TE; TM, E; M)$ and a linear
splitting $\Sigma^\nabla\colon TM\times_M E \to TE$.

We recall as well the relation between the connection and the Lie
bracket of vector fields on $E$.  Given $\nabla$, it is easy to see
using the equalities in \eqref{ableitungen} that, writing $\sigma$ for
$\sigma_{TM}^\nabla$:
\begin{equation}\label{Lie_bracket_VF}
  \left[\sigma(X), \sigma(Y)\right]=\sigma[X,Y]-\widetilde{R_\nabla(X,Y)},\qquad
  \left[\sigma(X), e^\uparrow\right]=(\nabla_Xe)^\uparrow,\qquad
  \left[e^\uparrow,e'^\uparrow\right]=0,
\end{equation}
for all $X,Y\in\mx(M)$ and $e,e'\in\Gamma(E)$.  That is, the Lie
bracket of vector fields on $M$ and the connection encode completely
the Lie bracket of vector fields on $E$.

\medskip

Now let us have a quick look at the other structure on the double
vector bundle $TE$. The lift
$\sigma_{E}^\nabla\colon\Gamma(E)\to\Gamma_{TM}^\ell(TE)$ is given by
\begin{equation*}
  \sigma_{E}^\nabla(e)(v_m) = T_me(v_m) +_{TM} (T_m0^E(v_m) -_E \overline{\nabla_{v_m} e}), \,\, v_m \in TM, \, e \in \Gamma(E).
\end{equation*}
Further, for $e\in\Gamma(E)$, the core section $e^\dagger\in\Gamma_{TM}(TE)$
is given by
\begin{equation*}
 e^\dagger(v_m)=T_m0^E(v_m)+_E\left.\frac{d}{dt}\right\an{t=0}te(m).
\end{equation*}

\subsubsection{Dualisation and lifts}\label{dual}
Double vector bundles can be dualised in two distinct ways.  We denote
by $D\duer A$ the dual of $D$ as a vector bundle over $A$ and likewise
for $D\duer B$. The dual $D\duer A$ is again a double vector
bundle, with
side bundles $A$ and $C^*$ and core $B^*$
\cite{Mackenzie99,Mackenzie11}.$$ 
{\xymatrix{
    D\ar[r]^{\pi_B}\ar[d]_{\pi_A}&   B\ar[d]^{q_{B}}\\
    A\ar[r]_{q_A}                   &  M\\
  }} \qquad\qquad {\xymatrix{
    D\duer A \ar[r]^{\pi_{C^*}}\ar[d]&   C^*\ar[d]^{q_{C^*}}\\
    A\ar[r]_{q_{A}}                   &  M\\
  }} \qquad\qquad {\xymatrix{
    D\duer B \ar[r]\ar[d]&   B\ar[d]^{q_B}\\
    C^*\ar[r]_{q_{C^*}}                   &  M\\
  }}
$$  The projection $\pi_{C^*}\colon D\duer A\to C^*$ is defined as
follows: if $\Phi\in D\duer A$ projects to $\pi_A(\Phi)=a_m$, then
$\pi_{C^*}(\Phi)\in C^*_m$ is defined by
$\pi_{C^*}(\Phi)(c_m)=\Phi(0^D_{a_m}+_B\overline{c_m})$ for all
$c_m\in C_m$. The addition in the fibers of the vector bundle $D\duer
A\to C^*$ is defined as follows: if $\Phi_1$ and $\Phi_2\in D\duer A$
satisfy $\pi_{C^*}(\Phi_1)=\pi_{C^*}(\Phi_2)$, $\pi_A(\Phi_1)=a^1_m$
and $\pi_A(\Phi_2)=a^2_m$, then $\Phi_1+_{C^*}\Phi_2$ is defined by
\[(\Phi_1+_{C^*}\Phi_2)(d_1+_Bd_2)=\Phi_1(d_1)+\Phi_2(d_2)
\]
for all $d_1,d_2\in D$ with $\pi_B(d_1)=\pi_B(d_2)$ and
$\pi_A(d_1)=a^1_m$, $\pi_A(d_2)=a^2_m$.  The core element
$\overline{\beta_m}\in D\duer A$ defined by $\beta_m\in B^*$ is given
by $\overline{\beta_m}(d)=\beta_m(\pi_B(d))$ for all $d\in D$ with
$\pi_A(d)=0^A_m$.  By playing with the vector bundle structures on
$D\duer A$ and \eqref{add_add}, one can show that each core element of
$D\duer A$ is of this form. We encourage the reader who is not
familiar with the dualisations of double vector bundles to check this,
and also to find out where the projection to $C^*$ is relevant in the
definition of the addition over $C^*$.

Given a linear splitting $\Sigma\colon A\times_M B\to D$ the ``dual''
linear spliting $\Sigma^\star\colon A\times_M C^*\to D\duer A$ is
defined by
\begin{equation}\label{lemma_dual_splitting}
\langle \Sigma^\star(a_m,\gamma_m),\Sigma(a_m,b_m)\rangle_A=0, \qquad
\langle
\Sigma^\star(a_m,\gamma_m),c^\dagger(a_m)\rangle_A=\langle\gamma_m,
c(m)\rangle
\end{equation}
for all $a_m\in A$, $c\in\Gamma(C)$, $b_m\in B$ and $\gamma_m\in C^*$.
(See \cite{GrJoMaMe14} for a more complicated construction of the dual
splitting. We let the reader check that the two constructions yield
the same splitting.) 

\subsubsection{Canonical (up to sign) pairing of $D\duer A$ with $D\duer B$}\label{fat_pairing_def}
The vector bundles $D\duer A\to C^*$ and $D\duer B\to C^*$ are, up to
a sign, naturally in duality to each other \cite{Mackenzie05}. The
pairing
\[ \nsp{\cdot\,}{\cdot} \colon D\duer A\times_{C^*} D\duer B\to \mathbb R
\]  
is defined as follows: 
for $\Phi\in D\duer A$ and $\Psi\in D\duer B$ projecting to the same element
$\gamma_m$ in $C^*$, choose $d\in D$ with $\pi_A(d)=\pi_A(\Phi)$ and
$\pi_B(d)=\pi_B(\Psi)$.  Then $\langle \Phi, d\rangle_A-\langle \Psi,d\rangle_B=:\nsp{\Phi}{\Psi}$ 
does not depend on the choice of $d$.
This implies in particular that $D\duer A$ is canonically (up to a
sign) isomorphic to $(D\duer B)\duer {C^*}$ and $D\duer B$ is isomorphic
to $(D\duer A)\duer {C^*}$.

\subsection{VB-algebroids}
\label{subsect:VBa}
Let $(D; A, B; M)$ be a double vector bundle with core $C$.
Then $(D \to B; A \to M)$ is a \textbf{VB-algebroid}
(\cite{Mackenzie98x}; see also \cite{GrMe10a}) if $D \to B$ has a Lie
algebroid structure the anchor of which is a bundle morphism
$\Theta_B\colon D \to TB$ over $\rho_A\colon A \to TM$ and such that
the Lie bracket is linear:
\begin{equation*} [\Gamma^\ell_B(D), \Gamma^\ell_B(D)] \subset
  \Gamma^\ell_B(D), \qquad [\Gamma^\ell_B(D), \Gamma^c_B(D)] \subset
  \Gamma^c_B(D), \qquad [\Gamma^c_B(D), \Gamma^c_B(D)]= 0.
\end{equation*}
The vector bundle $A\to M$ is then also a Lie algebroid, with anchor
$\rho_A$ and bracket defined as follows: if $\xi_1,
\xi_2\in\Gamma^\ell_B(D)$ are linear over $a_1,a_2\in\Gamma(A)$, then
the bracket $[\xi_1,\xi_2]$ is linear over $[a_1,a_2]$.  We also say
that the Lie algebroid structure on $D\to B$ is linear over the Lie
algebroid $A\to M$.

Since the anchor $\Theta_B$ is linear, it sends a core section
$c^\dagger$, $c\in\Gamma(C)$ to a vertical vector field on $B$.  This
defines the \textbf{core-anchor} $\partial_B\colon C\to B$; for
$c\in\Gamma(C)$ we have $\Theta_B(c^\dagger)=(\partial_Bc)^\uparrow$
(see \cite{Mackenzie92}).

\begin{example}\label{td}
  It is easy to see from the considerations in \S\ref{tangent_double}
  that the tangent double $(TE;E,TM;M)$ of a vector bundle $E\to M$
  has a VB-algebroid structure \linebreak $(TE\to E, TM\to M)$. 
\end{example}

\subsection{Representations up to homotopy}

Let $A\to M$ be a Lie algebroid and consider an $A$-connection
$\nabla$ on a vector bundle $E\to M$.  Then the space
$\Omega^\bullet(A,E)$ of $E$-valued Lie algebroid forms has an induced
degree $1$ operator $\dr_\nabla$ given by:
\begin{equation*}
\begin{split}
  \dr_\nabla\omega(a_1,\ldots,a_{k+1})=&\sum_{i<j}(-1)^{i+j}\omega([a_i,a_j],a_1,\ldots,\hat a_i,\ldots,\hat a_j,\ldots, a_{k+1})\\
  &\qquad +\sum_i(-1)^{i+1}\nabla_{a_i}(\omega(a_1,\ldots,\hat
  a_i,\ldots,a_{k+1}))
\end{split}
\end{equation*}
for all $\omega\in\Omega^k(A,E)$ and $a_1,\ldots,a_{k+1}\in\Gamma(A)$.
We have $\dr_\nabla(\alpha \wedge \omega) = \dr_A\alpha \wedge \omega
+(-1)^{|\alpha|} \alpha \wedge \dr_\nabla\omega$ for $\alpha \in
\Gamma(\wedge A^*)$ and $\omega \in \Omega(A,E)$ and $\dr_\nabla^2=0$ if
and only if the connection $\nabla$ is flat; that is, if and only if
$\nabla$ defines a representation of $A$ on $E$.  Let $\mathcal E=
\bigoplus_{k\in \mathbb{Z}} E_k[k]$ be now a graded vector
bundle. Consider the space $\Omega(A,\mathcal E)$ with grading given
by
$
\Omega(A,\mathcal E)[k] = \bigoplus_{i+j=k}\Omega^i(A, E_j)
$.

\begin{definition}\cite{ArCr12}\cite{GrMe10a} 
  A \emph{representation up to homotopy of $A$ on $\mathcal E$} is a
  map \linebreak$\mathcal D\colon \Omega(A, \mathcal E) \to \Omega(A,\mathcal
  E)$ with total degree $1$ and such that $\mathcal D^2=0$ and
\[\mathcal D(\alpha \wedge \omega) = \dr_A\alpha \wedge \omega +
(-1)^{|\alpha|} \alpha \wedge \mathcal D(\omega),\] for $\alpha
\in \Gamma(\wedge A^*)$, $\omega \in \Omega(A,\mathcal E)$,
where $\dr_A\colon \Gamma(\wedge A^*) \to \Gamma(\wedge A^*)$ is the
Lie algebroid differential.
\end{definition}

Let $A$ be a Lie algebroid. The representations up to homotopy which
we consider are always on graded vector bundles $\mathcal E=
E_0[0]\oplus E_1[1]$ concentrated on degrees 0 and 1, so called
\emph{$2$-term graded vector bundles}.  These representations are
equivalent to the following data (see \cite{ArCr12,GrMe10a}):
\begin{enumerate}
\item [(1)] a vector bundle morphism $\partial\colon E_0\to E_1$,
\item [(2)] two $A$-connections, $\nabla^0$ and $\nabla^1$ on $E_0$
  and $E_1$, respectively, such that $\partial \circ \nabla^0 =
  \nabla^1 \circ \partial$, \item [(3)] an element $R \in \Omega^2(A,
  \Hom(E_1, E_0))$ such that $R_{\nabla^0} = R\circ \partial$,
  $R_{\nabla^1}=\partial \circ R$ and $\dr_{\nabla^{\Hom}}R=0$,
  where $\nabla^{\Hom}$ is the connection induced on $\Hom(E_1,E_0)$
  by $\nabla^0$ and $\nabla^1$.
\end{enumerate}
For brevity we call such a 2-term representation up to homotopy a
\textbf{2-re\-pre\-sen\-ta\-tion}.

\subsection{2-representations and VB-algebroids}
\label{subsect:ruths}
Let $(D\to B, A\to M)$ be a VB-algebroid and choose a linear splitting
$\Sigma\colon A\times_MB\to D$. Since the anchor of a linear section
is linear, for each $a\in \Gamma(A)$ the vector field
$\Theta_B(\sigma_A(a))$ defines a derivation of $\Gamma(B)$ with
symbol $\rho(a)$ (see \S \ref{tangent_double}). This defines a linear
connection $\nabla^{B}\colon \Gamma(A)\times\Gamma(B)\to\Gamma(B)$:
\[\Theta_B(\sigma_A(a))=\widehat{\nabla_a^{B}}\]
for all $a\in\Gamma(A)$.  Since the bracket of a linear section with a
core section is again a core section, we find a linear connection
$\nabla^{C}\colon\Gamma(A)\times\Gamma(C)\to\Gamma(C)$ such
that \[[\sigma_A(a),c^\dagger]=(\nabla_a^{C}c)^\dagger\] for all
$c\in\Gamma(C)$ and $a\in\Gamma(A)$.  The difference
$\sigma_A[a_1,a_2]-[\sigma_A(a_1), \sigma_A(a_2)]$ is a core-linear
section for all $a_1,a_2\in\Gamma(A)$.  This defines a vector valued
Lie algebroid form \linebreak $R\in\Omega^2(A,\operatorname{Hom}(B,C))$ such that
\[[\sigma_A(a_1), \sigma_A(a_2)]=\sigma_A[a_1,a_2]-\widetilde{R(a_1,a_2)},
\]
for all $a_1,a_2\in\Gamma(A)$. See \cite{GrMe10a} for more details on
these constructions.  The following theorem is proved in
\cite{GrMe10a}.
\begin{theorem}\label{rajan}
  Let $(D \to B; A \to M)$ be a VB-algebroid and choose a linear
  splitting $\Sigma\colon A\times_MB\to D$.  The triple
  $(\nabla^{B},\nabla^{C},R)$ defined as above is a
  $2$-representation of $A$ on the complex $\partial_B\colon C\to B$,
  where $\partial_B$ is the core-anchor.

  Conversely, let $(D;A,B;M)$ be a double vector bundle such that $A$
  has a Lie algebroid structure and choose a linear splitting
  $\Sigma\colon A\times_MB\to D$. Then if
  $(\nabla^{B},\nabla^{C},R)$ is a 2-representation of $A$ on a
  complex $\partial_B\colon C\to B$, then the three equations above
  and the core-anchor $\partial_B$ define a VB-algebroid structure on
  $(D\to B; A\to M)$.

\end{theorem}

 In the situation of the previous theorem, we have
\begin{equation*}
\left[\sigma_A(a),\widetilde\phi\right]=\widetilde{\nabla_a^{\rm Hom}\phi}
\qquad \text{ and }\qquad 
\left[c^\dagger,\widetilde\phi\right]=(\phi(\partial_Bc))^\dagger
\end{equation*}
for all $a\in\Gamma(A)$, $\phi\in\Gamma(\operatorname{Hom}(B,C))$ and
$c\in\Gamma(C)$, see for instance \cite{GrJoMaMe14}.

\begin{remark}\label{change}
  If $\Sigma_1,\Sigma_2\colon A\times_M B\to D$ are two linear
  splittings of a VB-algebroid $(D\to B, A\to M)$ and
  $\phi_{12}\in\Gamma(A^*\otimes B^*\otimes C)$ is the change of
  splitting, then the two corresponding 2-representations are related
  by the following identities \cite{GrMe10a}.
\begin{equation}\label{nablas}
\nabla^{B,2}_a=\nabla^{B,1}_a+\partial_B\circ \phi_{12}(a), 
\quad \nabla^{C,2}_a=\nabla^{C,1}_a+\phi_{12}(a)\circ\partial_B\end{equation}
and
\begin{equation}\label{curv_new}
\begin{split}
R^2(a_1,a_2)=&R^1(a_1,a_2)+(\dr_{\nabla^{\operatorname{Hom}}}\phi_{12})(a_1,a_2)\\
&\qquad+\phi_{12}(a_1)\partial_B\phi_{12}(a_2)-\phi_{12}(a_2)\partial_B\phi_{12}(a_1)
\end{split}
\end{equation}
for all $a,a_1,a_2\in\Gamma(A)$. 

Given a 2-representation $\mathcal D$ of $A$ on $E_0[0]\oplus E_1[1]$ and
a tensor $\phi\in\Gamma(A^*\otimes E_1^*\otimes E_0)$, we say that
the new 2-representation defined by \eqref{nablas} and
\eqref{curv_new} is the \textbf{$\phi$-twist of $\mathcal D$}.
\end{remark}

\begin{example}\label{double_ruth}
  Choose a linear connection
  $\nabla\colon\mx(M)\times\Gamma(E)\to\Gamma(E)$ and consider the
  corresponding linear splitting $\Sigma^\nabla$ of $TE$ as in \S
  \ref{tangent_double}.  The description of the Lie bracket of vector
  fields in \eqref{Lie_bracket_VF} shows that the 2-representation
  induced by $\Sigma^\nabla$ is the 2-representation of $TM$ on
  $\Id_E\colon E\to E$ given by $(\nabla,\nabla,R_\nabla)$.
\end{example}

\begin{example}[The tangent of a Lie algebroid]\label{tangent_double_al}
Let $(A\to M,\rho,[\cdot\,,\cdot])$ be a Lie algebroid. Then the
tangent $TA\to TM$ has a Lie algebroid structure with bracket defined
by $[Ta_1, Ta_2]=T[a_1,a_2]$, $[Ta_1, a_2^\dagger]=[a_1,a_2]^\dagger$
and $[a_1^\dagger, a_2^\dagger]=0$ for all $a_1,a_2\in\Gamma(A)$. The
anchor of $Ta$ is $\widehat{[\rho(a),\cdot]}\in\mx(TM)$ and the anchor
of $a^\dagger$ is $\rho(a)^\uparrow$ for all $a\in\Gamma(A)$.  This
defines a VB-algebroid structure $(TA\to TM; A\to M)$ on
$(TA;TM,A;M)$.

Given a $TM$-connection on $A$, and so a linear splitting
$\Sigma^\nabla$ of $TA$ as in \S \ref{tangent_double}, the
2-representation of $A$ on $\rho\colon A\to TM$ encoding this
VB-algebroid is the \textbf{adjoint $2$-representation}
$(\nabla^{\rm bas},\nabla^{\rm bas}, R_\nabla^{\rm bas})$ \cite{GrMe10a}, where the
connections are defined by
\begin{equation*}
\begin{split}
\nabla^{\rm bas}&\colon \Gamma(A)\times\mx(M)\to\mx(M), \qquad \nabla^{\rm
    bas}_aX=[\rho(a), X]+\rho(\nabla_Xa)
\end{split}
\end{equation*} 
and 
\begin{equation*}
\begin{split}
  \nabla^{\rm bas}&\colon
  \Gamma(A)\times\Gamma(A)\to\Gamma(A), \qquad \nabla^{\rm bas}_{a_1}a_2=[a_1,a_2]+\nabla_{\rho(a_2)}a_1,
\end{split}
\end{equation*}  and $R_\nabla^{\rm
    bas}\in \Omega^2(A,\operatorname{Hom}(TM,A))$ is given by
  \[R_\nabla^{\rm
    bas}(a_1,a_2)X=-\nabla_X[a_1,a_2]+[\nabla_Xa_1,a_2]+[a_1,\nabla_Xa_2]+\nabla_{\nabla^{\rm
      bas}_{a_2}X}a_1-\nabla_{\nabla^{\rm bas}_{a_1}X}a_2
\]
for all $X\in\mx(M)$, $a, a_1,a_2\in\Gamma(A)$.
\end{example}

\subsubsection{Dualisation and $2$-representations}\label{dual_and_ruths}
Let $(D\to B, A\to M)$ be a VB-algebroid. Then $(D\duer A\to C^*, A\to
M)$ has in induced VB-algebroid structure \cite{Mackenzie11}. While
this can be defined in an abstract and natural manner (i.e.~without
the use of splittings), we characterise for simplicity the linear Lie
algebroid structure on $D\duer A\to C^*$ using Theorem \ref{rajan}.

Let $\Sigma\colon A\times_MB\to D$ be a linear splitting of $D$ and
denote by $(\nabla^B,\nabla^C,R)$ the induced 2-representation of the
Lie algebroid $A$ on $\partial_B\colon C\to B$. We have seen above
that the linear splitting $\Sigma$ induces a linear splitting
$\Sigma^\star\colon A\times_M C^*\to D\duer A$.  The induced
VB-algebroid structure on $(D\duer A\to C^*, A\to M)$ is given in this
splitting by the 2-representation $({\nabla^C}^*,{\nabla^B}^*,-R^*) $
of $A$ on the complex $\partial_B^*\colon B^*\to C^*$. This is proved
in the appendix of \cite{DrJoOr15}. (Note that the construction of the
``dual'' linear splitting of $D\duer A$, given a linear splitting of
$D$, is done in \cite{DrJoOr15} by dualising the corresponding
decomposition and taking its inverse. The resulting linear splitting
of $D\duer A$ is the same as ours.)

\section{{[2]}-manifolds and metric double vector bundles}\label{sec:metDVB}
In this section we recall the definitions of $\N$-manifolds of degree
$2$ and of their morphisms. Then we introduce linear metrics on double
vector bundles, and the dual objects, involutive double vector
bundles. We define morphisms of involutive double vector bundles and
we prove our main result: an equivalence between the category of
$\N$-manifolds of degree $2$ and the obtained category of involutive
double vector bundles.

We illustrate the theory with two standard classes of metric double vector bundles: the tangent
double $TE\to TM$ of a metric vector bundle $E$, and the Pontryagin bundle
$TE\oplus T^*E\to E$ of a vector bundle $E$. We describe the
dual involutive double vector bundles.
\subsection{$\N$-manifolds}
Here we give the definitions of $\N$-manifolds. We are particularly
interested in $\N$-manifolds of degree $2$. We refer to \cite{Mehta06,BoPo13} for
more details.

\begin{definition}\label{n_man} An \textbf{ $\N$-manifold}  $\mathcal M$ of degree $n$ and dimension $(m;
  r_1,\ldots, r_n)$ is a sheaf of
$\N$-graded, graded commutative, associative, unital
$C^\infty(M)$-algebras over a smooth $m$-dimensional manifold $M$, that is locally freely generated by
$r_1+\ldots+ r_n$
  elements
  $\xi_1^{1},\ldots,\xi_1^{r_1}$, $\xi_2^1,\ldots,\xi_2^{r_2},\ldots$,
  $\xi_n^1,\ldots,\xi_n^{r_n}$ with $\xi_i^j$ of degree $i$ for
  $i\in\{1,\ldots,n\}$ and $j\in\{1,\ldots,r_i\}$.

  A morphism of $\N$-manifolds $\mu\colon\mathcal N\to
  \mathcal M$ over a smooth map $\mu_0\colon N\to M$ of the underlying
  smooth manifolds is a morphism $\mu^\star\colon C^\infty(\mathcal
  M)\to C^\infty(\mathcal N)$ of sheaves of graded algebras over
  $\mu_0^*\colon C^\infty(M)\to C^\infty(N)$.
\end{definition}

Note that the degree $0$ elements of $C^\infty(\mathcal M)$ are
precisely the smooth functions on $M$.  We call
\textbf{$[n]$-manifold} an $\N$-manifold of degree $n\in \N$. We write
$|\xi|$ for the degree of a homogeneous element $\xi\in
C^\infty(\mathcal M)$, i.e.~an element which can be written as a sum
of functions of the same degree and we write $C^\infty(\mathcal M)^i$
for the elements of degree $i$ in $C^\infty(\mathcal M)$.  Note that a
\textbf{$[1]$-manifold} over a smooth manifold $M$ is equivalent to a
locally free and finitely generated sheaf of $C^\infty(M)$-modules.

\bigskip Our goal in this section is to prove that $[2]$-manifolds are
equivalent to double vector bundles endowed with a linear metric
(Theorem \ref{main_crucial}).  We begin with a few observations on the
equivalence of smooth vector bundles with locally free and finitely
generated sheaves of $C^\infty$-modules, i.e.~$[1]$-manifolds. Theorem \ref{main_crucial}
will generalise this result to the degree $2$ case.

\subsubsection{Vector bundles and $[1]$-manifolds}\label{classical_eq}
Here we recall the equivalence of categories between degree
$[1]$-manifolds (or locally free and finitely generated sheaves of
$C^\infty$-modules) and smooth vector bundles (see for instance
\cite[Theorem II.1.13]{Wells08}). This subsection can be seen as
introductory to the methods in \S \ref{eq_2_manifolds}.

Let $\operatorname{VB}$ be the category of smooth vector bundles. Let $E\to M$ and $F\to N$ be vector
bundles. Recall from Lemma
\ref{bundlemap_eq_to_morphism} that a morphism $\phi\colon F\to E$ of vector
bundles over
$\phi_0\colon N\to M$ is equivalent to a map
$\phi^\star\colon \Gamma(E^*)\to\Gamma(F^*)$ defined as in
\eqref{dual_of_VB_map} and satisfying
\[\phi^\star(f\cdot \varepsilon)=\phi_0^*f\cdot\phi^\star(\varepsilon)
\]
for all $f\in C^\infty(M)$ and $\varepsilon\in \Gamma(E^*)$.

Let $\operatorname{[1]-Man}$ be the category of $[1]$-manifolds. 
We now establish an equivalence between $\operatorname{VB}$ and
$\operatorname{[1]-Man}$.  The functor $\Gamma(\cdot)\colon
\operatorname{VB}\to\operatorname{[1]-Man}$ sends a vector bundle
$E\to M$ to the set of sections $\Gamma(E^*)$, a locally free and
finitely generated sheaf of $C^\infty(M)$-modules. We call this
$[1]$-manifold $E[-1]$. The functor $\Gamma(\cdot)$ sends a morphism
$\Phi=(\phi,\phi_0)\colon F\to E$ as above to the morphism
$\phi^\star\colon \Gamma(E^*)\to \Gamma(F^*)$ over $\phi_0^*\colon
C^\infty(M)\to C^\infty(N)$, defining a morphism $\Gamma(\Phi)\colon
F[-1]\to E[-1]$ of $[1]$-manifolds.

\medskip Next choose a $[1]$-manifold $\mathcal M$ over a smooth
manifold $M$.  There exists a maximal open covering $\{U_\alpha\}$ of $M$
such that $C^\infty_{U_\alpha}(\mathcal M)$ is finitely generated by
generators $\xi^\alpha_1,\ldots,\xi^\alpha_m$. For two indices
$\alpha,\beta$ such that $U_\alpha\cap U_\beta\neq\emptyset$, we can
write each generator in an unique manner as
$\xi^\beta_j=\sum_{i=1}^m\psi_{\alpha\beta}^{ij}\xi^\alpha_i$ with
smooth functions $\psi_{\alpha\beta}^{ij}\in C^\infty(U_\alpha\cap
U_\beta)$.  We define $A_{\alpha\beta}\in C^\infty(U_\alpha\cap
U_\beta, \operatorname{Gl}(\R^m))$ by
$A_{\alpha\beta}=(\psi_{\alpha\beta}^{ij})_{i,j=1,\ldots,m}$.  We
have then immediately \begin{equation} \label{cocycle}
  A_{\gamma\alpha}\cdot A_{\alpha\beta}=A_{\gamma\beta},
\end{equation}
 where $\cdot$ is the pointwise
multiplication of matrices.  Next we consider the disjoint union
$\tilde E=\bigsqcup_{\alpha}U_\alpha\times \R^m$
and identify for $x\in U_\alpha\cap U_\beta\neq \emptyset$
\[(x,v)\in U_\beta\times \R^m \quad \text{ with }\quad (x, A_{\alpha\beta}(x)v)\in U_\alpha\times \R^m .
\]
By \eqref{cocycle}, this defines an equivalence relation on $\tilde E$
and the quotient $E$ has a smooth vector bundle
structure with vector bundle charts given by the inclusions
$U_\alpha\times \R^m\hookrightarrow E$, and changes of charts the
cocycles $A_{\alpha\beta}$.  We set $E(\mathcal M):=E^*$. Note that the maps $e_i^\alpha\colon
U_\alpha\to U_\alpha\times\R^m$, $x\mapsto (x,e_i)$ define smooth
local sections of $E$ and
$e_i^\beta=\sum_{j=1}^n\psi_{\alpha\beta}^{ji}e_j^\alpha$ for
$\alpha,\beta$ such that $U_\alpha\cap U_\beta\neq \emptyset$. Hence,
we can identify $\xi_i^\alpha$ with the section $e_i^\alpha$ and we
see that a morphism $\mu\colon\mathcal N\to \mathcal M$
over $\mu_0\colon N\to M$ defines a morphism $E(\mu)^\star\colon
\Gamma(E(\mathcal M)^*)\to\Gamma(E(\mathcal N)^*)$ of modules over
$\mu_0^*\colon C^\infty(M)\to C^\infty(N)$, and so by Lemma
\ref{bundlemap_eq_to_morphism} a vector bundle morphism $E(\mathcal
N)\to E(\mathcal M)$ over $\mu_0\colon N\to M$.  Hence we have
constructed a functor $E(\cdot)\colon
\operatorname{[1]-Man}\to\operatorname{VB}$.

\medskip

Next we show that the two functors are part of an equivalence of
categories.  The functor $E(\cdot)\circ \Gamma(\cdot)\colon
\operatorname{VB}\to\operatorname{VB}$ sends a vector bundle to the
abstract vector bundle defined by its trivialisations and cocycles.
There is an obvious natural isomorphism between this functor and the
identity functor $\operatorname{VB}\to\operatorname{VB}$.
The functor $\Gamma(\cdot)\circ E(\cdot)\colon
\operatorname{[1]-Man}\to\operatorname{[1]-Man}$ sends a
$[1]$-manifold $\mathcal M$ over $M$ with local generators
$\xi^\alpha_i\in C^\infty_{U_\alpha}(\mathcal M)^1$ and
cocycles $A_{\alpha\beta}$ to the sheaf of sections of $E(\mathcal
M)^*$, with local basis sections $\varepsilon_i^\alpha\in \Gamma_{U_\alpha}(E(\mathcal
M)^*)$ and cocycles
$A_{\alpha\beta}$. There is an obvious natural isomorphism between this functor and the
identity functor $\operatorname{[1]-Man}\to\operatorname{[1]-Man}$.





\subsubsection{Split $\N$-manifolds}
Next we quickly discuss split $\N$-manifolds and we recall how each
$\N$-manifold is noncanonically isomorphic to a split $\N$-manifold of the
same degree and of the same dimension.

  Let $E$ be a smooth vector bundle of rank $r$ over a smooth manifold
  $M$ of dimension $p$ and assign the degree $n$ to the fiber coordinates
  of $E$. This defines $E[-n]$, an $[n]$-manifold of dimension
  $(p;r_1=0,\ldots,r_{n-1}=0,r_n=r)$ with
  $C^\infty(E[-n])^n=\Gamma(E^*)$.

 Now let $E_{1},E_{2},\ldots,E_{n}$ be smooth vector bundles
  of finite ranks $r_1,\ldots,r_n$ over $M$ and assign the degree $i$
  to the fiber coordinates of $E_{i}$, for each $i=1,\ldots,n$.  The
  direct sum $E=E_{1}\oplus \ldots\oplus E_{n}$ is a graded vector
  bundle with grading concentrated in degrees $1,\ldots,n$.  The
  $[n]$-manifold $E_{1}[-1]\oplus\ldots\oplus E_{n}[-n]$
  has local basis sections of ${E_{i}^*}$ as local generators of
  degree $i$, for $i=1,\ldots,n$, and so dimension $(p;r_1,\ldots,r_n)$. 
  The $[n]$-manifold $\mathcal M=E_{1}[-1]\oplus\ldots\oplus
  E_{n}[-n]$ is
  called a \textbf{split $[n]$-manifold}.

In this paper, we are exclusively interested in the cases $n=2$ and
$n=1$.  Choose two vector bundles $E_{1}$ and $E_{2}$ of ranks $r_1$
and $r_2$ over a smooth manifold $M$. Consider $\mathcal
M=E_{1}[-1]\oplus E_{2}[-2]$.  We find $C^\infty(\mathcal
M)^0=C^\infty(M)$, $C^\infty(\mathcal M)^1=\Gamma(E_{1}^*)$ and
$C^\infty(\mathcal M)^2=\Gamma(E_{2}^*\oplus \wedge^2E_{1}^*)$.

A morphism $\mu\colon F_{1}[-1]\oplus F_{2}[-2]\to
E_{1}[-1]\oplus E_{2}[-2]$ of split $[2]$-manifolds over the bases
$N$ and $M$, respectively, consists of a smooth map $\mu_0\colon N\to
M$, three vector bundle morphisms $\mu_1\colon F_{1}\to E_{1}$,
$\mu_2\colon F_{2}\to E_{2}$ and $\mu_{12}\colon \wedge^2F_{1}\to
E_{2}$ over $\mu_0$. The map $\mu^\star$ sends a degree $1$ function
$\xi\in\Gamma(E_{1}^*)$ to
 $\mu_1^\star\xi\in\Gamma(F_{1}^*)$
 and a degree $2$-function $\xi\in\Gamma(E_{2}^*)$ to
  \begin{equation*}
\mu_2^\star\xi+\mu_{12}^\star\xi\in\Gamma(F_{2}^*\oplus
  \wedge^2F_{1}^*).
\end{equation*}
Any $\N$-manifold is non-canonically diffeomorphic to a split
$\N$-manifold. Further, the categories of split $\N$-manifolds and of
$\N$-manifolds are equivalent.  This is proved for instance in
\cite{BoPo13}, following the proof of the $\mathbb Z/2\mathbb
Z$-graded version of this
theorem, which is known as \emph{Batchelor's theorem} \cite{Batchelor80}.

\begin{theorem}[\cite{Batchelor80,BoPo13}]\label{split_N} Any $[n]$-manifold
  is non-canonically diffeomorphic to a split $[n]$-manifold.
\end{theorem}

We give here the proof by \cite{BoPo13} in the case $n=2$. We are
especially interested in the morphism of split $[2]$-manifolds induced
by a change of splitting of a $[2]$-manifold and we emphasize this
in the proof.
\begin{proof}[Sketch of Proof, \cite{BoPo13}]
  Consider a $[2]$-manifold $\mathcal M$ over a smooth base manifold
  $M$. Since $C^\infty(\mathcal M)^0=C^\infty(M)$ and
  $C^\infty(\mathcal M)^0\cdot C^\infty(\mathcal M)^1\subset
  C^\infty(\mathcal M)^1$, the sheaf $C^\infty(\mathcal M)^1$ is a
  locally free and finitely generated sheaf of $C^\infty(M)$-modules
  and there exists a vector bundle $E_1\to M$ such that
  $C^\infty(\mathcal M)^1\simeq \Gamma(E_1^*)$. Now let ${\mathcal A}_1$
  be the subalgebra of $C^\infty(\mathcal M)$ generated by
  $C^\infty(\mathcal M)^0\oplus C^\infty(\mathcal M)^1$. We find
  easily that $\mathcal A_1\simeq \Gamma(\wedge^\bullet E_{1}^*)$ and
  ${\mathcal A}_{1}\cap C^\infty(\mathcal M)^2=\wedge^2C^\infty(\mathcal
  M)^1$ is a proper $C^\infty(M)$-submodule of $C^\infty(\mathcal
  M)^2$. Since the quotient $C^\infty(\mathcal M)^2/\wedge^2C^\infty(\mathcal
  M)^1$ is a locally free and finitely generated sheaf of
  $C^\infty(M)$-modules, we have  $C^\infty(\mathcal
  M)^2/\wedge^2C^\infty(\mathcal M)^1\simeq \Gamma(E_{2}^*)$, for a vector
  bundle $E_{2}$ over $M$. The short exact sequence \begin{equation*}
    0\rightarrow \wedge^2C^\infty(\mathcal M)^1\hookrightarrow
    C^\infty(\mathcal M)^2\rightarrow \Gamma(E_{2}^*)\rightarrow
    0\end{equation*} of $C^\infty(M)^0$-modules is non canonically
  split. Let us choose a splitting and identify $\Gamma(E_{2}^*)$
  with a submodule of $C^\infty(\mathcal M)^2$:
$$C^\infty(\mathcal M)^2\simeq \wedge^2C^\infty(\mathcal M)^1\oplus\Gamma(E_{2}^*)= \Gamma(\wedge^2E_{1}^*\oplus
E_{2}^*)\;.$$ Hence, the considered $[2]$-manifold is isomorphic, modulo
the chosen splitting, to the split $[2]$-manifold $E_{1}[-1]\oplus
E_{2}[-2]$.

Note finally that a change of splitting is equivalent to a section
$\phi$ of $\operatorname{Hom}(\wedge^2E_{1}, E_2)$ and
induces an isomorphism of split $[2]$-manifolds over the identity on
$M$: $\mu^\star(\xi)=\xi+\phi^\star\xi\in\Gamma(E_{2}^*\oplus
\wedge^2E_{1}^*)$ for all $\xi\in\Gamma(E_{2}^*)$ and
$\mu^\star(\xi)=\xi$ for all $\xi\in\Gamma(E_{1}^*)$.
\end{proof}

Note that $[1]$-manifolds are automatically split. As we have seen in
\S \ref{classical_eq}, $[1]$-manifolds are just vector bundles with a
degree shifting in the fibers, i.e.~$\mathcal M=E[-1]$ for some vector
bundle $E\to M$ and $C^\infty(\mathcal M)=\Gamma(\wedge^\bullet
E^*)$, the exterior algebra of $E$.

\subsubsection{Vector fields on $[n]$-manifolds.}
Let us quickly introduce the notion of \emph{vector field} on an
$\N$-manifold. Let $\mathcal M$ be an $[n]$-manifold.  A
\textbf{vector field of degree $j$} on $\mathcal M$ is a degree $j$
derivation $\phi$ of $C^\infty(\mathcal M)$:
$|\phi(\xi)|=j+|\xi|$
for a homogeneous element $\xi\in C^\infty(\mathcal M)$. We write
$\der(C^\infty(\mathcal M))$ for the sheaf of graded
derivations of $C^\infty(\mathcal M)$.

The vector fields on $\mathcal M$ and their Lie bracket defined 
by $[\phi,\psi]=\phi\psi-(-1)^{|\phi||\psi|}\psi\phi$
 satisfy the following 
conditions:
\begin{enumerate}
\item $\phi(\xi\eta)=\phi(\xi)\eta+(-1)^{|\phi||\xi|}\xi\phi(\eta)$,
\item $[\phi,\psi]=(-1)^{1+|\phi||\psi|}[\psi,\phi]$,
\item $[\phi,\xi\psi]=\phi(\xi)\psi+(-1)^{|\phi||\xi|}\xi[\phi,\psi]$,
\item $(-1)^{|\phi||\gamma|}[\phi,[\psi,\gamma]]+(-1)^{|\psi||\phi|}[\psi,[\gamma,\phi]]
+(-1)^{|\gamma||\psi|}[\gamma,[\phi,\psi]]=0$
\end{enumerate}
for $\phi,\psi,\gamma$ homogeneous elements of
$\der(C^\infty(\mathcal M))$ and $\xi,\eta$ homogeneous
elements of $C^\infty(\mathcal M)$.  For instance, given an open set
$U$ of $M$ where $C^\infty(\mathcal M)$ is freely generated by $\xi^i_j$ ,
the derivation
$\partial_{\xi^i_j}$ of $C^\infty_U(\mathcal M)$ sends $\xi^i_j$ to
$1$ and the other local generators to $0$. It is hence a derivation of
degree $-j$. $\der(C^\infty_U(\mathcal M))$ is
freeely generated as a $C^\infty_U(\mathcal M)$-module by $\partial_{x_k}$,
$k=1,\ldots,p$ and $\partial_{\xi^i_j}$, $j=1,\ldots,n$,
$i=1,\ldots,r_j$.

Finally note that if an $[n]$-manifold $\mathcal M$ splits as
$E_{1}[-1]\oplus E_{2}[-2]\oplus\ldots\oplus E_{n}[-n]$, then each
section $e$ of $E_j$ defines a derivation
$\hat{e}$ of degree $-j$ on $\mathcal M$:
$\hat{e}(f)=0$, $\hat{e}(\varepsilon_j^i)=\langle
e,\varepsilon_j^j\rangle$, and $\hat{e}(\varepsilon_k^i)=0$ for
$k\neq j$.  We find $\hat{e_j^i}=\partial_{\varepsilon_j^i}$ if
$\{e_j^1, \ldots,e_j^{r_j}\}$ is a local basis of $E_j$ and
$\{\varepsilon_j^1,\ldots,\varepsilon_j^{r_j}\}$ is the dual basis of
$E_j^*$.
Further, a vector field $\phi$ of degree $0$ on $\mathcal M$ can be
written as a sum
$X+\delta_1+\delta_2+\ldots+\delta_n$,
with $X\in \mx(M)$ and each $\delta_i$ a derivation of $E_i$ with symbol
$X\in\mx(M)$.  The derivation $X+\delta_1+\ldots+\delta_n$ sends
$\varepsilon_i\in\Gamma(E_i^*)$ to $\delta_i^*(\varepsilon_i)$.
In particular, if for each $j$ the map $\delta^j\colon \mx(M)\to\der(E_j)$ is a morphism of
$C^\infty(M)$-modules 
that sends a vector field $X$ to a derivation $\delta^j(X)$ over $X$, 
then 
\begin{equation}\label{der_deg_0}
\{X+\delta^1(X)+\ldots+\delta^n(X)\mid X\in\mx(M)\}\cup\{\hat{\varepsilon}\mid
\varepsilon\in\Gamma(E_j) \text{ for some } j\} 
\end{equation}
span $\der(C^\infty(\mathcal M))$ as a $C^\infty(\mathcal M)$-module.

\subsection{Metric double vector bundles} 
Next we introduce linear metrics on double vector bundles.
\begin{definition}
A \textbf{metric double vector bundle} is a double vector bundle
$(\mathbb E, Q; B, M)$ equipped with a \textbf{symmetric
  non-degenerate pairing $\mathbb E\times_B\mathbb E\to \R$} that is
also \textbf{linear over $Q$},
i.e.~such that the map $\Beta\colon \mathbb E\to\mathbb E\duer B$
{\small\begin{equation*}
  \begin{xy}
    \xymatrix{\mathbb E\ar[rrr]^{\Beta}\ar[rd]^{\pi_{Q}}\ar[dd]_{\pi_{B}}& && \mathbb E\duer B\ar[rd]^{\pi_{Q^{**}}}\ar[dd]&\\
& Q\ar[dd]\ar[rrr]^{\beta_Q}&&&C^*\ar[dd]^{q_{C^*}}\\
      B\ar[rd]_{q_{B}}\ar[rrr]^{\Id_B}&&&B\ar[rd]^{q_{B}}&\\
&M\ar[rrr]^{\Id_M}&&&M}
\end{xy}
\end{equation*}}
\noindent defined by the pairing $\langle\cdot,\cdot\rangle$ is an
\textbf{isomorphism of double vector bundles}.  In particular, the
core $C\to M$ of $\mathbb{E}$ is isomorphic to $Q^*\to M$.
 \end{definition}
 Note that, equivalently, a pairing
 $\langle\cdot\,,\cdot\rangle$ on $\mathbb E\to B$ is linear if and
 only if
\begin{equation}\label{linearity_pairing} \langle e_1+_Qe_2, f_1+_Qf_2\rangle_{\mathbb E}=\langle e_1,f_1\rangle_{\mathbb E}+\langle e_2,f_2\rangle_{\mathbb E}
\end{equation}
for $e_1,e_2,f_1,f_2\in \mathbb E$ with $\pi_B(e_i)=\pi_B(f_i)$, $i=1,2$.
In terms of sections, a bilinear pairing $\langle\cdot\,,\cdot\rangle_{\mathbb
  E}\colon \mathbb E\times_B\mathbb E\to\mathbb R$ is symmetric,
nondegenerate and linear over $Q$ if and only if 
the
 core of $\mathbb E$ is isomorphic to $Q^*$ and, via this isomorphism,
\begin{enumerate}
\item $\langle \tau_1^\dagger, \tau_2^\dagger\rangle=0$ for
  $\tau_1,\tau_2\in\Gamma(Q^*)$,
\item $\langle \chi, \tau^\dagger\rangle=q_B^*\langle q,\tau\rangle$
  for $\chi\in\Gamma_B^l(\mathbb E)$ linear over $q\in\Gamma(Q)$ and
  $\tau\in\Gamma(Q^*)$ and
\item $\langle\chi_1,\chi_2\rangle$ is a linear function on $B$ for
  $\chi_1,\chi_2\in \Gamma_B^l(\mathbb E)$.
\end{enumerate}
In the following, we always identify with $Q^*$ the core of a
metric double vector bundle $(\mathbb E, Q; B, M)$.

\subsubsection{Lagrangian decompositions of a metric double vector
  bundle}\label{Lagr_dec}

\begin{definition}
  Let $(\mathbb E, B; Q, M)$ be a metric double vector bundle.  A
  linear splitting $\Sigma\colon Q\times_MB\to \mathbb E$ is said to
  be \textbf{Lagrangian} if its image is (maximal) isotropic in $\mathbb
  E\to B$.  The corresponding horizontal lifts
  and the
  corresponding decomposition of $\mathbb E$ are then also said to be
  \textbf{Lagrangian}.
\end{definition}
Note that by definition, a horizontal lift $\sigma_Q\colon
\Gamma(Q)\to \Gamma^l_B(\mathbb E)$ is Lagrangian if and only if
$\langle \sigma_Q(q_1), \sigma_Q(q_2)\rangle=0$ for all $q_1,q_2\in
\Gamma(Q)$.

Let $\sigma_Q\colon\Gamma(Q)\to\Gamma^l_B(\mathbb E)$ be an arbitrary
horizontal lift. We have seen that by the definition of a linear
metric on $\mathbb E\to B$, the pairing of two linear sections is a
linear function on $B$. This implies with
\[ \sigma_Q(fq)=q_B^*f\cdot\sigma_Q(q) \text{ and }
\ell_{f\beta}=q_B^*f\cdot\ell_\beta \text{ for all } f\in C^\infty(M),
q\in \Gamma(Q) \text{ and } \beta\in\Gamma(B^*)
\]
the existence of a symmetric tensor $\Lambda\in \Gamma(S^2(Q,B^*))$ 
such that
  \begin{equation}\label{lambda_def}
    \langle \sigma_Q(q_1), \sigma_Q(q_2)\rangle_{\mathbb E}=\ell_{\Lambda(q_1,q_2)}.
\end{equation}
In particular, $\Lambda(q,\cdot): Q\to B^*$ is a morphism of vector
bundles for each $q\in\Gamma(Q)$.  Define a new horizontal lift
$\sigma_Q'\colon\Gamma(Q)\to \Gamma^l_B(\mathbb E)$ by
$\sigma_Q'(q)=\sigma_Q(q)-\frac{1}{2}\widetilde{\Lambda(q,\cdot)^*}$ for
all $q\in\Gamma(Q)$. 
Since for $\phi\in\Gamma(\operatorname{Hom}(B,Q^*))$, $\langle
\widetilde{\phi}, \chi\rangle=\ell_{\phi^*(q)}$ if
$\chi\in\Gamma_B^l(\mathbb E)$ is linear over $q\in\Gamma(Q)$, we find then
\[
\langle \sigma_Q'(q_1),\sigma_Q'(q_2)\rangle_{\mathbb E}
=\langle\sigma_Q(q_1),\sigma_Q(q_2)\rangle_{\mathbb E}-\frac{1}{2}\ell_{\Lambda(q_1,q_2)}-\frac{1}{2}\ell_{\Lambda(q_2,q_1)}
=0
\] for all $q_1,q_2\in\Gamma(Q)$. This proves
the following result.
\begin{proposition}\label{symmetrization}
  Let $(\mathbb E, B; Q, M)$ be a metric double vector bundle.  Then
  there exists a Lagrangian splitting of $\mathbb E$.
\end{proposition}

Next we show that a change of Lagrangian splitting corresponds to a
skew-symmetric element of $\Gamma(Q^*\otimes B^*\otimes Q^*)$.
\begin{proposition}\label{lagrangian_change_of_split}
  Let $(\mathbb E, B; Q, M)$ be a metric double vector bundle and
  choose a Lagrangian horizontal lift
  $\sigma_Q^1\colon\Gamma(Q)\to\Gamma_B^l(\mathbb E)$.  Then a second
  horizontal lift $\sigma_Q^2\colon\Gamma(Q)\to\Gamma_B^l(\mathbb E)$
  is Lagrangian if and only if the change of lift
  $\phi_{12}\in\Gamma(Q^*\otimes B^*\otimes Q^*)$ satisfies the
  following equality:
\[
\langle \phi_{12}(q),q'\rangle=-\langle\phi_{12}(q'),
q\rangle\in\Gamma(B^*)\] for all $q,q'\in\Gamma(Q)$, i.e.~if and only
if $\phi_{12}\in\Gamma(Q^*\wedge Q^*\otimes B^*)$.
\end{proposition}

\begin{proof}
  For $q\in\Gamma(Q)$ we have $\langle\widetilde{\phi_{12}(q)},
  \chi\rangle=\ell_{\langle\phi_{12}(q), q'\rangle}$ for any linear
  section $\chi\in\Gamma_B^l(\mathbb E)$ over $q'\in\Gamma(Q)$.  Hence
  we find
\begin{equation*}
\begin{split}
  &\langle \sigma_Q^1(q), \sigma_Q^1(q')\rangle_{\mathbb E}-\langle \sigma_Q^2(q),
  \sigma_Q^2(q')\rangle_{\mathbb E} \\
&=\langle \sigma_Q^1(q)-\sigma_Q^2(q),
  \sigma_Q^1(q')\rangle_{\mathbb E}
  +\langle \sigma_Q^2(q), \sigma_Q^1(q')-\sigma_Q^2(q')\rangle_{\mathbb E}=\ell_{\langle \phi_{12}(q),q'\rangle}+\ell_{\langle q,
    \phi_{12}(q')\rangle}.
\end{split}
\end{equation*}

\end{proof}

The last proposition shows that not any linear section of $\mathbb E$
over $B$ can be obtained as the Lagrangian horizontal lift of a
section of $Q$. This is easy to understand in Example
\ref{metric_connections}. 

Let $(\mathbb E, B; Q, M)$ be a metric double vector bundle. 
Choose a Lagrangian splitting $\Sigma\colon Q\times_MB\to \mathbb E$
and set
\begin{equation}\label{def_of_CE}
\mathcal C(\mathbb
E):=\sigma_B(\Gamma(B))+
\{\tilde\omega\mid \omega\in\Omega^2(Q)\}.
\end{equation}
Note that $\mathcal C(\mathbb E)$ together with $\Gamma^c_Q(\mathbb
E)\simeq\Gamma(Q^*)$ span $\mathbb E$ as a vector bundle over
$Q$. Note also that $\mathcal C(\mathbb E)$ is a sheaf of
$C^\infty(M)$-modules: for $\chi\in\mathcal C(\mathbb E)$ and $f\in
C^\infty(M)$, the product $q_Q^*f\cdot \chi$ is again an element of
$\mathcal C(\mathbb E)$. In particular, for a Lagrangian splitting
$\Sigma\colon Q\times_MB\to \mathbb E$,
$q_Q^*f\cdot(\sigma_B(b)+\tilde\omega)=\sigma_B(fb)+\widetilde{f\omega}$
for all $b\in \Gamma(B)$ and all $\omega\in\Omega^2(Q)$.

We begin by giving an intrinsic geometric description of $\mathcal
C(\mathbb E)$.
\begin{proposition}\label{prop_CE1}
  Let $(\mathbb E, B; Q, M)$ be a metric double vector bundle.  The
  space $\mathcal C(\mathbb E)\subseteq \Gamma^l_Q(\mathbb E)$ is the
  set of linear sections $Q\to \mathbb E$ with isotropic image
  relative to $\langle\cdot\,,\cdot\rangle$.
\end{proposition}

\begin{proof}
  We have already seen that for any section $\chi\in\mathcal
  C(\mathbb E)$, the pairing $\langle\chi(q),\chi(q')\rangle$
  vanishes for all $(q,q')\in Q\times_MQ$. Conversely, consider a
  linear section $\chi\in\Gamma_Q^l(\mathbb E)$ with isotropic
  image. Let $b\in\Gamma(B)$ be the basis section of $\chi$ and choose a
  Lagrangian splitting $\Sigma\colon Q\times_M B\to\mathbb E$. Then
  $\chi=\sigma_B(b)+\widetilde{\phi}$ with
  $\phi\in\Gamma(Q^*\otimes Q^*)$ and we get for all $q,q'\in Q$ with
  $q_Q(q)=q_Q(q')=m\in M$:
\begin{equation*}
\begin{split}
  0&\,\,=\langle\chi(q),\chi(q') \rangle=\left\langle
    \sigma_B(b)(q)+_Q(0^{\mathbb
      E}_{q}+_B\overline{\phi(q)}),\sigma_B(b)(q')+_Q(0^{\mathbb
      E}_{q'}+_B\overline{\phi(q')})\right\rangle\\
&\overset{\eqref{add_add}}{=}\left\langle
    \sigma_B(b)(q)+_B(0^{\mathbb
      E}_{b(m)}+_Q\overline{\phi(q)}),\sigma_B(b)(q')+_B(0^{\mathbb
      E}_{b(m)}+_Q\overline{\phi(q')})\right\rangle\overset{\eqref{linearity_pairing}}{=}\phi(q)(q')+\phi(q')(q).
\end{split}
\end{equation*}
Therefore, $\phi\in\Omega^2(Q)$.
\end{proof}

\begin{proposition}\label{prop_CE2}
  Let $(\mathbb E, B; Q, M)$ be a metric double vector bundle.  The
  space $\mathcal C(\mathbb E)\subseteq \Gamma^l_Q(\mathbb E)$ is a
  locally free and finitely generated sheaf of $C^\infty(M)$-modules,
  that fits in the following short exact sequence of sheaves of
  $C^\infty(M)$-modules:
\begin{equation}\label{ses_isotropic}
0 \longrightarrow \Omega^2(Q) \hookrightarrow \mathcal C(\mathbb E)
\longrightarrow \Gamma(B) \longrightarrow 0.
\end{equation}
The maps are the restrictions of the maps in
\eqref{fat_seq_gamma}. Lagrangian splittings of $\mathbb E$ are
equivalent to splittings of \eqref{ses_isotropic}.
\end{proposition}

\begin{proof}
  Choose a Lagrangian splitting $\Sigma\colon B\times_MQ\to\mathbb E$
  of $\mathbb E$. Let $r_1$ be
the rank of $Q$ and $r_2$ the rank of $B$.
Choose $p\in M$. Then there exists an open neighborhood
  $U$ of $p$ in $M$ that trivialises both $Q$ and $B$. Choose local
  basis sections $b_1,\ldots, b_{r_2}\in\Gamma_U(B)$ of $B$ over $U$
  and local basis sections $\tau_1,\ldots,\tau_{r_1}\in\Gamma_U(Q^*)$
  of $Q^*$ over $U$. Then by \eqref{def_of_CE} and the considerations
  below it, $\mathcal C_U(\mathbb E)$ is freely generated as a
  $C^\infty_U(M)$-module by
  $\{\sigma_B(b_1),\ldots,\sigma_B(b_{r_2})\}\cup\{\widetilde{\tau_{i}\wedge\tau_j}\mid
  1\leq i<j\leq r_1\}$.

  It is easy to check as in the proof of Proposition \ref{prop_CE1}
  that isotropic core-linear sections of $\mathbb E\to Q$ are exactly
  the sections $\widetilde \phi$ for $\phi\in\Omega^2(Q)$. Since the
  inclusion $\Gamma(Q^*\otimes Q^*)\hookrightarrow \Gamma_Q^l(\mathbb
  E)$ is injective (see \eqref{fat_seq_gamma}), its restriction to
  $\Omega^2(Q)\hookrightarrow \mathcal C(\mathbb E)$ is also
  injective. The rest follows from the construction of $\mathcal
  C(\mathbb E)$ in \eqref{def_of_CE}, or more precisely from the
  existence of Lagrangian splittings.
\end{proof}

\begin{corollary}\label{cor_fat_CE}
  There exists a vector bundle $\widehat{B}$ over $M$ which set of
  sections is isomorphic to $\mathcal C(\mathbb E)$.  The short exact
  sequence \eqref{ses_isotropic} induces a short exact sequence of
  vector bundles over $M$:
\begin{equation*}
0 \longrightarrow Q^*\wedge Q^* \hookrightarrow \widehat{B}
\longrightarrow B \longrightarrow 0.
\end{equation*}
\end{corollary}

\bigskip

We end this section with a characterisation of Lagrangian splittings
that will be useful in \S\ref{sec:Poisson}.  Recall from Section \ref{dual} that
given a linear splitting $\Sigma\colon Q\times_M B\to \mathbb E$, one
can construct the dual linear splitting
$\Sigma^\star\colon Q^{**}\times_M B\to \mathbb E\duer B$.
\begin{lemma}\label{lem_Lagr_split_sections}
  Let $(\mathbb E; Q, B; M)$ be a metric double vector bundle and
  choose a linear splitting $\Sigma$ of $\mathbb E$.  Then $\Sigma$ is
  Lagrangian if and only if the linear map $\Beta\colon \mathbb E\to
  \mathbb E\duer B$ sends $\sigma_B(b)$ to $\sigma^\star_{B}(b)$ for
  all $b\in\Gamma(B)$.
\end{lemma}

\begin{proof}
  Recall from \eqref{lemma_dual_splitting} that given a horizontal lift
  $\sigma_B\colon\Gamma(B)\to\Gamma_Q^l(\mathbb E)$, the dual
  horizontal lift
  $\sigma^\star_B\colon\Gamma(B)\to\Gamma_{Q^{**}}^l(\mathbb E\duer
  B)$ can be defined by
\[\langle \sigma_B^\star(b)(p_m),\sigma_B(b)(q_m)\rangle_B=0, \qquad
\langle \sigma_B^\star(b)(p_m),\tau^\dagger(b(m))\rangle_B=\langle p_m,
\tau(m)\rangle
\]
for all $b\in\Gamma(B)$, $\tau\in\Gamma(Q^*)$, $q_m\in Q$ and $p_m\in
Q^{**}\simeq Q$.

On the other hand, if $\Sigma\colon B\times_MQ\to\mathbb E$ is a
Lagrangian splitting, we have \begin{equation*}
\begin{split}
\langle \Beta(\sigma_B(b)(p(m))),
\sigma_B(b)(q(m))\rangle_B&= \langle \sigma_B(b)(p(m)),
\sigma_B(b)(q(m))\rangle_{\mathbb E}\\
&=\langle
\sigma_Q(p),\sigma_Q(q)\rangle_{\mathbb E}(b(m))=0
\end{split}
\end{equation*}
 for all $q,p\in\Gamma(Q)$ and
$b\in\Gamma(B)$, and
\begin{equation*}
\begin{split}
  \langle \Beta(\sigma_B(b)(p(m))), \tau^\dagger(b(m))\rangle_B&=
  \langle \sigma_B(b)(p(m)), \tau^\dagger(b(m))\rangle_{\mathbb E}\\
  &=\langle \sigma_Q(p)(b(m)), \tau^\dagger(b(m))\rangle_{\mathbb
    E}=\langle p,\tau\rangle(m)
\end{split}
\end{equation*}
for all $\tau\in\Gamma(Q^*)$. This proves that $\Beta$ sends the
linear section $\sigma_B(b)\in\Gamma_Q^l(\mathbb E)$ to
$\sigma^\star_B(b)$ in $\Gamma^l_{Q^{**}}(\mathbb E\duer B)$.  It is
easy to see from the four equalities above that this condition is
necessary for $\Sigma$ to be Lagrangian.
\end{proof}

\subsubsection{Examples of metric double vector bundles}
Next we describe a couple of examples of metric double vector bundles.
\begin{example}\label{metric_connections}
  Let $E\to M$ be a metric vector bundle, i.e.~a vector bundle endowed
  with a symmetric non-degenerate pairing
  $\langle\cdot\,,\cdot\rangle\colon E\times_M E\to \R$.  Then
  $E\simeq E^*$ and the tangent double is a metric double vector
  bundle $(TE,E;TM,M)$ with pairing $TE\times_{TM}TE\to \R$ the
  tangent of the pairing $E\times_M E\to \R$. In particular, we have
\[\langle Te_1, Te_2\rangle_{TE}=\ell_{\dr\langle e_1,e_2\rangle}, 
\quad \langle Te_1, e_2^\dagger\rangle_{TE}=p_M^* \langle e_1,e_2\rangle
\quad \text{ and }\langle e_1^\dagger, e_2^\dagger\rangle_{TE}=0\] for
$e_1,e_2\in\Gamma(E)$.

Recall from \S\ref{tangent_double} that linear splittings of $TE$ are
equivalent to linear connections $\nabla\colon
\mx(M)\times\Gamma(E)\to\Gamma(E)$.  We have then for all $e_1,e_2\in\Gamma( E)$:
  \[\left\langle \sigma_{ E}^\nabla(e_1), e_2^\dagger\right\rangle=\left\langle
    Te_1-\widetilde{\nabla_\cdot e_1},
    e_2^\dagger\right\rangle=p_M^*\langle e_1, e_2\rangle\] and
  \[\left\langle \sigma_{ E}^\nabla (e_1), \sigma_{ E}^\nabla (e_2)\right\rangle=\left\langle
    Te_1-\widetilde{\nabla_\cdot e_1}, Te_2-\widetilde{\nabla_\cdot
    e_2}\right\rangle=\ell_{\dr\langle e_1, e_2\rangle-\langle
    e_2, \nabla_\cdot e_1\rangle-\langle e_1, \nabla_\cdot e_2\rangle}.\] 
The Lagrangian splittings of $TE$
are hence exactly the linear splittings that correspond to \textbf{metric}
connections, i.e.~linear connections $\nabla\colon
\mx(M)\times\Gamma(E)\to\Gamma(E)$ that preserve the metric:
$\langle\nabla_\cdot e_1, e_2\rangle+\langle e_1,\nabla_\cdot
e_2\rangle=\dr\langle e_1, e_2\rangle$ for $e_1,e_2\in\Gamma(E)$.

More generally, the isotropic linear vector fields on $E$ are the
linear vector fields corresponding to derivations of $E$ that preserve the
pairing.
\end{example}

\begin{example}\label{met_TET*E}
Let $q_E\colon E\to M$ be a vector bundle and consider the double vector bundle
\begin{equation*}
\begin{xy}
  \xymatrix{
    TE\oplus T^*E\ar[rr]^{\Phi_E:=({q_E}_*, r_E)}\ar[d]_{\pi_E}&& TM\oplus E^*\ar[d]\\
    E\ar[rr]_{q_E}&&M }
\end{xy}
\end{equation*} with sides $E$ and $TM\oplus E^*\to M$, and 
with core $E\oplus T^*M\to M$.  The projection
$r_E\colon T^*E\to E^*$ is defined by
\[r_E(\theta_{e_m})\in E^*_m, \qquad \langle r_E(\theta_{e_m}),
e_m'\rangle =\left\langle \theta_{e_m},
  \left.\frac{d}{dt}\right\an{t=0}e_m+te_m'\right\rangle,
 \]
 and is a fibration of vector bundles over the projection $q_E\colon
 E\to M$.  The core elements are identified in the following manner with elements of $E\oplus
 T^*M\to M$. For $m\in M$ and
 $(e_m,\theta_m)\in E_m\times T_m^*M$, the pair
\[\left(\left.\frac{d}{dt}\right\an{t=0}te_m, (T_{0_m^E}q_E)^*\theta_m
\right)\] projects to $(0^{TM}_m,0^{E^*}_m)$ under $\Phi_E$ and to
$0^E_m$ under $\pi_E$. Conversely, any element of $TE\oplus T^*E$ in
the double kernel can be written in this manner.  Next recall that
$TE\oplus T^*E\to E$ has a natural symmetric nondegenerate pairing
given by
\begin{equation}\label{standard_pairing}
\langle (v^1_{e_m},\theta^1_{e_m}),  (v^2_{e_m},\theta^2_{e_m})\rangle=\theta^1_{e_m}(v^2_{e_m})+\theta^2_{e_m}(v^1_{e_m}),
\end{equation}
the natural pairing underlying the standard Courant algebroid
structure on $TE\oplus T^*E\to E$.

Dull algebroids and Dorfman connections were introduced in
\cite{Jotz13a}. 
A \textbf{Dorfman $TM\oplus E^*$-connection on its dual $E\oplus T^*M$} is an $\R$-bilinear map
\[\Delta\colon \Gamma(TM\oplus E^*)\times\Gamma(E\oplus T^*M)\to\Gamma(E\oplus T^*M)\]
satisfying \begin{enumerate}
\item $\Delta_q(f\cdot\tau)=f\cdot\Delta_q\tau+\ldr{\pr_{TM}(q)}(f)\cdot \tau$,
\item $\Delta_{f\cdot q}\tau=f\cdot \Delta_q\tau+\langle q,\tau\rangle\cdot(0,\dr f)$, and 
\item $\Delta_q(0,\dr f)=(0,\dr(\ldr{\pr_{TM}q}f))$
\end{enumerate}
for all $q\in\Gamma(TM\oplus E^*)$, $\tau\in\Gamma(E\oplus T^*M)$ and
$f\in C^\infty(M)$. The first axiom says that $\Delta$ defines a map
$\Delta\colon q\mapsto \Delta_q\in\operatorname{Der}(E\oplus
T^*M)$. The dual of this map in the sense of derivations defines a
\textbf{dull bracket on sections of $TM\oplus E^*$}, i.e.~an
$\R$-bilinear map \[\llb\cdot\,,\cdot\rrb_\Delta\colon \Gamma(TM\oplus
E^*)\times \Gamma(TM\oplus E^*)\to \Gamma(TM\oplus E^*)\] satisfying 
\begin{enumerate}
\item $\pr_{TM}\llb
  q_1,q_2\rrb_\Delta=[\pr_{TM}q_1,\pr_{TM}q_2]$,
\item $\llb f_1q_1, f_2q_2\rrb=f_1f_2\llb q_1,q_2\rrb_\Delta+f_1\ldr{\pr_{TM}q_1}(f_2)q_2-f_2\ldr{\pr_{TM}q_2}(f_1)q_1$
\end{enumerate}
for all $q_1,q_2\in\Gamma(TM\oplus E^*)$ and $f_1,f_2\in
C^\infty(M)$.

We prove in \cite{Jotz13a} that linear splittings of $TE\oplus T^*E$
are in bijection with dull brackets on sections of $TM\oplus E^*$, or
equivalently with Dorfman connections $\Delta\colon \Gamma(TM\oplus
E^*)\times\Gamma(E\oplus T^*M)\to\Gamma(E\oplus T^*M)$. Choose such a
Dorfman connection.  For any pair $(X,\epsilon)\in\Gamma(TM\oplus
E^*)$, the horizontal lift $\sigma:=\sigma_{TM\oplus
  E^*}^\Delta\colon\Gamma(TM\oplus E^*)\to\Gamma_E(TE\oplus
T^*E)=\mx(E)\times\Omega^1(E)$ is given by
\[\sigma(X,\epsilon)(e_m)=\left(T_me X(m), \dr
    \ell_\epsilon(e_m)\right)-\Delta_{(X,\epsilon)}(e,0)^\dagger(e_m)
\]
for all $e_m\in E$, where for $(e,\theta)\in\Gamma(e\oplus T^*M)$, the
pair $(e,\theta)^\dagger\in\Gamma_E(TE\oplus T^*E)$ is defined by
$(e,\theta)^\dagger=(e^\uparrow, q_E^*\theta)$.

Since the vector bundle $TM\oplus E^*$ is anchored by the morphism
$\pr_{TM}\colon TM\oplus E^*\to TM$, the $TM$-part
of $\llb q_1, q_2\rrb_\Delta +\llb q_2, q_1\rrb_\Delta$ is trivial and
this sum can be seen as an element of $\Gamma(E^*)$.  We proved the
following result in \cite{Jotz13a}.
\begin{theorem}\label{Lagrangin1}
  Choose $q, q_1,q_2\in\Gamma(TM\oplus E^*)$ and $\tau,\tau_1,
  \tau_2\in\Gamma(E\oplus T^*M)$.  The natural pairing on fibres of
  $TE\oplus T^*E\to E$ is given by
 \begin{enumerate}
\item $\left\langle \sigma(q_1), \sigma(q_2)\right\rangle=\ell_{\llb q_1, q_2\rrb_\Delta +\llb q_2, q_1\rrb_\Delta}$,
\item $\left\langle \sigma(q),
    \tau^\dagger\right\rangle=q_E^*\langle q, \tau\rangle$.
\item $\left\langle \tau_1^\dagger,
    \tau_2^\dagger\right\rangle=0$.
\end{enumerate}
\end{theorem}
As a consequence, the natural pairing on fibres of $TE\oplus T^*E\to
E$ is a linear metric on $(TE\oplus T^*E;TM\oplus E^*,E;M)$ and the
Lagrangian splittings are equivalent to skew-symmetric dull brackets
on sections of the anchored vector bundle $(TM\oplus E^*,\pr_{TM})$.

\end{example}

\subsection{Involutive double vector bundles}
We begin here the study of involutive double vector bundles, which we
find to be in duality with metric double vector bundles.  Note that in
the original work (see \cite{Pradines77}) where double vector bundles
were introduced, Pradines already introduced involutive double vector
bundles, called \emph{fibr\'es vectoriels doubles \`a sym\'etrie
  inverse}\footnote{\emph{Double vector bundles with inverse symmetry}
  in \cite{Pradines77}. A symmetry of a double vector bundle is an
  involution as in Definition \ref{def_idvb}, but without the
  condition on the core morphism. It is \emph{direct} if the induced
  morphism on the core is the identity, and \emph{inverse} if it is
  minus the identity.}, for his study of nonholonomic jets.

\begin{definition}\label{def_idvb}
An \textbf{involutive double vector bundle} is a double vector bundle 
$(D,Q,Q,M)$ with core $B^*$ equipped with a morphism
{\small \begin{equation*}
  \begin{xy}
    \xymatrix{D\ar[rrr]^{\mathcal I}\ar[rd]^{\pi_1}\ar[dd]_{\pi_2}& && D\ar[rd]^{\pi_2}\ar[dd]&\\
& Q\ar[dd]\ar[rrr]^{\Id_Q}&&&Q\ar[dd]^{q_{Q}}\\
      Q\ar[rd]_{q_{Q}}\ar[rrr]^{\Id_Q}&&&Q\ar[rd]^{q_{Q}}&\\
&M\ar[rrr]^{\Id_M}&&&M}
\end{xy}
\end{equation*}}
satisfying $\mathcal I^2=\Id_D$ and $\pi_1\circ\mathcal I=\pi_2$, $\pi_2\circ\mathcal I=\pi_1$
and with core morphism $-\Id_{B^*}\colon B^*\to B^*$.
 
\end{definition}

We begin by proving that metric double vector bundles are dual to
involutive double vector bundles.
\begin{proposition}\label{dual_inv_metric}
\begin{enumerate}
\item  Let $(D,Q,Q,M)$ be an involutive double vector bundle with
  involution $\mathcal I$ and core $B^*$. Then the dual
  $(\mathbb E:=D\duer {\pi_1},Q,B,M)$ inherits a linear metric
  $\mathbb E\times_B\mathbb E\to\R$ defined by
\[ \langle e_1,e_2\rangle_{\mathbb E}=\langle e_1, d\rangle_Q+\langle e_2,\mathcal
I(d)\rangle_Q
\]
for $(e_1,e_2)\in\mathbb E\times_B\mathbb E$ and any $d\in D$ with
$\pi_Q(e_1)=\pi_1(d)$ and $\pi_Q(e_2)=\pi_2(d)$.
\item Conversely, consider a metric double vector bundle $(\mathbb
  E,Q,B,M)$ with core $Q^*$. Then the dual $D=\mathbb E\duer Q$ with
  sides $\pi_1\colon \mathbb E\duer Q\to Q$, $\pi_2\colon \mathbb
  E\duer Q\to Q^{**}\simeq Q$ and with core $B^*$ is an involutive
  double vector bundle with $\mathcal I\colon D\to D$ defined by
\[ \langle \mathcal I(d), e\rangle_Q=\nsp{e}{d}
\]
for $d\in D$ and $e\in\mathbb E\simeq \mathbb E\duer B$ with
$\pi_2(d)=\pi_Q(e)$. (Recall that the pairing $\nsp{\,}{}$ was defined
in \S\ref{fat_pairing_def}.)
\item The constructions in (1) and (2) are inverse to each other.
\end{enumerate}
\end{proposition}

\begin{proof}
 (1) We begin by proving that the pairing is well-defined. Choose
  $e_1,e_2\in\mathbb E\times_B\mathbb E$ and any $d\in D$ with
  $\pi_Q(e_1)=\pi_1(d)=q_1\in Q_m$ and $\pi_Q(e_2)=\pi_2(d)=q_2\in
  Q_m$. Then, for any $\beta\in B^*_m$, we have by\eqref{add_add}
  $d':=d+_{1}(0^D_{q_1}+_{2}\overline{\beta})=d+_{2}(0^D_{q_2}+_{1}\overline{\beta})$
  and thefore $\pi_Q(e_1)=q_1=\pi_1(d')\in Q_m$ and
  $\pi_Q(e_2)=q_2=\pi_2(d')\in Q_m$. Conversely, any $d'\in D$ with
  $\pi_Q(e_1)=\pi_1(d')\in Q_m$ and $\pi_Q(e_2)=\pi_2(d')\in Q_m$ can
  be obtained in this manner. We compute 
\begin{equation*}
\begin{split}
\langle e_1, d'\rangle_Q+\langle e_2,\mathcal
I(d')\rangle_Q
&=\langle e_1, d+_{1}(0^1_{q_1}+_{2}\overline{\beta})\rangle_Q+\langle e_2,\mathcal
I(d+_{1}(0^1_{q_1}+_{2}\overline{\beta}))\rangle_Q\\
&=\langle e_1, d+_{1}(0^1_{q_1}+_{2}\overline{\beta})\rangle_Q+\langle e_2,\mathcal
I(d)+_{2}(0^2_{q_1}+_{1}\overline{-\beta}))\rangle_Q\\
&=\langle e_1, d+_{1}(0^1_{q_1}+_{2}\overline{\beta})\rangle_Q+\langle e_2,\mathcal
I(d)+_{1}(0^1_{q_2}+_{2}\overline{-\beta}))\rangle_Q\\
&=\langle e_1, d\rangle_Q+\langle e_2,\mathcal
I(d)\rangle_Q+\langle \pi_B(e_1),\beta\rangle-\langle\pi_B(e_2),\beta\rangle\\
&=\langle e_1, d\rangle_Q+\langle e_2,\mathcal
I(d)\rangle_Q.
\end{split}
\end{equation*}

To check the symmetry of the pairing, recall that if $d$ is as above, then 
$\mathcal I(d)$ satisfies $\pi_1(\mathcal I(d))=q_2=\pi_Q(e_2)$ and 
$\pi_2(\mathcal I(d))=q_1=\pi_Q(e_1)$. Hence, we have by definition:
\[\langle e_2,e_1\rangle_{\mathbb E}=\langle e_2, \mathcal I(d)\rangle_Q+\langle e_1,\mathcal
I^2(d)\rangle_Q
=\langle e_2, \mathcal I(d)\rangle_Q+\langle e_1,d\rangle_Q
=\langle e_1, e_2\rangle_{\mathbb E}.
\]

Finally, consider $e\in \mathbb E$ with $\langle e,e'\rangle_{\mathbb
  E}=0$ for all $e'\in \mathbb E$ with $\pi_B(e)=\pi_B(e')=b_m$.
In particular, we find for all $\tau\in Q^*_m$: 
\[0=\langle e, 0^{\mathbb E}_{b_m}+_Q\overline{\tau}\rangle=\langle e, 0^1_{q}\rangle_Q+\langle 0^{\mathbb E}_{b_m}+_Q\overline{\tau}, \mathcal I(0^1_{q})\rangle_Q
=\langle e, 0^1_{q}\rangle_Q+\langle \tau^\dagger(b_m), 0^2_{q}\rangle_Q=\langle\tau, q\rangle.
\]
This shows that $q=\pi_1(e)\in Q_m$ must vanish, and so that
$e=0^{\mathbb E}_{b_m}+_Q\overline{\eta}$ for some $\eta\in Q^*_m$. In
the same manner as above, we find then that $\langle\eta, q'\rangle=0$
for all $q'\in Q_m$, and so that $\eta=0$. This shows that
$e=0^{\mathbb E}_{b_m}$.
The linearity of $\langle\cdot\,,\cdot\rangle_{\mathbb E}$ is
immediate and its proof is left to the reader.
\medskip

(2) A straightforward computation shows that both 
\[  \begin{xy}
\xymatrix{
D\ar[r]^{\mathcal I}\ar[d]_{\pi_1} &D\ar[d]^{\pi_2}\\
Q\ar[r]_{\Id_Q}&Q
}
\end{xy}\qquad \text{ and }\qquad \begin{xy}
\xymatrix{
D\ar[r]^{\mathcal I}\ar[d]_{\pi_2} &D\ar[d]^{\pi_1}\\
Q\ar[r]_{\Id_Q}&Q
}
\end{xy}
\] are morphisms of vector bundles. To find the core morphism of
$\mathcal I$, consider $\beta\in B^*_m$ and $0^{\mathbb E}_{b}+_Q\overline{\tau}$ for some 
$b\in B_m$ and $\tau\in Q^*_m$. 
Then 
\begin{equation*}
\begin{split}
  \langle \mathcal I(\overline{\beta}), 0^{\mathbb
    E}_{b}+_Q\overline{\tau}\rangle_Q&=\nsp{0^{\mathbb
      E}_{b}+_Q\overline{\tau}}{\overline{\beta}}=\langle
  0^{\mathbb E}_{b}+_Q\overline{\tau}, 0^{\mathbb
    E}_b\rangle_{\mathbb E}-\langle \overline{\beta},0^{\mathbb
    E}_b\rangle_Q\\
  &=-\langle\beta,b\rangle=\left\langle\overline{-\beta},  0^{\mathbb
    E}_{b}+_Q\overline{\tau}\right\rangle_Q.
\end{split}
\end{equation*}
Since any element of $\pi_Q^{-1}(0^Q_m)\subseteq \mathbb E$ can be
written $0^{\mathbb E}_{b}+_Q\overline{\tau}$ for some $b\in B_m$ and
$\tau\in Q^*_m$, we have proved that $\mathcal
I(\overline{\beta})=\overline{-\beta}$ for all $\beta\in B^*$. We prove that $\mathcal I^2=\Id_D$. Choose $d\in
D$ and $e\in \mathbb E$ with $\pi_Q(e)=\pi_1(d)=\pi_2(\mathcal I(d))$.
Then, by definition, with $e'\in\mathbb E$ such that
$\pi_B(e')=\pi_B(e)$ and $\pi_Q(e')=\pi_2(d)=\pi_1(\mathcal I(d))$:
\begin{equation*}
\begin{split}
  \langle \mathcal I^2(d), e\rangle_Q&=\nsp{e}{\mathcal I(d)}=\langle e,e'\rangle_{\mathbb E}-\langle \mathcal I(d), e'\rangle _Q\\
  &=\langle e,e'\rangle_{\mathbb E}-\nsp{e'}{d}=\langle e,e'\rangle_{\mathbb E}-\langle e',e\rangle_{\mathbb E}+\langle          d, e\rangle _Q\\
\end{split}
\end{equation*}

\medskip

(3) We start from an involutive double vector bundle $(D,Q,Q,M)$ with core
$B^*$ and involution $\mathcal I$. We build the dual
metric double vector bundle $(\mathbb E=D\duer{\pi_1},Q,Q,M)$ as in (i), and construct the dual
involutive double vector bundle as in (ii). Let $\mathcal I'$ be the
new involution obtained in this manner on $D$.
By definition, we have for $d\in D$ and $e\in\mathbb E$ with $\pi_Q(e)=\pi_2(d)$:
\[ \langle \mathcal I'(d), e\rangle_Q=\nsp{e}{d}=\langle
e,e'\rangle_{\mathbb E}-\langle d, e'\rangle_Q=\langle 
\mathcal I(d), e\rangle_Q+\langle \mathcal I^2(d), e'\rangle_Q -\langle d, e'\rangle_Q=\langle \mathcal I(d), e\rangle_Q
\]
for any $e'\in\mathbb E$ with $\pi_B(e')=\pi_B(e)$ and
$\pi_Q(e')=\pi_1(d)$. This shows $\mathcal I=\mathcal I'$.

Conversely, if we start with a metric double vector bundle $(\mathbb
E, Q,B,M)$ and take the dual involutive double vector bundle
$(D=\mathbb E\duer Q, Q,Q,M)$ with core $B^*$ and involution $\mathcal
I$, the involution defines a new metric $\langle\cdot\,,\cdot\rangle'$
on $\mathbb E$.
We have for all $(e_1,e_2)\in \mathbb E\times_B\mathbb E$:
\[\langle e_1,e_2\rangle'=\langle d,e_1\rangle_Q+\langle \mathcal
I(d),e_2\rangle_Q= \langle d,e_1\rangle_Q+\nsp{e_2}{d}
=\langle e_1, e_2\rangle_{\mathbb E}
\]
with $d$ any element of $D$ satisfying $\pi_1(d)=\pi_Q(e_1)$ and
$\pi_2(d)=\pi_Q(e_2)$. 
\end{proof}

Let $(D,Q,Q,M)$ be an involutive double vector bundle with involution
$\mathcal I$ and core $B^*$. 
Take a Lagrangian linear splitting $\Sigma\colon Q\times_MB\to \mathbb
E$ of the metric double vector bundle $\mathbb E=D\duer{\pi_1}$ and
the dual splitting $\Sigma^\star\colon Q\times _MQ\to D$ of
$D$. Consider $(q_1,q_2)\in Q\times_MQ$. Each element 
$e\in\mathbb E$ with $\pi_Q(e)=q_2$ and $\pi_B(e)=:b$
can be written $e=\Sigma(q_2,b)+_Q(0^{\mathbb E}_{q_2}+_B\overline{\tau})$ with 
a core element $\tau\in Q^*$ with $q_{Q^*}(\tau)=q_Q(q_2)=q_B(b)$:
\begin{equation*}
\begin{split}
 \langle \mathcal I(\Sigma^\star(q_1,q_2)),e\rangle_Q&=\nsp{e}{\Sigma^\star(q_1,q_2)}=\nsp{\Sigma(q_2,b)+_Q(0^{\mathbb E}_{q_2}+_B\overline{\tau})}{\Sigma^\star(q_1,q_2)}\\
&=\langle\Sigma(q_2,b)+_Q(0^{\mathbb E}_{q_2}+_B\overline{\tau}),
\Sigma(q_1,b)\rangle_{\mathbb E}-\langle\Sigma^\star(q_1,q_2), \Sigma(q_1,b)\rangle_Q\\
&=0+\langle \tau,q_1\rangle-0=\langle \tau,q_1\rangle.
\end{split}
\end{equation*}
Since also $\langle \Sigma^\star(q_2,q_1),e\rangle_Q=\langle \Sigma^\star(q_2,q_1),\Sigma(q_2,b)+_Q(0^{\mathbb E}_{q_2}+_B\overline{\tau})\rangle_Q=\langle \tau,q_1\rangle$,
we find that $\mathcal
I(\Sigma^\star(q_1,q_2))=\Sigma^\star(q_2,q_1)$.  We call such a
splitting an \textbf{involutive splitting} of $D$. 

Note that the existence of involutive splittings can also be proved
directly. Take an arbitrary splitting $\Sigma\colon Q\times_M Q\to D$
of $D$ and set $\Sigma'\colon Q\times_MQ\to
D$, \[\Sigma'(q_1,q_2):=\frac{1}{2}\cdot_1(\Sigma(q_1,q_2)+_1\mathcal
I(\Sigma(q_2,q_1)))=\frac{1}{2}\cdot_2(\Sigma(q_1,q_2)+_2\mathcal
I(\Sigma(q_2,q_1))).\] It is easy to check that $\Sigma'$ is a linear
splitting of $D$. The involutivity of $\Sigma'$ is immediate.

We leave to the reader the proof of the following lemma.
\begin{lemma}\label{inv_cois_sections}
  Let $(D;Q,Q;M)$ be an involutive double vector bundle and consider
  the dual metric double vector bundle
  $(\mathbb E={D}\duer{\pi_1}, B,Q,M)$. Then a section
  $\chi\in\Gamma_Q^l(\mathbb E)$ lies in $\mathcal C(\mathbb E)$ if
  and only if $\mathcal I^*(\ell_\chi)=-\ell_\chi$.

  Further, given $\tau\in\Gamma(Q^*)$, the morphism $\mathcal
  I^*\colon C^\infty(M)\to C^\infty(M)$ sends
  $\ell_{\tau^\dagger}=\pi_2^*\ell\tau$ to $\pi_1^*\ell_\tau$ and
  consequently $\pi_1^*\ell_\tau$ to $\ell_{\tau^\dagger}$. For $f\in
  C^\infty(M)$, $\mathcal I^*(\pi_1^*q_Q^*f)=\pi_1^*q_Q^*f$.
\end{lemma}



We end this section with a few examples. 
\begin{example}\label{inv_dvb_metric_connections}
  Consider the metric double vector bundle $TE$ in Example
  \ref{metric_connections}.  The dual over $E$ is
  $(T^*E,E,E^*,M)$ with core $T^*M$. Since $E\simeq E^*$ via the
  metric, we find $(T^*E,E,E,M)$ with the involution $\mathcal I$
  sending $\dr_{e_1(m)}\ell_{e_2}+(T_{e_1(m)}q_E)^*(\theta_m)$ to
  $\dr_{e_2(m)}\ell_{e_1}-(T_{e_2(m)}q_E)^*(\theta_m+\dr\langle
  e_1,e_2)$ for $m\in M$, $e_1,e_2\in\Gamma(E)$ and $\theta_m\in
  T^*_mM$. Up to the identification of $E$ with $E^*$, the isomorphism
  $\mathcal I$ is the \emph{reversal isomorphism $T^*E\simeq T^*E^*$} in
  \cite{MaXu94}.
\end{example}

\begin{example}
  Consider the metric double vector bundle $TE\oplus_E T^*E$ in
  Example \ref{met_TET*E}.  It is a sub-vector bundle of $(TE\times_M T^*E)\to (TM\oplus
  E^*)$, which is the pullback under
  $\iota\colon TM\oplus E^*\to TM\times E^*$ of the vector bundle $TE\times
  T^*E\to TM\times E^*$.

 The dual $(TE)\duer{TM}$ is isomorphic to $TE^*$. The
  dual $(T^*E)\duer{E^*}$ is $TE^*$, modulo the reversal isomorphism
  \[-R\colon T^*E\to T^*E^*, \, -R(\dr_{e(m)}\ell_\varepsilon+_{E}(T_{e(m)}q_E)^*\theta_m)=-\dr_{\varepsilon(m)}\ell_e+_{E}(T_{\varepsilon(m)}q_{E^*})^*(\theta_m+\dr\langle e,\varepsilon\rangle)
\]
\cite{MaXu94}.  Therefore, the dual of $TE\times_MT^*E$ over $TM\oplus
E^*$ is, modulo the reversal isomorphism, $TE^*\times_M TE^*$ (the
pullback under $\iota$ of the vector bundle
$TE^*\times TE^*\to TM\times E^*$). Hence, the dual
of $TE\oplus T^*E$ over $TM\oplus E^*$ is the quotient
\[(TE\oplus T^*E)\duer{TM\oplus E^*}\simeq (TE^*\times_M TE^*)/(TE\oplus T^*E)^{\ann}.
\]
 Consider an arbitrary element
$(T_m\epsilon_1v_m+\eta_1^\uparrow(\epsilon_1(m)),T_m\epsilon_2w_m+\eta_2^\uparrow(\epsilon_2(m)))$
of $TE^*\times_M TE^*$, with
$\epsilon_1,\epsilon_2,\eta_1,\eta_2\in\Gamma(E^*)$, $v_m,w_m\in
T_mM$. Its pairing over
$(v_m,\epsilon_2(m))$ with an element
$(T_mev_m+(e')^\uparrow(e(m)),\dr_{e(m)}\ell_{\epsilon_2}+(T_{e(m)}q_E)^*\theta_m)$
of $TE\oplus T^*E$ over $e(m)\in E$ and $(v_m,\epsilon_2(m))\in
TM\oplus E^*$ is
\begin{equation*}\begin{split}
&\langle T_m\epsilon_1v_m+\eta_1^\uparrow(\epsilon_1(m)),
T_mev_m+(e')^\uparrow(e(m))\rangle\\
&+\langle
T_m\epsilon_2w_m+\eta_2^\uparrow(\epsilon_2(m)),
-\dr_{\epsilon_2(m)}\ell_e+(T_{\epsilon_2(m)}q_{E^*})^*(\theta_m+\dr\langle\epsilon_2,e\rangle),
\end{split}\end{equation*}
which is easily computed to be 
$v_m\langle \epsilon_1,e\rangle+\langle e, \eta_1-\eta_2\rangle(m)+\langle\theta_m, w_m\rangle+\langle e',\epsilon_1\rangle$.

In particular, we find that 
the fiber of $(TE\oplus T^*E)^{\ann}$ over $(v_m,\epsilon_m)\in TM\oplus E^*$ is 
\[\left\{\left.(T_m0^{E^*}v_m+\eta^\uparrow(0^{E^*}_m),\eta^\uparrow(\epsilon(m)))\right|
\eta_m\in E^*\right\},
\]
and that the core of $(TE\oplus T^*E)\duer{TM\oplus E^*}$ is identified as follows with $E^*$:
\[\eta_m\in E^* \mapsto \left(\left.\frac{d}{dt}\right\an{t=0}0^{E^*}_m+t\frac{\eta_m}{2},\left.\frac{d}{dt}\right\an{t=0}0^{E^*}_m-t\frac{\eta_m}{2}\right).
\]
 The sides of
$(TE\oplus T^*E)\duer{TM\oplus E^*}$ are $TM\oplus E^*$ and $E^*\oplus TM$.
The projection of the class in $(TE\oplus T^*E)\duer{TM\oplus E^*}$ of 
$(T_m\epsilon_1v_m+\eta_1^\uparrow(\epsilon_1(m)),T_m\epsilon_2w_m+\eta_2^\uparrow(\epsilon_2(m)))$ under
$\pi_1$ is $(v_m,\epsilon_2(m))$ and its
projection under $\pi_2$ is $(\epsilon_1(m),w_m)$.
A computation yields the equality of 
\[\mathcal
  I\left(T_m\epsilon_1(v_m)+\eta_1^\uparrow(\epsilon_1(m)),
  T_m\epsilon_2(w_m)+\eta_2^\uparrow(\epsilon_2(m))\right)\] with
\[\left(T_m\epsilon_2(w_m)+\eta_2^\uparrow(\epsilon_2(m)), T_m\epsilon_1(v_m)+\eta_1^\uparrow(\epsilon_1(m))\right).\]
The morphism $\mathcal I$ is therefore minus the identity on the core, and
it exchanges $\pi_1$ and $\pi_2$.

\end{example}

\subsubsection{The category of involutive double vector bundles}\label{morphisms_of_met_DVB}

\begin{definition}
A \textbf{morphism} 
$\Omega\colon D_1\to D_2$
of \textbf{involutive double vector bundles} is a morphism 
\begin{equation*}
  \begin{xy}
    \xymatrix{D_1\ar[rrr]^{\Omega}\ar[rd]^{\pi_1}\ar[dd]_{\pi_2}& && D_2\ar[rd]^{\pi_1}\ar[dd]
      &\\
& Q_1\ar[dd]\ar[rrr]^{\omega_Q}&&&Q_2\ar[dd]\\
      Q_1\ar[rd]\ar[rrr]&&&Q_2\ar[rd]&\\
&M_1\ar[rrr]^{\omega_0}&&&M_2}
\end{xy}
\end{equation*}
of double vector bundles such that
\[ \Omega\circ\mathcal I_1=\mathcal I_2\circ\Omega.
\]
We write $\operatorname{IDVB}$ for the obtained category of involutive
double vector bundles.
\end{definition} 

We call $\omega_{B^*}$ the induced morphism $B_1^*\to B_2^*$ on the
cores. We let the reader check that the morphism on the second side of
the diagram must coincide with $\omega_Q\colon Q_1\to Q_2$.


\begin{theorem}\label{met_maps}
  Let $(D_1;Q_1,Q_1;M_1)$ and $(D_2;Q_2,Q_2;M_2)$ be two involutive
  double vector bundles and consider the dual metric double vector
  bundles $(\mathbb E_1={D_1}\duer{\pi_1}, B_1,Q_1,M_1)$ and
  $(\mathbb E_2={D_2}\duer{\pi_1}, B_2,Q_2,M_2)$.  A morphism
  $\Omega\colon D_1\to D_2$ of involutive double vector bundles is
  equivalent to a pair of morphisms of modules
\begin{align*}
\omega^\star \colon \mathcal C(\mathbb E_2)\to \mathcal C(\mathbb E_1),\quad 
\omega_Q^\star \colon \Gamma(Q_2^*)\to \Gamma(Q_1^*)
\end{align*} over a smooth map
$\omega_0 \colon M_1\to M_2$, such that
$\omega^\star\left(\widetilde{\tau_1\wedge\tau_2}\right)=\widetilde{\omega_Q^\star\tau_1\wedge\omega_Q^\star\tau_2}
$
for all $\tau_1,\tau_2\in\Gamma(Q_2^*)$.
\end{theorem}
Recall that this means in particular that
$\omega^\star(q_Q^*f\cdot\chi)=q_{Q_1}^*(\omega_0^*f)\cdot\omega^\star(\chi)$
and $\omega_Q^\star(f\cdot \tau)=\omega_0^*f\cdot \omega_Q^\star\tau$
for all $\tau\in\Gamma(Q_2^*)$, $f\in C^\infty(M_2)$ and $\chi \in
\mathcal C(\mathbb E_2)$.

\begin{proof}
Recall that the restriction of $\Omega^\star$ to core sections
$\tau^\dagger$, $\tau\in\Gamma(Q_2^*)$ is given by
$\Omega^\star(\tau^\dagger)=(\omega_Q^\star(\tau))^\dagger$ (see
\S\ref{usual_VB_morphisms}). 

Recall further that $\Omega^\star$ restricts to a morphism
$\Omega^\star\colon \Gamma_{Q_2}^l(\mathbb E_2)\to
\Gamma_{Q_1}^l(\mathbb E_1)$ of modules over $\omega_0^*\colon
C^\infty(M_2)\to C^\infty(M_1)$. Choose $\chi\in\mathcal C(\mathbb
E_2)$. Since by \eqref{pullback_linear}
\[ \mathcal I_1^*(\ell_{\Omega^\star\chi})=\ell_{\mathcal
  I_1^\star\Omega^\star\chi}=\ell_{\Omega^\star \mathcal
  I_2^\star\chi}
=\Omega^*\mathcal I_2^*\ell_\chi=\Omega^*(-\ell_\chi)=-\ell_{\Omega^\star\chi},
\]
we find that $\Omega^\star\chi\in\mathcal C(\mathbb E_1)$ by Lemma \ref{inv_cois_sections}. Therefore,
$\Omega^\star$ restricts to a morphism
$\omega^\star\colon\mathcal C(\mathbb E_2)\to \mathcal C(\mathbb E_1)$ of modules over $\omega_0^*\colon
C^\infty(M_2)\to C^\infty(M_1)$.
Next we choose $\tau_1,\tau_2\in\Gamma(Q_2^*)$. Then 
\[
\frac{1}{2}\Omega^\star\left(
\pi_1^*\ell_{\tau_1}\tau_2^\dagger-\pi_1^*\ell_{\tau_2}\tau_1^\dagger
 \right)=\frac{1}{2}\left(\pi_1^*\ell_{\omega_Q^\star\tau_1}(\omega_Q^\star\tau_2)^\dagger-\pi_1^*\ell_{\omega_Q^\star\tau_2}(\omega_Q^\star\tau_1)^\dagger
\right)
\]
shows that
$\Omega^\star\left(\widetilde{\tau_1\wedge\tau_2}\right)=\widetilde{\omega_Q^\star\tau_1\wedge
  \omega_Q^\star\tau_2}$.

Since $\mathcal C(\mathbb E)$ and $\Gamma^c_Q(\mathbb E)$ span
pointwise $\mathbb E=D\duer{\pi_1}$, we find that the morphism
$\Omega$ is completely encoded by the two maps $\omega^\star$ and $\omega_Q^\star$.
\end{proof}

\begin{remark}
 Recall that the morphism
\[\omega_{B^*}^\star\colon \Gamma(B_2)\to \Gamma(B_1)
\]
of modules over $\omega_0^*\colon C^\infty(M)\to C^\infty(N)$,
i.e.~the vector bundle morphism $\omega_{B^*}\colon B_1^*\to B_2^*$ is
induced as follows by the two maps in the theorem.  If
$\chi\in\Gamma_{Q_2}^l(\mathbb E_2)$ is linear over $b\in \Gamma(B_2)$, then
$\omega^\star(\chi)$ is linear over $\omega_{B^*}^\star(b)$. 
\medskip

Further, a morphism $\Omega\colon
  Q_1\times_{M_1}Q_1\times_{M_1}B_1^*\to
  Q_2\times_{M_2}Q_2\times_{M_2}B_2^* $ of decomposed metric double
  vector bundles is described by $\omega_Q\colon Q_1\to Q_2$,
  $\omega_{B^*}\colon B_1^*\to B_2^*$ and $\omega_{12}\colon Q_1\wedge
  Q_1\to B_2^*$, all morphisms of vector bundles over a smooth map
  $\omega_0\colon M_1\to M_2$:
\[\Omega(q,q',\beta)=(\omega_Q(q),\omega_Q(q'),\omega_{B^*}(\beta)+\omega_{12}(q,q')).
\]
For $b\in\Gamma(B_1)$ the isotropic section $b^l\in\Gamma_{Q_2}^l(
  B_2\times_{M_2}Q_2\times_{M_2}Q_2^*)$,
  $b^l(q_m)=(b(m),q_m,0_m^{Q_2^*})$, is sent by $\omega^\star$ to the
  isotropic section
  $(\omega_B^\star(b))^l+\widetilde{\omega_{12}^\star(b)}\in\Gamma_{Q_1}^l(B_1\times_{M_1}Q_1\times_{M_1}Q_1^*)$.

\end{remark}

\subsection{Equivalence of [2]-manifolds and involutive double vector
  bundles}\label{eq_2_manifolds}
In this section we describe the equivalence of the category of
involutive double vector bundles with the category of $[2]$-manifolds.


\subsubsection{The functor $\mathcal M(\cdot)\colon \operatorname{IDVB}\to \operatorname{[2]-Man}$.}
Let $(D, Q, Q, M)$ be an involutive double vector bundle with core
$B^*$ and consider the dual metric double vector bundle $(\mathbb
E=D\duer Q, Q, B,M)$. We construct a $[2]$-manifold by assigning the
degree $0$ to elements of $C^\infty(M)$, the degree $1$ to elements of
$\Gamma(Q^*)$ and the degree $2$ to elements of $\mathcal C(\mathbb
E)$.

\medskip

Recall from Corollary \ref{cor_fat_CE} the existence of the vector
bundle $\widehat{B}$ over $M$ with $\Gamma(\widehat{B})=\mathcal
C(\mathbb E)$.  We construct as follows the sheaf $C^\infty(\mathcal
M(D))^\bullet$ of $\N$-graded, graded commutative, associative, unital
$C^\infty(M)$-algebras.  For an arbitrary open set $U\subseteq M$ we
set $C^\infty_U(\mathcal M(D))^0:=C^\infty_U(M)$. For $k\geq 0$ we set
$C^\infty_U(\mathcal M(D))^{2k}=\Gamma_U(S^k\widehat B)$, that is, the
space of symmetric elements of $\Gamma(\otimes_k\widehat B)$, and we
set $C^\infty_U(\mathcal M(D))^{2k+1}=\Gamma_U(Q^*\otimes S^k\widehat
B)$.  In particular, we have defined $C^\infty_U(\mathcal M(D))^{2}$
to be $\Gamma(\widehat B)\simeq \mathcal C(\mathbb E)$ and
$C^\infty_U(\mathcal M(D))^{1}$ to be $\Gamma(Q^*)$.  For each
$i\in\mathbb N$, the sheaf $C^\infty_U(\mathcal M(D))^{i}$ is a sheaf
of $C^\infty(M)$-modules, so the multiplication of $f\in
C^\infty_U(M)$ with $\xi\in C^\infty_U(\mathcal M(D))^i$ is already
given. Note that elements $\frac{1}{k!}\sum_{\sigma\in
  S_{k}}\chi_{\sigma(1)}\otimes\ldots\otimes\chi_{\sigma(k)}=:\chi_1\cdot\ldots\cdot\chi_k$
with $\chi_1,\ldots,\chi_k\in \widehat{B}_p$ generate
$S^k\widehat{B}_p$ over a point $p\in M$.  The symmetric product
$(\cdot)\colon S^k\widehat{B}\otimes S^l\widehat{B}\to
S^{k+l}\widehat{B}$ sends generators
$\xi\otimes\eta=(\chi_1\cdot\ldots\cdot\chi_k)\otimes(\chi_{k+1}\cdot\ldots\cdot\chi_{k+l})$
to $\frac{1}{(k+l)!}\sum_{\sigma\in
  S_{k+l}}\chi_{\sigma(1)}\otimes\ldots\otimes\chi_{\sigma(k+l)}$, and
induces a product $(\cdot)\colon \Gamma(S^k\widehat{B})\otimes
\Gamma(S^l\widehat{B})\to \Gamma(S^{k+l}\widehat{B})$, wich gives us
the product $(\cdot)\colon C^\infty_U(\mathcal M(D))^{2k}\otimes
C^\infty_U(\mathcal M(D))^{2l}\to C^\infty_U(\mathcal M(D))^{2(k+l)}$
for $k,l\geq 0$. We further define the products
$(\cdot)\colon(Q^*\otimes S^k\widehat{B})\otimes S^l\widehat{B}\to
Q^*\otimes S^{k+l}\widehat{B}$ and $(\cdot)\colon S^k\widehat{B}\otimes
(Q^*\otimes S^l\widehat{B})\to Q^*\otimes S^{k+l}\widehat{B}$ by
\[(\tau\otimes\chi_1\cdot\ldots\cdot\chi_k)\cdot
(\chi_{k+1}\cdot\ldots\cdot\chi_{k+l})
=\tau\otimes(\chi_1\cdot\ldots\cdot\chi_{k+l})=(\chi_1\cdot\ldots\cdot\chi_k)\cdot
(\tau\otimes\chi_{k+1}\cdot\ldots\cdot\chi_{k+l}),\]
and the skew-symmetric product 
\[(\cdot) \colon(Q^*\otimes S^k\widehat{B})\otimes(Q^*\otimes S^l\widehat{B})\to S^{k+l+1}\widehat{B},\]
\[(\tau_1\otimes\chi_1\cdot\ldots\cdot\chi_k)\cdot (\tau_2\otimes \chi_{k+1}\cdot\ldots\cdot\chi_{k+l})
=\widetilde{\tau_1\wedge\tau_2}\cdot\chi_1\cdot\ldots\cdot\chi_{k+l}.\]
These pointwise
multiplications induce an associative,
graded-commutative multiplication
\[(\cdot)\colon C^\infty(\mathcal M(D))^{\bullet}\times
C^\infty(\mathcal M(D))^{\blacktriangle}\to C^\infty(\mathcal
M(D))^{\bullet+\blacktriangle}.
\]

Choose now $p\in M$ and a coordinate neighborhood $U\ni p$
such that $Q^*$ and $B$ are trivialised on $U$ by the basis frames
$(\tau_1,\ldots, \tau_{r_1})$ and $(b_1,\ldots,b_{r_2})$. Recall also that after the choice of
a Lagrangian splitting $\Sigma\colon Q\times_M B\to \mathbb E$, each
element $\chi\in\mathcal C(\mathbb E)$ over $b\in\Gamma(B)$ can be
written as $\sigma_B(b)+\widetilde{\phi}$ with
$\phi\in\Omega^2(Q)$.
Then $C^\infty_U(\mathcal M(D))$ is generated on $U$ as a
$C^\infty_U(M)$-algebra by
$\{\tau_1,\ldots,\tau_{r_1},\sigma_B(b_1),\ldots,\sigma_B(b_{r_2})\}$ (see
also the proof of Proposition \ref{prop_CE2}), where
$\sigma_B(b_1),\ldots,\sigma_B(b_{r_2})$ are considered as sections of
$\widehat{B}$. We obtain so the $[2]$-manifold $\mathcal M(D)$ of
dimension $(m;r_1,r_2)$, where $m$ is the dimension of $m$, $r_1$ is
the rank of $Q$ and $r_2$ is the rank of $B$.


  \bigskip We have constructed a map $\mathcal M(\cdot)$ sending
  involutive double vector bundles to $[2]$-manifolds.  By
  Theorem \ref{met_maps} a morphism
  $\Omega\colon D_1\to D_2$ of metric double
  vector bundles is the same as a triple of maps
  \[\omega_0\colon M_1\to M_2 \quad \Leftrightarrow \quad \omega_0^*\colon
  C^\infty(M_2)\to C^\infty(M_1),\]
\[\omega^\star\colon \mathcal C(\mathbb E_2)\to \mathcal C(\mathbb E_1) \quad 
\text{ and } \quad \omega_Q^\star\colon \Gamma(Q_2^*)\to \Gamma(Q_1^*)
\]
with
\[
\omega^\star\left(\widetilde{\tau_1\wedge\tau_2}\right)=\widetilde{\omega_Q^\star\tau_1\wedge\omega_Q^\star\tau_2},\qquad
q_{Q_1}^*\omega_0^*f\cdot
\omega^\star(\chi)=\omega^\star(q_{Q_2}^*f\cdot\chi)\] and
\[\omega_0^*f\cdot\omega_Q^\star(\tau)=\omega^\star(f\cdot\tau)
\]
for $f\in C^\infty(M_2)$, $\tau\in \Gamma(Q_2^*)$ and $\chi\in
\mathcal C(\mathbb E_2)$. Hence we find that the triple
$(\omega^\star,\omega_Q^\star,\omega_0^*)$ defines in this manner a
morphism $\mathcal M(\Omega)\colon \mathcal M(D_1)\to \mathcal M(D_2)$
of the $[2]$-manifolds constructed above.  \medskip

 We have so defined a covariant functor $\mathcal M(\cdot)\colon
 \operatorname{IDVB}\to \operatorname{[2]-Man}$ from the category of
 involutive double vector bundles to the category of $[2]$-manifolds.

 \subsubsection{The functor $\mathcal G\colon\operatorname{[2]-Man}
   \to \operatorname{IDVB}$.}\label{Geom_functor} (The letter $\mathcal G$ stands for
 \emph{geometrisation}.)  Conversely, we construct explicitly a metric
 double vector bundle associated to a given $[2]$-manifold $\mathcal
 M$. The idea is to adapt the construction of the equivalence of
 locally free and finitely generated sheaves of $C^\infty(M)$-modules
 with vector bundles over $M$ (see \S\ref{classical_eq}).

First we give Pradines'
original definition of a double vector bundle \cite{Pradines77} (in
the smooth and finite-dimensional case).
\begin{definition}\cite[C. \S 1]{Pradines77}\label{double_atlas}
  Let $M$ be a smooth manifold and $\mathbb E$ a topological space
  with a map $\Pi\colon \mathbb E\to M$. A \textbf{double vector
    bundle chart} is a quintuple $c=(U,\Theta,V_1,V_2,V_0)$, where $U$
  is an open set in $M$, $V_1,V_2,V_3$ are three (finite dimensional)
  vector spaces and $\Theta\colon \Pi\inv(U)\to U\times V_1\times
  V_2\times V_0$ is a homeomorphism such that $\Pi=\pr_1\circ\Theta$.

  Two smooth double vector bundle charts $c$ and $c'$ are
  \textbf{smoothly compatible} if the ``change of chart''
  $\Theta'\circ\Theta\inv$ over $U\cap U'$ has the following form:
\[(x,v_1,v_2,v_0)\mapsto (x,A_1(x)v_1,A_2(x)v_2,A_0(x)v_0+\omega(x)(v_1,v_2))
\]
with $x\in U\cap U'$, $v_i\in V_i$, $A_i\in C^\infty(M,
\operatorname{Gl}(V_i))$ for $i=0,1,2$ and $\omega\in
C^\infty(M,\operatorname{Hom}(V_1\otimes V_2,V_0))$.

A \textbf{smooth double vector bundle atlas} $\lie A$ on $\mathbb E$
is a set of double vector bundle charts of $\mathbb E$ that are
pairwise smoothly compatible and such that the set of underlying open
sets in $M$ is a covering of $M$.  As usual, $\mathbb E$ is then a
smooth manifold and two smooth double vector bundle atlases $\lie A_1$
and $\lie A_2$ are \textbf{equivalent} if their union is a smooth
atlas.  A (smooth) double vector bundle structure on $\mathbb E$ is an
equivalence class of smooth double vector bundle atlases on $\mathbb
E$.
\end{definition}

Given a $[2]$-manifold $\mathcal M$, we interpret its local generators
as smooth frames given by smoothly compatible double vector bundle
charts, and we show that the obtained smooth double vector bundle has a
natural linear metric.

Let $M$ be the smooth manifold underlying $\mathcal M$ and assume that
$\mathcal M$ has dimension $(l;m,n)$.  Choose a maximal open covering
$\{U_\alpha\}$ of $M$ such that $C^\infty_{U_\alpha}(\mathcal M)$ is
freely generated by $\xi_1^\alpha,\ldots,\xi_m^\alpha$ (in degree $1$)
and $\eta^\alpha_1,\ldots,\eta^\alpha_n$ (degree $2$ generators).
Smooth functions of degree $2$ on $U_\alpha$ are therefore
$C^\infty(U_\alpha)$-linear combinations of
$\eta^\alpha_1,\ldots,\eta^\alpha_n$ and of $\xi^\alpha_k\wedge
\xi^\alpha_l$ for $1\leq k<l\leq m$.  Therefore, if $\eta\in
C^\infty_{U_\alpha}(\mathcal M)$ has degree $2$, then $\eta$ is
represented in this basis by a pair $(v,G)$, with $v$ a vector $v\in
(C^\infty(U_\alpha))^n$ and $G$ a smooth, $(n\times n)$-skew-symmetric
matrix valued function on $U_\alpha$:
\[\eta=\sum_{i=1}^nv_i\eta_i^\alpha+\sum_{1\leq k<l\leq m}G_{kl}\xi^\alpha_k\wedge\xi^\alpha_l.
\]
Denote by $\Lambda(m,\R)$ the space of skew-symmetric, real $(m\times
m)$-matrices.

Choose now $\alpha,\beta$ such that $U_\alpha\cap U_\beta\neq
\emptyset$. Then each generator $\xi_i^\beta$ can be written in a
unique manner as $\sum_{j=1}^m\omega^{\alpha\beta}_{ji}\xi^\alpha_j$
with $\omega^{\alpha\beta}_{ji}\in C^\infty(U_\alpha\cap U_\beta)$. Define
$\omega^{\alpha\beta}\in C^\infty(U_\alpha\cap
U_\beta,\operatorname{Gl}(m,\R))$ by
$\omega^{\alpha\beta}_p:=\omega^{\alpha\beta}(p)=(\omega_{ij}^{\alpha\beta}(p))_{i,j}$
for all $p\in U_\alpha\cap U_\beta$.  Then, if a degree $1$ function
$\xi\in C^\infty_{U_\alpha\cap U_\beta}(\mathcal M)$ has the
coordinates $(f_1,\ldots, f_m)$, in the $C^\infty(U_\alpha\cap
U_\beta)$-basis $(\xi_1^\beta,\ldots,\xi_m^\beta)$, then it has
coordinates $\omega^{\alpha\beta}\cdot(f_1,\ldots, f_m)^t$ in the
$C^\infty(U_\alpha\cap U_\beta)$-basis
$(\xi_1^\alpha,\ldots,\xi_m^\alpha)$.

Each generator $\eta^\beta_i$ can be written as \[
\eta^\beta_i=\sum_{j=1}^n\psi_{ji}^{\alpha\beta}\eta^\alpha_j+\sum_{1\leq
  k<l\leq m}\rho_{kli}^{\alpha\beta}\,\xi^\alpha_k\wedge\xi^\alpha_l\]
with $\psi^{\alpha\beta}_{ji}, \rho^{\alpha\beta}_{kli}\in
C^\infty(U_\alpha\cap U_\beta)$.  Set
$\psi^{\alpha\beta}=(\psi_{ij}^{\alpha\beta})_{i,j}\in
C^\infty(U_\alpha\cap U_\beta,\operatorname{Gl}(n,\R))$. Define
$\rho^{\alpha\beta}\in C^\infty(U_\alpha\cap U_\beta,
\operatorname{Hom}(\R^n,\Lambda(m,\mathbb R)))$ by $\langle
e_l,\rho^{\alpha\beta}(p)(e_i)\cdot
e_k\rangle=\rho_{kli}^{\alpha\beta}(p)$ for $p\in U_\alpha\cap
U_\beta$, $1\leq k<l\leq m$ and $i=1,\ldots,n$.  Then if a degree $2$
function $\eta\in C^\infty_{U_\alpha\cap U_\beta}(\mathcal M)$ has the
coordinates $(v,G)$, in the $C^\infty(U_\alpha\cap U_\beta)$-basis
$(\eta_1^\beta,\ldots,\eta_m^\beta)\cup(\xi_k^\beta\wedge \xi_l^\beta
\mid 1\leq k<l\leq m)$, then it has coordinates
$(\psi^{\alpha\beta}\cdot v,
\rho^{\alpha\beta}(v)+\omega^{\alpha\beta}\cdot G\cdot
(\omega^{\alpha\beta})^t)$ in the $C^\infty(U_\alpha\cap
U_\beta)$-basis
$(\eta_1^\alpha,\ldots,\eta_m^\alpha)\cup(\xi_k^\alpha\wedge
\xi_l^\alpha \mid 1\leq k<l\leq m)$.

 Then
by construction
\begin{equation}\label{2_cocycle}
\begin{split}
  \omega^{\gamma\beta}_p&=\omega^{\gamma\alpha}_p\cdot
  \omega^{\alpha\beta}_p
  ,\qquad \psi^{\gamma\beta}_p=\psi^{\gamma\alpha}_p\cdot \psi^{\alpha\beta}_p\quad\text{ and }\\
  \rho^{\gamma\beta}_p(v)&=\rho^{\gamma\alpha}_p(\psi^{\alpha\beta}_p(v))+\omega^{\gamma\alpha}_p\cdot\rho^{\alpha\beta}_p(v)\cdot(\omega^{\gamma\alpha}_p)^t
\end{split}
\end{equation}
for all $p\in U_\alpha\cap U_\beta\cap U_\gamma$ and all $v\in \R^n$.
By construction we have also $\omega^{\alpha\alpha}=1$,
$\psi^{\alpha\alpha}=1$ and $\rho^{\alpha\alpha}=0$, which yields
$\rho^{\beta\alpha}_p(\psi^{\alpha\beta}_p(v))=-\omega^{\beta\alpha}_p\cdot\rho^{\alpha\beta}_p(v)\cdot(\omega^{\beta\alpha}_p)^t$.

Set $\tilde{\mathbb E}=\bigsqcup_\alpha
U_\alpha\times \R^m\times \R^n\times \R^m$ (the disjoint union)
and identify \[ (p,w,v,u)\in U_\beta\times\R^m\times \R^n\times \R^m \] with
\[\left(p, (\omega^{\beta\alpha}_p)^t w,\psi^{\alpha\beta}_p
  v,\omega^{\alpha\beta}_pu+\rho^{\alpha\beta}_p(v)((\omega^{\beta\alpha}_p)^tw)\right)\]
in $U_\alpha\times\R^m\times \R^n\times \R^m$ for $p\in U_\alpha\cap
U_\beta$. The cocycle equations \eqref{2_cocycle} imply that this
defines an equivalence relation on $\tilde{\mathbb E}$.  The quotient
space is $\mathbb E$, a double vector bundle: The map $\Pi\colon
\mathbb E\to M$, $[p,w,v,u]\mapsto p$ is well-defined and, by
construction, the charts
$c_\alpha=(U_\alpha,\Theta_\alpha=\Id,\R^m,\R^n,\R^m)$ define a smooth
double vector bundle atlas on $\mathbb E$ with smooth changes of
charts $\Theta_\alpha\circ\Theta_\beta\inv\colon (U_\alpha\cap
U_\beta)\times \R^m\times\R^n\times\R^m\to (U_\alpha\cap
U_\beta)\times \R^m\times\R^n\times\R^m$,
\[ (\Theta_\alpha\circ\Theta_\beta\inv)(p,w,v,u)=\left(p, (\omega^{\beta\alpha}_p)^t w,\psi^{\alpha\beta}_p
  v,\omega^{\alpha\beta}_pu+\rho^{\alpha\beta}_p(v)((\omega^{\beta\alpha}_p)^tw)\right).
\]
 Since the covering was
chosen to be maximal, the obtained double vector bundle $\mathbb E$ does
not depend on any choices.

Recall from the proof of Theorem \ref{split_N} that there are two
vector bundles $E_{1}$ and $E_{2}$ associated canonically to a
$[2]$-manifold $\mathcal M$ (only the inclusion of $\Gamma(E_{2})$
in $C^\infty(\mathcal M)^2$ is non-canonical). $E_{1}$ and $E_{2}^*$ are the
sides of $\mathbb E$ and $E_{1}^*$ is the core of $\mathbb E$.  The
vector bundle $E_{1}^*$ is defined by the cocycles
$\omega^{\alpha\beta}$; $E_1^*=\tilde E_1/\sim$ with $\tilde
E_{1}=\bigsqcup_{\alpha}U_\alpha\times \R^m$, and $(p,v)\sim
(x,\omega^{\alpha\beta}_p(v))$ for $p\in U_\alpha\cap U_\beta$. The
vector bundle $E_{2}$ can be defined in the same manner using
the cocycles $\psi^{\alpha\beta}$ and the model space $\R^n$.

Finally we check the existence of a canonical linear metric on
$\mathbb E$. Over a chart domain $U_\alpha$ we set
\[\langle (p, w,v,u), (p,w',v,u')\rangle=\langle u, w'\rangle+\langle
u', w\rangle,\]
where the pairing used on the right-hand side is just the standard
scalar product on $\R^m$.
By construction, this does not depend on the choice of $\alpha$ with
$p\in U_\alpha$, since 
\begin{equation*}
\begin{split}
&\left\langle \left(p, (\omega^{\beta\alpha}_p)^t w,\psi^{\alpha\beta}_p
  v,\omega^{\alpha\beta}_pu+\rho^{\alpha\beta}_p(v)((\omega^{\beta\alpha}_p)^tw)\right),
\left(p, (\omega^{\beta\alpha}_p)^t w',\psi^{\alpha\beta}_p
  v,\omega^{\alpha\beta}_pu'+\rho^{\alpha\beta}_p(v)((\omega^{\beta\alpha}_p)^tw')\right)
\right\rangle\\
&=\langle (\omega^{\beta\alpha}_p)^t w,\omega^{\alpha\beta}_pu'\rangle+\cancel{\rho^{\alpha\beta}_p(v)((\omega^{\beta\alpha}_p)^tw', (\omega^{\beta\alpha}_p)^t w)}
+\langle (\omega^{\beta\alpha}_p)^t w',\omega^{\alpha\beta}_pu\rangle+\cancel{\rho^{\alpha\beta}_p(v)((\omega^{\beta\alpha}_p)^tw, (\omega^{\beta\alpha}_p)^t w')}\\
&=\langle w,u'\rangle+\langle w',u\rangle.
\end{split}
\end{equation*}

By dualising $\mathbb E$ over $E_1$, we get an involutive double
vector bundle, which we call $\mathcal G(\mathcal M)$, with sides
$E_1$ and core $E_2$.  Again by
definition of the morphisms in the category of $[2]$-manifolds
(Definition \ref{n_man}) and in
the category of involutive double vector bundles (Theorem \ref{met_maps}), this defines a
functor $\mathcal G\colon \operatorname{[2]-Man}\to
\operatorname{IDVB}$ between the two categories.

\subsubsection{Equivalence of categories.}

Finally we need to prove that the two obtained functors define an
equivalence of categories.  The functor $\mathcal G\circ\mathcal
M(\cdot)$ is the functor that sends an involutive double vector bundle
to its maximal double vector bundle atlas, hence it is naturally isomorphic to the identity
functor.

The functor $\mathcal M(\cdot)\circ \mathcal G\colon
\operatorname{[2]-Man}\to\operatorname{[2]-Man}$ sends a
$[2]$-manifold $\mathcal M$ over $M$ with degree 1 local generators $\xi_\alpha^i$ and
cocycles $\omega^{\alpha\beta}$ and degree 2 generators $\eta_\alpha^i$ and cocycles
$\psi^{\alpha\beta}$ and $\rho^{\alpha\beta}$ on $U_\alpha$
to the sheaf of core and coisotropic linear sections of the dual of
$\mathcal G(\mathcal M)$  with degree 1 local generators $\xi_\alpha^i$ and
cocycles $\omega^{\alpha\beta}$ and degree 2 generators $\eta_\alpha^i$ and cocycles
$\psi^{\alpha\beta}$ and $\rho^{\alpha\beta}$ on $U_\alpha$. 
There is an obvious natural isomorphism between this functor and the
identity functor $\operatorname{[2]-Man}\to\operatorname{[2]-Man}$.

Hence we have the following result.
\begin{theorem}\label{main_crucial}
  The functors
  $\mathcal G\colon [2]\operatorname{-Man}\to \operatorname{IDVB}$ and
  $\mathcal M\colon \operatorname{IDVB}\to [2]\operatorname{-Man}$ and
  the two natural isomorphisms above are an equivalence between the
  category of involutive double vector bundles and the category of
  $[2]$-manifolds.
\end{theorem}

\begin{remark}\label{dcm}
  Note that del Carpio-Marek's defines \emph{self-conjugate double
    vector bundles} in his thesis \cite{delCarpio-Marek15}: those are
  double vector bundles $(D;Q,Q;M)$ with identical sides and a
  morphism $\mathcal H\colon D\to D$ of double vector bundles
  satisfying $\mathcal H^4=\Id_D$, $\pi_1\circ\mathcal H=-\pi_2$,
  $\pi_2\circ\mathcal H=\pi_1$, and restricting to the identity on the
  core. Morphisms of self-conjugate double vector bundles are defined
  like our morphisms of involutive double vector bundles.
  \cite{delCarpio-Marek15} shows that self-conjugate double vector bundles are
  dual to metric double vector bundles, and establishes an equivalence
  between their category and the category of $[2]$-manifolds.

In (2) of Proposition \ref{dual_inv_metric}, we work with the
  nondegenerate pairing $\nsp{\,}{}\colon\mathbb
  E\duer{B}\times_{Q}\mathbb E\duer{Q}$ to define the involution on
  the dual $E\duer{Q}$ of a metric double vector bundle $(\mathbb E;
  B,Q;M)$. The difference between our definition and del
  Carpio-Marek's is due to his choice of the
  nondegenerate pairing $\mathbb E\duer{Q}\times_{Q}\mathbb
  E\duer{B}$, which equals $\nsp{\,}{}$ up to a sign. Accordingly, in
  (1) of Proposition \ref{dual_inv_metric}, the metric on the dual
  $\mathbb E=\mathbb D\duer{\pi_1}$ of an involutive double vector
  bundle in the sense of \cite{delCarpio-Marek15} would be given by
  $\langle e_1,e_2\rangle_{\mathbb E}=\langle e_1, d\rangle_Q-\langle
  e_2,\mathcal H(d)\rangle_Q$ for $(e_1,e_2)\in\mathbb
  E\times_B\mathbb E$ and any $d\in D$ with $\pi_Q(e_1)=\pi_1(d)$ and
  $\pi_Q(e_2)=\pi_2(d)$.

  Del Carpio-Marek's analogue of Proposition \ref{dual_inv_metric} is
  phrased as follows: he shows that the defining map
  $\mathcal H$ of a self-conjugate double vector bundle $(D,Q,Q,M)$
  with core $B^*$ is equivalent to an isomorphism
  $D\duer{\pi_1}\simeq(D\duer{\pi_1})\duer{B}$, and so to a linear metric
  on $D\duer{\pi_1}$. This result is (up to a sign) equivalent to our
  Proposition \ref{dual_inv_metric}, and the two approaches are
  therefore very similar in nature (although developed
  independently). Del Carpio-Marek's equivalence of self-conjugate
  double vector bundles with $[2]$-manifolds is then based on an
  equivalence of self-conjugate double vector bundles with short exact
  sequences as in Lemma \ref{cor_fat_CE}, the duals of which belong to
  a family of short exact sequences that was reportedly proved by
  Bursztyn, Cattaneo, Mehta and Zambon\footnote{In an unpublished work
    in preparation.} to be equivalent to $[2]$-manifolds.
Up to the sign convention in the construction of the dual to a given
  metric double vector bundle, our Theorem \ref{main_crucial} and del
  Carpio-Marek's equivalence of categories work the same: the
  functions of a given
  $[2]$-manifold are interpreted in both methods as the
  special sections of two metric double vector bundles in the same
  isomorphy class.
\end{remark}

\subsubsection{Correspondence of splittings.}\label{cor_splittings}
Via the functors above, a decomposed involutive double vector bundles
$Q\times_MQ\times_MB^*$ is sent to a split $[2]$-manifold $Q[-1]\oplus
B^*[-2]$ and vice versa.

Choose an involutive double vector bundle $(D;Q,Q;M)$ with core $B^*$
and the corresponding $[2]$-manifold $\mathcal M$.  Each choice of an
involutive decomposition $\mathbb I$ of $D$ is equivalent to a
choice of splitting $\mathcal S$ of the corresponding $[2]$-manifold,
such that the following diagram commutes
\begin{equation*}
\begin{xy}
  \xymatrix{
    D\ar[r]^{\mathbb I\qquad }\ar[d]_{\mathcal M(\cdot)}
&Q\times_MQ\times_MB^*\ar[d]^{\mathcal M(\cdot)}\\
    \mathcal M(D)\ar[r]_{\mathcal S\qquad }&Q[-1]\oplus B^*[-2].  }
\end{xy}
\end{equation*}

Note also that the category of split $[2]$-manifolds is equivalent to
the category of $[2]$-manifolds, and the category of decomposed metric
double vector bundles is equivalent to the category of metric double
vector bundles.  We will use this in the following section.

\subsubsection{Geometric interpretation of the local generators $C^\infty(\mathcal M(D))$}\label{geom_int}
The equivalence of a vector bundle $E$ over a smooth manifold $M$ with
the $[1]$-manifold $E[-1]$ (see \S\ref{classical_eq}) is often
described as follows: \emph{the generators of $C^\infty(E[-1])$ are the
  linear functions on $E$}. 

In degree $2$, the $[2]$-manifold that corresponds to an involutive
double vector bundle $(D,Q,Q,M)$ with core $B^*$ and involution
$\mathcal I$ is generated by $\Gamma^c_Q(\mathbb E)\simeq \Gamma(Q^*)$
in degree $1$ and by $\mathcal C(\mathbb E)$ in degree $2$, where
$\mathbb E=D\duer{\pi_1}$. Recall that $\mathbb E$ is dual to
$\pi_1\colon D\to Q$ by construction, but also, since $\mathbb E\simeq
\mathbb E\duer B$, dual to $\pi_2\colon D\to Q$ via $\nsp{\cdot}{\cdot}\colon
\mathbb E\times_{Q,\pi_2} D\to \mathbb R$.

Since both $C^\infty(M)$-modules are contained in $\Gamma_Q(\mathbb
E)$, their elements can be understood in two manners as linear
functions on $\mathbb E$. The duality of $\mathbb E$ with $D$ over
$\pi_1$ sends $\tau^\dagger\in\Gamma_Q^c(\mathbb E)$ to
$\ell_{\tau^\dagger}=\pi_2^*\ell_\tau\in C^\infty(D)$ and
$\chi\in\mathcal C(\mathbb E)$ to $\ell_{\chi}\in C^\infty(D)$. The
duality of $\mathbb E$ with $D$ over $\pi_2$ sends
$\tau^\dagger\in\Gamma_Q^c(\mathbb E)$ to $\pi_1^*\ell_\tau=\mathcal
I^*(\ell_{\tau^\dagger})\in C^\infty(D)$ and $\chi\in\mathcal
C(\mathbb E)$ to $\mathcal I^*(\ell_{\chi})=-\ell_{\chi}\in
C^\infty(D)$. Define $\mathcal P(D)\subseteq C^\infty(D)$ to be
$C^\infty(M)$-module of functions that are affine linear in the fibers
of $\pi_1$ and of $\pi_2$ and on which $\mathcal I^*\colon
C^\infty(D)\to C^\infty(D)$ is just multiplication with $-1$.
 
A careful study of $\mathcal P(\mathbb E)$ shows that the morphism of
$C^\infty(M)$-modules $\psi\colon \Gamma^c_Q(\mathbb E)\oplus \mathcal
C(\mathbb E)\to\mathcal P(D)$ sending $\tau^\dagger\in
\Gamma^c_Q(\mathbb E)$ to
$\frac{1}{2}(\pi_2^*\ell_\tau-\pi_1^*\ell_\tau)=\frac{1}{2}(\ell_{\tau^\dagger}-\mathcal
I^*\ell_{\tau^\dagger})$ and $\chi\in\mathcal C(\mathbb E)$ to
$\ell_\chi=\frac{1}{2}(\ell_\chi-\mathcal I^*\ell_\chi)$ is an
isomorphism. Given a splitting $\Sigma\colon B\times_MQ\to\mathbb E$,
$\psi(\tau^\dagger)$ is the function that sends
$d=\Sigma^\star(q_1,q_2)+_1(0^1_{q_1}+_2\bar\beta)$ to
$\frac{1}{2}\langle\tau, q_2-q_1\rangle$, and
$\psi(\sigma_B(b)+\tilde{\omega})$ is the function that sends $d$ to
$\langle b,\beta\rangle+\omega(q_1,q_2)$. This shows that the elements
of $\mathcal P(D)$ are polynomial in the sides of $D$.

In this picture, the degrees assigned in a rather artificial manner to
the elements of $\Gamma^c_Q(\mathbb E)\oplus \mathcal C(\mathbb E)$
become more natural: the elements of $\Gamma^c_Q(\mathbb E)$
correspond via $\psi$ to functions on $D$ that are polynomial of
degree $1$ in the sides of $D$, and the elements of $\mathcal
C(\mathbb E)$ correspond via $\psi$ to functions on $D$ that are
polynomial of degree $2$ in the sides of $D$. 

Finally, for $\tau_1,\tau_2\in\Gamma(Q^*)$, the function
$\psi(\widetilde{\tau_1\wedge\tau_2})=\ell_{\widetilde{\tau_1\wedge\tau_2}}=\frac{1}{2}(\pi_1^*\ell_{\tau_1}\pi_2^*\ell_{\tau_2}-\pi_1^*\ell_{\tau_2}\pi_2^*\ell_{\tau_1})$
equals
$\frac{1}{2}\left((\pi_1^*\ell_{\tau_1}+\pi_2^*\ell_{\tau_1})\psi(\tau_2^\dagger)-(\pi_1^*\ell_{\tau_2}+\pi_2^*\ell_{\tau_2})\psi(\tau_1^\dagger)\right)$.

\section{Poisson $[2]$-manifolds, metric VB-algebroids and Poisson
  involutive double vector bundles.}\label{sec:Poisson}
In this section we study $[2]$-manifolds endowed with a Poisson
structure of degree $-2$. We show how split Poisson $[2]$-manifolds
are equivalent to a special family of 2-representations. Then we prove
that Poisson $[2]$-manifolds are equivalent to metric double vector
bundles endowed with a linear Lie algebroid structure that is
compatible with the metric, or equivalenty to involutive double vector
bundles with a linear Poisson structure that is $\mathcal I$-invariant.

\begin{definition}
  A Poisson $[2]$-manifold is a $[2]$-manifold endowed with a Poisson
  structure of degree $-2$. A morphism of Poisson $[2]$-manifolds 
is a morphism of $[2]$-manifolds that preserves the Poisson structure. 
\end{definition}

Note that a Poisson bracket of degree $-2$ on a $[2]$-manifold
$\mathcal M$ is an $\R$-bilinear map $\{\cdot\,,\cdot\}\colon
C^\infty(\mathcal M)\times C^\infty(\mathcal M)\to C^\infty(\mathcal
M)$ of the graded sheaves of functions, such that\linebreak $|\{\xi,\eta\}|=|\xi|+|\eta|-2$
for homogeneous elements $\xi,\eta\in C^\infty_{\mathcal M}(U)$. The
bracket is graded skew-symmetric;
$\{\xi,\eta\}=-(-1)^{|\xi|\,|\eta|}\{\eta,\xi\}$ and satisfies the
graded Leibniz and Jacobi identities
 \begin{equation}\label{graded_leibniz}
\{\xi_1,\xi_2\cdot\xi_3\}=\{\xi_1,\xi_2\}\cdot\xi_3+(-1)^{|\xi_1|\,|\xi_2|}\xi_2\cdot\{\xi_1,\xi_3\}
\end{equation}
and
\begin{equation}\label{graded_jacobi}
\{\xi_1,\{\xi_2,\xi_3\}\}=\{\{\xi_1,\xi_2\},\xi_3\}+(-1)^{|\xi_1|\,|\xi_2|}\{\xi_2,\{\xi_1,\xi_3\}\}
\end{equation}
for homogeneous $\xi_1,\xi_2,\xi_3\in C^\infty_{\mathcal M}(U)$.
A morphism $\mu\colon\mathcal N\to\mathcal M$ of Poisson
$[2]$-manifolds
satisfies $\mu^\star\{\xi_1,\xi_2\}=\{\mu^\star\xi_1,\mu^\star\xi_2\}$
for all $\xi_1,\xi_2\in C^\infty_{\mathcal M}(U)$, $U$ open in $M$.

\subsection{Split Poisson $[2]$-manifolds and self-dual $2$-representations}\label{sec:Poisson2man}
We begin by defining self-dual $2$-representations. Recall from
\S\ref{dual_and_ruths} the dual of a $2$-representation.

\begin{definition}\label{saruth}
  Let $(A,\rho,[\cdot\,,\cdot])$ be a Lie algebroid. A
  2-representation $(\nabla^Q,\nabla^{Q^*}, R)$ of $A$ on a complex
  $\partial_Q\colon Q^*\to Q$ is said to be \textbf{self-dual} if it
  equals its dual, i.e.~$\partial_Q=\partial_Q^*$,
the connections $\nabla^Q$ and $\nabla^{Q^*}$ are dual to each other,
and 
  $R^*=-R\in\Omega^2(A,\operatorname{Hom}(Q,Q^*))$.
\end{definition}

We prove the following result.
\begin{theorem}\label{poisson_is_saruth}
  There is a bijection between split Poisson $[2]$-manifolds and
   self-dual 2-representations.
 \end{theorem}

\begin{proof}
  First let us consider a split 2-manifold $\mathcal M=Q[-1]\oplus
  B^*[-2]$. (For simplicity, we adopt the notation $Q=E_1$ and $B^*=E_2$
  for the bases of the metric double vector bundle
  geometrising $\mathcal M$.) That is, $Q$ and $B$ are vector bundles
  over $M$ and the functions of degree $0$ on $\mathcal M$ are the
  elements of $C^\infty(M)$, the functions of degree $1$ are sections
  of $Q^*$ and the functions of degree $2$ are sections of $B\oplus
  Q^*\wedge Q^*$.  Let us now take a Poisson bracket
  $\{\cdot\,,\cdot\}$ of degree $-2$ on $C^\infty(\mathcal M)$.  In
  the following, we consider arbitrary $f,f_1,f_2\in C^\infty(M)$,
  $\tau,\tau_1,\tau_2\in\Gamma(Q^*)$, and $b,b_1,b_2\in\Gamma(B)$.

The brackets $\{f_1,f_2\}$,
  $\{f,\tau\}$ have degree $-2$ and $-1$, respectively, and
  must hence vanish. The bracket $\{\tau_1,\tau_2\}$ is a function
  on $M$ because it has degree $0$.
  Since $\{f,\tau\}=0$ for all $f\in C^\infty(M)$ and
  $\tau\in\Gamma(Q^*)$, this defines a vector bundle morphism
  $\partial_Q\colon Q^*\to Q$ by \eqref{graded_leibniz}:
  $\langle\tau_2, \partial_Q(\tau_1)\rangle=\{\tau_1,\tau_2\}$.  Since
  $\{\tau_1,\tau_2\}=-(-1)^{|\tau_2|}\{\tau_2,\tau_1\}=\{\tau_2,\tau_1\}$,
  we find that $\partial_Q^*=\partial_Q$. The Poisson bracket
  $\{b,f\}$ has degree $0$ and is hence an element of
  $C^\infty(M)$. Again by \eqref{graded_leibniz}, this defines a
  derivation $\{b,\cdot\}\an{C^\infty(M)}$ of $C^\infty(M)$, hence a
  vector field $\rho_B(b)\in \mx(M)$; $\{b,f\}=\rho_B(b)(f)$. By the
  Leibnitz identity \eqref{graded_leibniz} for the Poisson bracket and
  the equality $\{f_1,f_2\}=0$ for all $f_1,f_2\in C^\infty(M)$, we
  get in this manner a vector bundle morphism (an anchor)
  $\rho_B\colon B\to TM$. The bracket $\{b,\tau\}$ has degree $1$ and
  is hence a section of $Q^*$.  Since
  $\{b,f\tau\}=f\{b,\tau\}+\{b,f\}\tau=f\{b,\tau\}+\rho_B(b)(f)\tau$
  and $\{f b,\tau\}=f\{b,\tau\}+\{f,\tau\}b=f\{b,\tau\}$, we find a
  linear $B$-connection $\nabla$ on $Q^*$ by setting
  $\nabla_b\tau=\{b,\tau\}$.  Let us finally look at the bracket
  $\{b_1,b_2\}$. This function has degree $2$ and is hence the sum of
  a section of $B$ and an element of $\Omega^2(Q)$.  We write
  $\{b_1,b_2\}=[b_1,b_2]-R(b_1,b_2)$ with $ [b_1,b_2]\in\Gamma(B)$ and
  $R(b_1,b_2)\in\Omega^2(Q)$. By a similar reasoning as
  before, we find that this defines a skew-symmetric bracket
  $[\cdot\,,\cdot]$ on $\Gamma(B)$ that satisfies a Leibniz equality
  with respect to $\rho_B$, and an element
  $R\in\Omega^2(B,\operatorname{Hom}(Q,Q^*))$ such that $R^*=-R$.
  Note also here that the bracket $\{b,\phi\}$ for
  $\phi\in\Omega^2(Q)\subseteq
  \Gamma(\operatorname{Hom}(Q,Q^*))$ is just
  $\nabla^{\operatorname{Hom}}_b\phi$, where
  $\nabla^{\operatorname{Hom}}$ is the $B$-connection induced on $
  \operatorname{Hom}(Q,Q^*)$ by $\nabla$ and $\nabla^*$.  \medskip

  Now we explain how the dull algebroid structure on $B$ is in
  reality a Lie algebroid structure, and that $(\nabla,\nabla^*,R)$ is
  a self-dual 2-representation of $B$ on $\partial_Q\colon Q\to Q^*$.
  In order to do this, we only need to recall that the Poisson
  structure $\{\cdot\,,\cdot\}$ satisfies the Jacobi identity.  The
  Jacobi identity for the three functions $b_1,b_2,f$ yields the
  compatibility of the anchor on $B$ with the bracket on
  $\Gamma(B)$. The Jacobi identity for $b,\tau_1,\tau_2$ yields
  $\partial_Q\circ \nabla=\nabla^*\circ\partial_Q$, and the Jacobi
  identity for $b_1,b_2,\tau$ yields $R_\nabla=R\circ\partial_Q$. The
  equality $R_{\nabla^*}=\partial_Q\circ R$ follows using
  $\partial_Q=\partial_Q^*$, $R^*=-R$ and $R_\nabla^*=-R_{\nabla^*}$.
  The Jacobi identity for $b_1,b_2,b_3\in\Gamma(B)$ yields in a
  straightforward manner the Jacobi identity for $[\cdot\,,\cdot]$ on
  sections of $\Gamma(B)$ and the equation
  $\dr_{\nabla^{\operatorname{Hom}}}R=0$.  \medskip

  Take conversely a self dual 2-representation of a Lie algebroid $B$
  on a 2-term complex $\partial_Q\colon Q^*\to Q$ and consider the
  $[2]$-manifold $\mathcal M=Q[-1]\oplus B^*[-2]$. Then the self-dual
  2-representation defines as described above a Poisson bracket of
  degree $-2$ on $C^\infty(\mathcal M)$.
\end{proof}

\subsection{(Split) symplectic $[2]$-manifolds}\label{symplectic}
Note that an ordinary Poisson manifold $(M,\{\cdot\,,\cdot\})$ is
symplectic if and only if the vector bundle morphism $\sharp\colon
T^*M\to TM$ defined by $\dr f\mapsto X_f$ is surjective, where
$X_f\in\mx(M)$ is the derivation $\{ f,\cdot\}$.  Alternatively, we
can say that the Poisson manifold is symplectic if the image of the
map $\sharp\colon C^\infty(M)\to \mx(M)$, $f\mapsto \{f,\cdot\}$
generates $\mx(M)$ as a $C^\infty(M)$-module.

In the same manner, if $(\mathcal M, \{\cdot\,,\cdot\})$ is a Poisson
$[n]$-manifold, the map \[\sharp\colon C^\infty(\mathcal M)\to
\der(C^\infty(\mathcal M))\] sends $\xi$ to $\{\cdot,\xi\}$. Then
$(\mathcal M, \{\cdot\,,\cdot\})$ is a \textbf{symplectic
  $[n]$-manifold} if the image of this map generates $\der(C^\infty(\mathcal M))$
as a $C^\infty(\mathcal M)$-module.

\medskip Let $(q_E\colon E\to M, \langle\cdot\,,\cdot\rangle)$ be a
metric vector bundle, i.e.~a vector bundle endowed with a
nondegenerate fiberwise pairing $\langle\cdot\,,\cdot\rangle\colon
E\times_M E\to \R$.  Choose a metric connection $\nabla\colon
\mx(M)\times\Gamma(E)\to\Gamma(E)$.  Then, identifying $E$ with $E^*$
via $\Beta\colon E\to E^*$ given by the metric, we find that the
$2$-representation $(\Id_E\colon E\to E, \nabla, \nabla, R_\nabla)$ is
self-dual (an easy calculation shows that if $\nabla$ is metric, then
$\langle R_\nabla(X_1,X_2)e_1,e_2\rangle=-\langle
R_\nabla(X_1,X_2)e_2,e_1\rangle$ for all $e_1,e_2\in\Gamma(E)$ and
$X_1,X_2\in\mx(M)$).  Consider the split Poisson $[2]$-manifold
$E[-1]\oplus T^*M[-2]$, with the Poisson bracket given by the
self-dual $2$-representation. That is,
the Poisson bracket is given by
\[ \{f_1,f_2\}=0,\quad \{f,e\}=0,\quad
\{e_1, e_2\}=\langle e_1, e_2\rangle,\]
\[\{X,e\}=\nabla_Xe, \quad \{X,f\}=X(f)\]
and 
\[\{X_1,X_2\}=[X_1,X_2]-R_\nabla(X_1,X_2).
\]

Recall from \eqref{der_deg_0} the special derivations that we found on
split $[n]$-manifolds.  The function $\sharp\colon
C^\infty(E[-1]\oplus T^*M[-2])\to
\operatorname{Der}(C^\infty(E[-1]\oplus T^*M[-2]))$ sends a function
$f$ of degree $0$ to $\hat{\dr f}$, a derivation of degree
$-2$. $\sharp$ sends $e$ to $\hat{e}+\dr_\nabla e$,
which is a derivation of degree $-1$. Note that locally,
$\dr_\nabla e\in\Omega^1(M,E)$ can be written as a sum $\sum_i
\nabla_{\partial_{x_i}}e\cdot \hat{\dr x_i}$.  Finally $\sharp$ sends
$X$ to $X+\nabla_X+[X,\cdot]-R(X,\cdot)$,
which is a derivation of degree $0$.  Note that $R(X,\cdot)$ can be
written as $\sum f_{ijk}e_ie_j\hat{\dr x_k}$ 
for some
basis sections $e_1,\ldots,e_n\in\Gamma(E)$ and some functions $f_{ijk}$ in
$C^\infty(M)$.  Hence, since the derivations $\hat{\dr f}$, $\hat e$
and $X+\Beta\circ\nabla_X\circ\Beta\inv+[X,\cdot]$ for $f\in
C^\infty(M)$, $e\in\Gamma(E)$ and $X\in\mx(M)$, span
$\operatorname{Der}(C^\infty(E[-1]\oplus T^*M[-2]))$ as a
$C^\infty(E[-1]\oplus T^*M[-2])$-module, we find as a consequence that
$E[-1]\oplus T^*M[-2]$ is a symplectic $[2]$-manifold.

\medskip 

Conversely, take a split Poisson $[2]$-manifold $Q[-1]\oplus
B^*[-2]$, hence a self-dual $2$-re\-pre\-sen\-ta\-tion $(\partial_Q\colon
Q^*\to Q, \nabla, \nabla^*,R)$ of a Lie algebroid $B$. Then $\sharp
f=\rho_B^*\dr f$ for all $f\in C^\infty(M)$,
$\sharp\tau=\hat{\partial_Q\tau}-\dr_{\nabla^*}\tau$ and $\sharp
b=\rho_B(b)+\nabla_b^*+[b,\cdot]-R(b,\cdot)$. A discussion as the one
above shows that the Poisson structure is symplectic if and only if
$\rho_B\colon B\to TM$ is injective and surjective, hence an
isomorphism and $\partial_Q\colon Q^*\to Q$ is surjective, hence an
isomorphism. The isomorphism $\partial_Q$ identifies then $Q$ with its
dual and $Q$ becomes so a metric vector bundle with the pairing
$\langle q_1, q_2\rangle_Q=\langle \partial_Q\inv(q_1),
q_2\rangle=\{\partial_Q\inv q_1,\partial_Q\inv q_2\}$. Via the
identification $\Beta\inv=\partial_Q\colon Q^*\overset{\sim}{\to} Q$, the linear
connection $\nabla$ is then automatically a metric connection and the
self-dual $2$-representation is $(\Id_Q\colon Q\to
Q,\nabla,\nabla,R_\nabla)$.

\medskip

We have hence found that split symplectic $[2]$-manifolds are
equivalent to self-dual $2$-representations $(\Id_E\colon E\to E,
\nabla, \nabla, R_\nabla)$ defined by a metric vector bundle $E$
together with a metric connection $\nabla$, see also
\cite{Roytenberg02}.

\subsection{Metric VB-algebroids and Poisson involutive double vector bundles}\label{sec:metVBa}
Next we introduce the notions of metric VB-algebroids, involutive
Poisson double vector bundles, and their morphisms.
\begin{definition}\label{metric_VBLA}
 \begin{enumerate}
\item  Let $(\mathbb E;Q,B;M)$ be a metric double vector bundle (with core $Q^*$)
  and assume that $(\mathbb E\to Q, B\to M)$ is a VB-algebroid. Then
  $(\mathbb E\to Q, B\to M)$ is a \textbf{metric VB-algebroid} if the isomorphism
  $\Beta\colon \mathbb E\to \mathbb E\duer B$ is an isomorphism of VB-algebroids.
\item Let $(D,Q,Q,M)$ be an involutive double vector bundle with core
  $B^*$, and let
  $\{\cdot\,,\cdot\}\colon C^\infty(D)\times C^\infty(D)\to
  C^\infty(D)$
  be a Poisson structure on $D$. Then $(D,\{\cdot\,,\cdot\})$ is an
  \textbf{involutive Poisson double vector bundle} if the Poisson
  structure is linear over both sides of $D$ and if $\mathcal I$ is an
  anti-Poisson morphism:
  $\left\{ \mathcal I^*F, \mathcal I^*F'\right\}=-\mathcal I^*\{F,F'\}$ for
    all $F,F'\in C^\infty(D)$.
\item
  A morphism $\Omega\colon D_1\to D_2$ of Poisson involutive double
  vector bundles is a morphism of the underlying involutive double vector
  bundles that is a Poisson map:
$\left\{ \Omega^*F, \Omega^*F'\right\}=\Omega^*\{F,F'\}$
for all $F,F'\in C^\infty(D)$.
\end{enumerate}
\end{definition}
\begin{remark}
  Note that in his thesis \cite{delCarpio-Marek15}, del Carpio-Marek
  defines a Poisson self-conjugate double vector bundle as a
  self-conjugate double vector bundle (see Remark \ref{dcm}) endowed
  with a Poisson structure such that the self-conjugation $\mathcal H$
  is a Poisson morphism. This is, again up to a sign, the same setting
  as the one of Poisson involutive double vector bundles.  Del
  Carpio-Marek recovers independently our results Corollary
  \ref{duality_poissoninv_metricdvb} and Theorem \ref{eq_poisson}
  below.
\end{remark}

Consider a linear VB-algebroid structure $(\mathbb E\to Q, [\cdot\,,\cdot],\Theta\colon
\mathbb E\to TQ)$ on a metric double
vector bundle $(\mathbb E;Q,B;M)$. The linear Lie algebroid structure
defines a linear Poisson structure on $D=\mathbb E\duer Q$:
\begin{equation*}
\begin{split}
\{\ell_{\chi_1},\ell_{\chi_2}\}=\ell_{[\chi_1,\chi_2]}, \quad 
\{\ell_{\chi}, \pi_1^*F\}= \Theta(\chi)(F), \quad \{F_1,F_2\}=0
\end{split}
\end{equation*}
for $\chi,\chi_1,\chi_2\in \Gamma_Q(\mathbb E)$ and
$F,F_1,F_2\in C^\infty(Q)$. This Poisson structure is automatically
also linear over the other side $\pi_2\colon D\to Q$
\cite{Mackenzie11}. 
We will prove below that it is involutive if and
only if the corresponding Lie algebroid structure was metric.
\bigskip

Recall from Theorem \ref{rajan} that linear splittings of
VB-algebroids define $2$-re\-pre\-sen\-ta\-tions.  First we prove that
Lagrangian splittings of metric VB-algebroids correspond to self-dual
$2$-re\-pre\-sen\-ta\-tions.

\begin{proposition}\label{metric_VBLA_via_sa_ruth}
  Let $(\mathbb E\to Q, B\to M)$ be a VB-algebroid with core $Q^*$ and
  assume that $\mathbb E$ is endowed with a linear metric.  Choose a
  Lagrangian splitting of $\mathbb E$ and consider the corresponding
  2-representation of $B$ on $\partial_Q\colon Q^*\to Q$. This
  2-representation is self-dual if and only if $(\mathbb E\to Q, B\to
  M)$ is a metric VB-algebroid.
\end{proposition}

\begin{proof}
  It is easy to see that $\Beta\colon \mathbb E\to \mathbb E\duer B$
  sends core sections $\tau^\dagger\in\Gamma^c_Q(\mathbb E)$ to core
  sections $\tau^\dagger\in\Gamma^c_Q(\mathbb E\duer B)$.  (As always,
  we identify $Q^{**}$ with $Q$ via the canonical isomorphism.)  Let
  $\Sigma\colon B\times_MQ\to \mathbb E$ be a Lagrangian splitting of
  $\mathbb E$.  We have seen in Section \ref{dual} that the map
  $\sigma_B\colon \Gamma(B)\to\Gamma_Q^l(\mathbb E)$ induces a
  horizontal lift $\sigma_B^\star\colon \Gamma(B)\to\Gamma_Q^l(\mathbb
  E\duer B)$.  Recall from Lemma \ref{lem_Lagr_split_sections} that
  $\Beta$ sends then also the linear sections $\sigma_B(b)$ to
  $\sigma_B^\star(b)$, for all $b\in\Gamma(B)$.

\medskip

The double vector bundle $\mathbb E\duer B$ has a VB-algebroid
structure $(\mathbb E\duer B\to Q^{**}, B\to M)$ (see
\S\ref{subsect:VBa}).  Given the splitting $\Sigma^\star\colon
B\times_M Q^{**}\to \mathbb E\duer B$ defined by a Lagrangian
splitting $\Sigma\colon B\times_M Q\to \mathbb E$, the VB-algebroid
structure is given by the dual of the $2$-representation
$(\partial_Q\colon Q^*\to Q, \nabla^Q, \nabla^{Q^*},
R\in\Omega^2(B,\Hom(Q,Q^*)))$, i.e.
\begin{equation*}
\begin{split}
  \rho_{\mathbb E\duer B}(\tau^\dagger)&=(\partial_Q^*\tau)^\uparrow \in \mx^c(Q^{**}),\qquad \rho_{\mathbb E\duer B}(\sigma_B^\star(b))=\widehat{{\nabla^{Q^*}}^*}\in \mx^l(Q^{**}),\\
  \left[\sigma_B^\star(b), \tau^\dagger\right]&=({\nabla^Q}^*_b\tau)^\dagger, \quad \text{ and }\quad \left[\sigma_B^\star(b_1),
    \sigma_B^\star(b_2)\right]=\sigma_B^\star[b_1,b_2]+\widetilde{R(b_1,b_2)^*}
\end{split}
\end{equation*}
(see \S\ref{dual_and_ruths}).
This shows immediately that $\Beta$ is an isomorphism of VB-algebroids
over the canonical isomorphism $Q\to Q^{**}$ if and only if the $2$-representation
\[(\partial_Q\colon Q^*\to Q, \nabla^Q, \nabla^{Q^*},
R\in\Omega^2(B,\Hom(Q,Q^*)))\] is self-dual.
\end{proof}

\begin{corollary}\label{cor_1_met}
  Consider a metric double vector bundle $(\mathbb E, B,Q,M)$ endowed
  with a linear Lie algebroid structure on $\mathbb E\to Q$ over $B\to
  M$. Then the Lie algebroid structure is metric if and only if
  $\mathcal C(\mathbb E)$ is closed under the Lie bracket.
\end{corollary}
\begin{proof}
  Fix a Lagrangian splitting of $\mathbb E$ and consider the
  corresponding $2$-representation $(\nabla^Q,
  \nabla^{Q^*},R\in\Omega^2(B,\operatorname{Hom}(Q,Q^*)))$ of $B$ on
  $\partial_Q\colon Q^*\to Q$.  The equation
  $[\sigma_B(b_1),\sigma_B(b_2)]=\sigma_B[b_1,b_2]-\widetilde{R(b_1,b_2)}$
  shows that $[\sigma_B(b_1),\sigma_B(b_2)]\in\mathcal C(\mathbb E)$
  for all $b_1,b_2\in\Gamma(B)$ if and only if
  $R\in\Omega^2(B,Q^*\wedge Q^*)$.  

An easy computation shows that
  $\left[\sigma_B(b), \widetilde{\phi}\right]=\widetilde{\nabla_b^{\rm
      Hom}\phi}$ is an element of $\mathcal C(\mathbb E)$ for all
  $b\in\Gamma(B)$ and $\phi\in\Gamma(Q^*\wedge Q^*)$ if and only if
  $(\nabla^Q)^*=\nabla^{Q^*}$.

  Finally,
  $\left[\widetilde{\phi_1},\widetilde{\phi_2}\right]=\widetilde{\phi_2\circ\partial_Q\circ\phi_1-\phi_1\circ\partial_Q\circ\phi_2}\in\mathcal
  C(\mathbb E)$ if and only if
  $\phi_2\circ\partial_Q\circ\phi_1-\phi_1\circ\partial_Q\circ\phi_2\in\Gamma(Q^*\wedge
  Q^*)$. This is the case for all $\phi_1,\phi_2\in\Gamma(Q^*\wedge
  Q^*)$ if and only if $\partial_Q=\partial_Q^*$.
\end{proof}

\begin{corollary}\label{duality_poissoninv_metricdvb}
Let $(D;Q,Q;M)$ be an involutive double vector bundle and consider
  the dual metric double vector bundle
  $(\mathbb E={D}\duer{\pi_1}, B,Q,M)$. A linear Lie algebroid
  structure on $\mathbb E\to Q$ is metric if and only if the dual
  linear Poisson structure on $D$ is involutive.
\end{corollary}

\begin{proof}
  We need to find
  \begin{equation}\label{I_anti_p}
\{\mathcal I^*(F_1), \mathcal I^*(F_2)\}=-\mathcal I^*\{F_1,F_2\}
\end{equation}
for all $F_1, F_2\in C^\infty(D)$.  Since $\mathcal C(\mathbb E)$ and
$\Gamma_Q^c(\mathbb E)$ span pointwise $\mathbb E$, it is sufficient
to check \eqref{I_anti_p} on functions $\ell_\chi$,
$\ell_{\tau^\dagger}=\pi_2^*\ell_\tau$, $\pi_1^*\ell_\tau$ and
$\pi_1^*q_Q^*f$ for $\chi \in \mathcal C(\mathbb E)$,
$\tau\in\Gamma(Q^*)$ and $f\in C^\infty(M)$. Using Lemma
\ref{inv_cois_sections}, it is easy to see that \eqref{I_anti_p} is
trivially satisfied on $\pi_1^*q_Q^*f_1$ and $\pi_1^*q_Q^*f_2$, on
$\pi_1^*q_Q^*f$ and $\pi_1^*\ell_\tau$, and equivalently on
$\pi_1^*q_Q^*f$ and $\ell_{\tau^\dagger}$, on $\pi_1^*\ell_{\tau_1}$ and $\pi_1^*\ell_{\tau_2}$
and equivalently on $\ell_{\tau_1^\dagger}$ and
$\ell_{\tau_2^\dagger}$ for $f,f_1,f_2\in C^\infty(M)$ and
$\tau,\tau_1,\tau_2\in\Gamma(Q^*)$.

The equalities
\[\{\mathcal I^*\ell_{\tau_1^\dagger}, \mathcal
I^*\pi_1^*\ell_{\tau_2}\}=\{\pi_1^*\ell_{\tau_1},
\ell_{\tau_2^\dagger}\}=-\{\ell_{\tau_2^\dagger},
\pi_1^*\ell_{\tau_1}\}=-\pi_1^*(\partial_Q\tau_2)^\dagger(\ell_{\tau_1})=-\pi_1^*q_Q^*\langle\partial_Q\tau_2,\tau_1\rangle\]
and 
\[\mathcal I^*\{\ell_{\tau_1^\dagger},\pi_1^*\ell_{\tau_2}\}=\mathcal
I^*\pi_1^*q_Q^*\langle\partial_Q\tau_1,\tau_2\rangle=\pi_1^*q_Q^*\langle\partial_Q\tau_1,\tau_2\rangle\]
show that 
$\{\mathcal I^*\ell_{\tau_1^\dagger}, \mathcal
I^*\pi_1^*\ell_{\tau_2}\}=-\mathcal
I^*\{\ell_{\tau_1^\dagger},\pi_1^*\ell_{\tau_2}\}$ for all
$\tau_1,\tau_2\in\Gamma(Q^*)$ if and only if
$\partial_Q=\partial_Q^*$.
Further, we find $\{\mathcal I^*\ell_\chi,\mathcal
I^*\pi_1^*q_Q^*f\}=-\{\ell_\chi,\pi_1^*q_Q^*f\}=-\pi_1^*q_Q^*\rho_B(b)(f)
=-\mathcal I^*\pi_1^*q_Q^*\rho_B(b)(f)=-\mathcal
I^*\{\ell_\chi,\pi_1^*q_Q^*f\}$ for $f\in C^\infty(M)$ and
$\chi\in\mathcal C(\mathbb E)$.

Finally, we have $\{\mathcal I^*\ell_{\chi_1},\mathcal
I^*\ell_{\chi_2}\}=\{\ell_{\chi_1},\ell_{\chi_2}\}=\ell_{[\chi_1,\chi_2]}$
and $-\mathcal I^*\{\ell_{\chi_1},\ell_{\chi_2}\}=-\mathcal
I^*\ell_{[\chi_1,\chi_2]}=\ell_{[\chi_1,\chi_2]}$ if and only if
$[\chi_1,\chi_2]\in\mathcal C(\mathbb E)$. Hence, we can conclude
using Corollary \ref{cor_1_met}.
\end{proof}

\begin{remark}
Note that a linear Poisson structure on $D$ is involutive if
and only if $\mathcal P(D)$ defined in \S\ref{geom_int} is closed
under the Poisson bracket: by \eqref{I_anti_p}, if $\mathcal
I^*F_1=-F_1$ and $\mathcal I^*F_2=-F_2$, then $\mathcal I^*\{F_1,F_2\}=-\{\mathcal I^*F_1,\mathcal I^*F_2\}=-\{F_1,F_2\}$.  The
Poisson bracket of $\pi_2^*\ell_{\tau_1}-\pi_1^*\ell_{\tau_1}$ with
$\pi_2^*\ell_{\tau_2}-\pi_1^*\ell_{\tau_2}$ vanishes and the Poisson
bracket of $\ell_{\sigma_B(b)+\widetilde{\omega}}$ with
  $\pi_2^*\ell_{\tau}-\pi_1^*\ell_{\tau}$ is
  $\pi_2^*\ell_{\nabla_b\tau-\ip{\partial_Q\tau}\omega}-\pi_1^*\ell_{\nabla_b\tau-\ip{\partial_Q\tau}\omega}$. 
The rest follows as in the proof of the preceding theorem.
\end{remark}

\subsection{Equivalence of Poisson $[2]$-manifolds with Poisson involutive
double vector bundles.}
The functors found in Section \ref{eq_2_manifolds} between the
category of metric double vector bundles and the category of
$[2]$-manifolds induce functors between the category of metric
VB-algebroids and the category of Poisson $[2]$-manifolds.

\medskip
  Let $(\mathcal M, \{\cdot\,,\cdot\})$ be a Poisson $[2]$-manifold
  and consider the involutive double vector bundle $(\mathcal
  G(\mathcal M), E_1,E_1,M) $ with core $E_2$ corresponding to
  $\mathcal M$ as in \S\ref{Geom_functor}. Then the isotropic linear
  sections $\mathcal C(\mathbb E)$ can be identified with the degree
  $2$ functions on $\mathcal M$ and the core sections $\Gamma_Q^c(
  \mathbb E)$ can be identified with the degree $1$ functions on
  $\mathcal M$.  Since the sections $\mathcal C(\mathbb E)\cup
  \Gamma_Q^c( \mathbb E)$ span $\mathbb E\to Q$, the Poisson bracket on $C^\infty(\mathcal M)$ defines
   a linear Poisson bracket on $\mathcal G(\mathcal M)$:
\begin{equation}\label{geom_functor_poisson}
\begin{split}
  \{\ell_{\chi_1},\ell_{\chi_2}\}&=\ell_{\{\chi_1,\chi_2\}},\qquad \{\ell_\chi, \pi_1^*q_Q^*f\}=\pi_1^*q_Q^*\{\chi,f\}, \qquad \{\ell_\chi, \pi_1^*\ell_\tau\}=\pi_1^*\ell_{\{\chi,\tau\}}\\
  \{\ell_\chi,\ell_{\tau^\dagger}\}&=\ell_{\{\chi,\tau\}^\dagger}, \qquad \{\ell_{\tau_1^\dagger}, \ell_{\tau_2^\dagger}\}=0, \qquad \{\ell_{\tau_1^\dagger},\pi_1^*\ell_{\tau_2}\}=\pi_1^*q_Q^*{\{\tau_1,\tau_2\}}\\
  \{\pi_1^*\ell_{\tau_1}, \pi_1^*\ell_{\tau_2}\}&=0, \qquad
  \{\ell_{\tau^\dagger},\pi_1^*q_Q^*f\}=0, \qquad
  \{\pi_1^*\ell_\tau,\pi_1^*q_Q^*f\}=0, \,\text{ and }\,
  \{\pi_1^*q_Q^*f_1,\pi_1^*q_Q^*f_2\}=0
\end{split}
\end{equation}
for all $f,f_1,f_2\in C^\infty(M)$, $\tau,\tau_1,\tau_2\in\Gamma(E_1^*)$
and $\chi,\chi_1,\chi_2\in\mathcal C(\mathbb E)$. Per definition, the
involution $\mathcal I$ on $\mathcal G(\mathcal M)$ is anti-Poisson,
and so $(\mathcal G(\mathcal M), \mathcal I, \{\cdot\,,\cdot\})$ is a
Poisson involutive double vector bundle.

Let now $\mu\colon \mathcal M_1\to\mathcal M_2$ be a morphism of
Poisson $[2]$-manifolds. By the definition of the Poisson bracket in
\eqref{geom_functor_poisson},
the morphism $\mathcal G(\mu)\colon \mathcal G(\mathcal
M_1)\to\mathcal G(\mathcal M_2)$ is automatically a Poisson
morphism. For example, for functions $\chi_1,\chi_2\in
C^\infty(\mathcal M_2)$ of degree $2$, we have 
\begin{equation*}
\begin{split}
\{\mathcal G(\mu)^*\ell_{\chi_1},\mathcal G(\mu)^*\ell_{\chi_1}\}
&=\{\ell_{\mu^\star\chi_1},\ell_{\mu^\star\chi_2}\}=\ell_{\{\mu^\star\chi_1,\mu^\star\chi_2\}}=\ell_{\mu^\star\{\chi_1,\chi_2\}}\\
&=\mathcal
  G(\mu)^*\ell_{\{\chi_1,\chi_2\}}=\mathcal
  G(\mu)^*\{\ell_{\chi_1},\ell_{\chi_2}\}.
\end{split}
\end{equation*}
We let the reader check the other cases.  Hence, the functor
$\mathcal G$ induces a functor $\mathcal G_P$ from the category of
Poisson $[2]$-manifolds to the category of Poisson involutive double
vector bundles.

\bigskip

Conversely, we consider a Poisson involutive double vector bundle
$(D,Q,Q,M;\{\cdot\,,\cdot\})$ with core $B^*$, or
equivalently, a metric VB-algebroid $(D\duer{\pi_1}=:\mathbb E\to Q,
B\to M)$. Consider the image $\mathcal M(D)$ of $D$ under the functor
$\mathcal M$. Then $\mathcal M(D)$ is the $[2]$-manifold which degree
$0$ functions are the elements of $C^\infty(M)$, which degree $1$
functions are the elements of $\Gamma(Q^*)$ and which degree
$2$-functions are the elements of $\mathcal C(\mathbb E)$, with
$\tau_1\wedge\tau_2=\widetilde{\tau_1\wedge\tau_2}\in \mathcal
C(\mathbb E)$ for $\tau_1,\tau_2\in\Gamma(Q^*)$. We define a Poisson
bracket on $\mathcal M(D)$ by 
\begin{equation}\label{alg_functor_poisson}
\begin{split}
  \{\chi_1,\chi_2\}&=[\chi_1,\chi_2], \quad
  \{\chi,\tau\}^\dagger=[\chi,\tau^\dagger], \\
  \{\tau_1,\tau_2\}&=\langle\tau_1,\partial_Q\tau_2\rangle, \quad
  \{\chi, f\}=\rho_B(b)(f),\quad \{\tau,f\}=\{f_1,f_2\}=0
\end{split}
\end{equation}
on the generators $\chi,\chi_1,\chi_2\in\mathcal C(\mathbb E)$, with $\chi$ linear
over $b\in\Gamma(B)$, $f,f_1,f_2\in C^\infty(M)$ and
$\tau,\tau_1,\tau_2\in \Gamma(Q^*)$, and by graded symmetry and graded
Leibniz extension to all functions on $\mathcal M$. Clearly, this graded
bracket has degree $-2$ and is well-defined. Let $\Omega\colon D_1\to
D_2$ be a morphism of Poisson involutive double vector bundle. In a
similar manner as above, it is easy to check that $\mathcal
M(\Omega)\colon\mathcal M(D_1)\to\mathcal M(D_2)$
is a morphism of Poisson $[2]$-manifolds. Hence, $\mathcal M$ induces
a functor $\mathcal M_P$ from the category of Poisson involutive
double vector bundles to the category of Poisson $[2]$-manifolds.

\bigskip

The two functors $\mathcal M_P$ and $\mathcal G_P$ define together an
equivalence of the category of Poisson involutive double vector bundle
with the category of Poisson $[2]$-manifolds. Hence, we have proved the following theorem.

\begin{theorem}\label{eq_poisson}
  The functors $\mathcal M_P$ and $\mathcal G_P$ define together an
  equivalence of the category of Poisson involutive double vector
  bundle with the category of Poisson $[2]$-manifolds.
\end{theorem}

\subsection{Examples}

We conclude by discussing three important classes of examples.
\subsubsection{Tangent doubles of metric vector bundles vs symplectic $[2]$-manifolds }\label{tangent_euclidean}
Consider a metric vector bundle $E\to M$ and a metric connection 
$\nabla\colon\mx(M)\times\Gamma(E)\to\Gamma(E)$. The double tangent
\begin{equation*}
\begin{xy}
\xymatrix{
TE \ar[d]_{Tq_E}\ar[r]^{p_E}& E\ar[d]^{q_E}\\
 TM\ar[r]_{p_M}& M}
\end{xy}
\end{equation*}
has a VB-algebroid structure $(TE\to E;TM\to M)$ and a linear metric\linebreak
$\langle\cdot\,,\cdot\rangle\colon TE\times_{TM}TE\to \R$ defined as
in Example \ref{metric_connections}.  

Recall that Lagrangian linear splittings of $TE$ are equivalent to metric
connections $\nabla\colon \mx(M)\times\Gamma(E)\to\Gamma(E)$.  In other
words, $\nabla=\nabla^*$ when $E^*$ is identified with $E$ via the
non-degenerate pairing. The $2$-representation $(\Id_E\colon E\to E,
\nabla, \nabla, R_\nabla)$ defined by the Lagrangian splitting
$\Sigma^\nabla\colon E\times_M TM\to TE$ and the VB-algebroid $(TE\to
E, TM\to M)$ is then self-dual (see also \S\ref{symplectic}).

The Poisson $[2]$-manifold $\mathcal M(TE)$ associated to $TE$ is
given as follows. The functions of degree $0$ are elements of
$C^\infty(M)$, the functions of degree $1$ are sections of $E$ ($E$ is
identified with $E^*$ via the isomorphism $\Beta\colon E\to E^*$
defined by the pairing) and the functions of degree $2$ are the vector
fields $\widehat{\delta}\in\mx(E)$ for a derivation $\delta$ of $E$ over
$X\in\mx(M)$, that preserves the pairing. The Poisson bracket is given
by $\{\widehat{\delta_1},\widehat{\delta_2}\}=\widehat{[\delta_1,\delta_2]}$, $\{\widehat{\delta},
e\}=\delta(e)$ and $\{\widehat{\delta}, f\}=X(f)$, $\{e_1,e_2\}=\langle
e_1, e_2\rangle$, and $\{e, f\}=\{f_1,f_2\}=0$ for all
$e,e_1,e_2\in\Gamma(E)$, $f,f_1,f_2\in C^\infty(M)$ and $\widehat{\delta},
\widehat{\delta_1}, \widehat{\delta_2}\in \mx^{\langle\cdot\,,\cdot\rangle, l}(E)$.  The Poisson
$[2]$-manifold $\mathcal M(TE)$ splits as the split Poisson
$[2]$-manifold described in \S\ref{symplectic}. It is hence
symplectic.

\bigskip Let $N$ be a smooth manifold. Then $T^*N$ carries the
canonical symplectic structure $\omega_{\rm can}$ given by
$\omega_{\rm can}=-\dr\theta_{\rm can}$ with $\theta_{\rm can}\in
\Omega^1(T^*N)$ given by $\theta_{\rm can}(v_{\alpha_n})=\langle
\eta_n, T_{\eta_n}c_Nv_{\alpha_n}\rangle$, where $c_N\colon T^*N\to N$
is the canonical projection. Each vector field $X\in\mx(N)$, the
derivation $\ldr{X}$ of $T^*N$ defines a linear vector field
$\widehat{\ldr{X}}\in\mx^l(T^*N)$. Further, each $1$-form $\eta\in
\Omega^1(N)$ defines a vertical vector field $\eta^\uparrow\in
\mx^c(T^*N)$. It is easy to see that 
\[\ip{\widehat{\ldr{X_2}}}\ip{\widehat{\ldr{X_1}}}\omega_{\rm can}=-\ell_{[X_1,X_2]}, \quad \ip{\eta^\uparrow}\ip{\widehat{\ldr{X}}}\omega_{\rm can}=c_N^*\langle \eta,X\rangle, \quad 
\ip{\eta_1^\uparrow}\ip{\eta_2^\uparrow}\omega_{\rm can}=0
\]
for $X,X_1,X_2\in\mx(N)$ and $\eta,\eta_1,\eta_2\in\Omega^1(N)$.  Next we
apply this to the smooth manifold $E$ to compute the canonical
symplectic form on $T^*E$.

The involutive double vector bundle that is dual to $TE$ is $T^*E$
with sides $E$ and $E\simeq E^*$ and core $T^*M$ (see Example
\ref{inv_dvb_metric_connections}). Recall that $\mathcal
C(TE)\subseteq \mx^l(E)$ is here the set of linear vector fields on
$E$ defined by derivations of $E$ that preserve the pairing. We
consider $\widehat{\ldr{\widehat{\delta}}}\in \mx(T^*E)$ and
$\widehat{\ldr{e^\uparrow}}\in\mx(T^*E)$ for $\widehat{\delta}\in \mathcal
C(TE)$ and $e\in\Gamma(E)$. We consider the $1$-forms $\dr\ell_e,
q_e^*\dr f\in \Omega^1(E)$ for $e\in\Gamma(E)$ and $f\in C^\infty(M)$,
and so the induced vector fields $(\dr\ell_e)^\uparrow, (q_E^*\dr
f)^\uparrow\in \mx(T^*E)$.  Recall also that $\widehat{\delta}\in \mathcal
C(TE)$ and $e^\uparrow\in\mx^c(E)$ define linear functions on $T^*E$,
and that $c_E^*\ell_e$ and $c_E^*q_E^*f\in C^\infty(T^*E)$. The considerations above
show that
\begin{equation*}
\begin{split}
  \omega_{\rm can}\left(\widehat{\ldr{\widehat{\delta_1}}},
    \widehat{\ldr{\widehat{\delta_2}}}\right)&=\ell_{\left[\widehat{\delta_2},\widehat{\delta_1}\right]}=\left\{\ell_{\widehat{\delta_2}},
    \ell_{\widehat{\delta_1}}\right\}, \qquad \omega_{\rm can}\left(\widehat{\ldr{\widehat{\delta}}},
    \widehat{\ldr{e^\uparrow}}\right)=-\ell_{\delta(e)^\uparrow}=\{\ell_{e^\uparrow},\ell_{\widehat{\delta}}\},\\
\omega_{\rm can}\left(\widehat{\ldr{\widehat{\delta}}},
   (\dr\ell_e)^\uparrow \right)&=c_E^*\ell_{\delta(e)}=\{\ell_{\widehat{\delta}},c_E^*\ell_{e}\}, \,
\omega_{\rm can}\left(\widehat{\ldr{\widehat{\delta}}},
   (q_E^*\dr f)^\uparrow \right)=c_E^*q_E^*X(f)=\left\{ \ell_{\widehat{\delta}}, c_E^*q_E^*f\right\},\\
\omega_{\rm can}\left(\widehat{\ldr{e_1^\uparrow}},
    \widehat{\ldr{e_2^\uparrow}}\right)&=0=\{\ell_{e_1^\uparrow},\ell_{e_2^\uparrow}\}, \qquad  \omega_{\rm can}\left(\widehat{\ldr{e_1^\uparrow}},
   (\dr\ell_e)^\uparrow \right)=c_E^*q_E^*\langle e_1,e_2\rangle=\left\{\ell_{e_1^\uparrow}, c_E^*\ell_{e_2}\right\},\\
\end{split}
\end{equation*}
and that $\omega_{\rm can}$ vanish on any other combincation of two
vector fields of the four type. This shows that $\omega_{\rm
  can}^\flat(\widehat{\ldr{\widehat{\delta}}})=\dr\ell_{\widehat{\delta}}$,
$\omega_{\rm
  can}^\flat(\widehat{\ldr{e^\uparrow}})=\dr\ell_{e^\uparrow}$,
$\omega_{\rm can}^\flat((\dr\ell_e)^\uparrow)=-\dr(c_E^*\ell_{e})$ and
$\omega_{\rm can}^\flat((q_E^*\dr f)^\uparrow)=-\dr(c_E^*q_E^*f)$. An
easy computation shows then that the Poisson structure constructed in
\eqref{geom_functor_poisson} equals the Poisson structure on
$C^\infty(T^*E)$ that is induced by $-\omega_{\rm can}$.

Thus, we have found that the equivalence found in Theorem
\ref{eq_poisson} restricts to an equivalence of symplectic
$[2]$-manifolds with symplectic cotangent doubles of metric vector bundles (see
also \cite{Roytenberg02}).

\subsubsection{The metric double of a VB-algebroid}\label{metric_double_VB_alg}
Take a VB-algebroid $(D\to A, B \to M)$ with core $C$ and a linear
splitting $\Sigma\colon A\times_M B\to D$.  Let $(\partial_A\colon
C\to A, \nabla^A,\nabla^C, R\in \Omega^2(B,\operatorname{Hom}(A,C)))$
be the $2$-representation of the Lie algebroid $B$ that is induced by
$\Sigma$.  Recall from \eqref{lemma_dual_splitting} that the splitting
$\Sigma$ induces a splitting $\Sigma^\star\colon B\times_M C^*\to
D\duer B$, and from \S\ref{dual_and_ruths} that $(D\duer B\to C^*,
B\to M)$ has an induced VB-algebroid structure given in this splitting
by the $2$-representation \[(\partial_A^*\colon A^*\to C^*,
{\nabla^A}^*,{\nabla^C}^*, -R^*\in
\Omega^2(B,\operatorname{Hom}(C^*,A^*))).\]

The direct sum $D\oplus_BD\duer B$ over $B$ \begin{equation*}
\begin{xy}
\xymatrix{
D\oplus_BD\duer B \ar[d]\ar[r]& B\ar[d]\\
A\oplus C^*\ar[r]& M}
\end{xy}
\end{equation*}
has then a VB-algebroid structure $(D\oplus_BD\duer B\to A\oplus
C^*, B\to M)$ with core $C\oplus A^*$. It is easy to see that $\Sigma$
and $\Sigma^\star$ define a linear splitting $\tilde\Sigma\colon
B\times_M (A\oplus C^*)\to D\oplus_BD\duer B$, $\tilde \Sigma(b_m,
(a_m,\gamma_m))=(\Sigma(a_m,b_m),\Sigma^\star(b_m,\gamma_m))$.
The induced $2$-representation is 
\[(\partial_A\oplus\partial_A^*\colon C\oplus A^*\to A\oplus C^*, \nabla^A\oplus{\nabla^C}^*,\nabla^C\oplus{\nabla^A}^*,
R\oplus(-R^*)),
\]
a self-dual $2$-representation of the Lie algebroid $B$. This gives us
a new class of examples of (split) Poisson $2$-manifolds induced from
ordinary $2$-representations or VB-algebroids.  Note that the
splittings of $D\oplus D\duer B$ obtained as above are not the only
Lagrangian splittings, and that the example of $(TA\oplus T^*A\to
TM\oplus A^*, A\to M)$ discussed in the next example and in \cite{Jotz13a} is a
special case.

\subsubsection{The Pontryagin algebroid over a Lie algebroid}\label{pont_VB_LA}
If $A$ is a Lie algebroid, then \linebreak $(TA\to TM, A\to M)$ is a VB-algebroid
and since $TA\duer A=T^*A$, the double vector bundle $T^*A$ has a
VB-algebroid structure $(T^*A\to A^*, A\to M)$ with core $T^*M$.  As a
consequence, the direct sum $TA\oplus T^*A$ over $A$ has a
VB-algebroid structure $(TA\oplus T^*A\to TM\oplus A^*, A\to
M)$. Recall from Example \ref{met_TET*E} that $(TA\oplus T^*A;TM\oplus
A^*,A;M)$ has also a natural linear metric, which is given by
\eqref{standard_pairing}.

Recall from Example \ref{met_TET*E} that linear splittings of
$TA\oplus T^*A$ are in bijection with dull brackets on sections of
$TM\oplus A^*$, and so also with Dorfman connections\linebreak $\Delta\colon
\Gamma(TM\oplus A^*)\times\Gamma(A\oplus T^*M)\to\Gamma(A\oplus
T^*M)$. We give in \cite{Jotz13a} the $2$-representation
$((\rho,\rho^*)\colon A\oplus T^*M\to TM\oplus A^*, \nabla^{\rm
  bas},\nabla^{\rm bas}, R_\Delta^{\rm bas})$ of $A$ that is defined
by the VB-algebroid $(TA\oplus T^*A\to TM\oplus A^*, A\to M)$ and any
such Dorfman connection: The connections $\nabla^{\rm bas}\colon
\Gamma(A)\times\Gamma(A\oplus T^*M) \to\Gamma(A\oplus T^*M)$ and
$\nabla^{\rm bas}\colon \Gamma(A)\times\Gamma(TM\oplus A^*)
\to\Gamma(TM\oplus A^*)$ are
\begin{equation*}
\nabla^{\rm bas}_a(X,\alpha)=(\rho,\rho^*)(\Omega_{(X,\alpha)}a)+\ldr{a}(X,\alpha)
\quad \text{ and } \quad \nabla^{\rm
  bas}_a(b,\theta)=\Omega_{(\rho,\rho^*)(b,\theta)}a+\ldr{a}(b,\theta), 
\end{equation*}
where $\Omega\colon \Gamma(TM\oplus A^*)\times\Gamma(A)\to\Gamma(A\oplus T^*M)$ is defined by
\[\Omega_{(X,\alpha)}a=\Delta_{(X,\alpha)}(a,0)-(0,\dr\langle\alpha, a\rangle)
\] and for $a\in\Gamma(A)$, the derivations $\ldr{a}$ over $\rho(a)$
are defined by:
\[\ldr{a}\colon \Gamma(A\oplus T^*M)\to\Gamma(A\oplus T^*M),\quad \ldr{a}(b,\theta)=([a,b], \ldr{\rho(a)}\theta)\]
and 
\[\ldr{a}\colon \Gamma(TM\oplus A^*)\to\Gamma(TM\oplus A^*), \quad \ldr{a}(X,\alpha)=([\rho(a),X], \ldr{a}\alpha).\]
We prove in \cite{Jotz13a} that the two connections above are dual to
each other if and only if the dull bracket dual to $\Delta$ is
skew-symmetric. Hence, the two connections are dual to each other if
and only if the chosen linear splitting is Lagrangian (see Example
\ref{met_TET*E}).  The basic curvature
$R_\Delta^{\rm bas}\colon \Gamma(A)\times\Gamma(A)\times\Gamma(TM\oplus A^*)\to\Gamma(A\oplus T^*M)$
is given by 
\begin{align*}
R_\Delta^{\rm bas}(a,b)(X,\xi)=&-\Omega_{(X,\xi)}[a,b] +\ldr{a}\left(\Omega_{(X,\xi)}b\right)-\ldr{b}\left(\Omega_{(X,\xi)}a\right)\\
&\qquad                                                     + \Omega_{\nabla^{\rm bas}_b(X,\xi)}a-\Omega_{\nabla^{\rm bas}_a(X,\xi)}b.
\end{align*}
Assume that the linear splitting is Lagrangian.  A relatively long but
straightforward computation shows that ${R_\Delta^{\rm
    bas}}^*=-R_\Delta^{\rm bas}$, and so that the $2$-representation
is self-dual. Hence $(TA\oplus T^*A\to TM\oplus A^*, A\to M)$ is a
metric VB-algebroid.

\appendix

\section{Proof of Lemma \ref{bundlemap_eq_to_morphism}}\label{morphisms_VB_sheaves}
Let $\omega\colon A\to B$ be a vector bundle morphism over a smooth
map $\omega_0\colon M\to N$.  The morphism $\omega$ induces then a
vector bundle morphism $\omega^!\colon A\to \omega_0^*B$,
$\omega^!(a_m)=(m,\omega(a_m))$ over the identity on $M$. For a
section $b\in\Gamma_V(B)$, we get in a similar manner a section
$\omega_0^!b\in\Gamma_{\omega_0\inv(V)}(\omega_0^*B)$; defined by
$(\omega_0^!b)(m)=(m,b(\omega_0(m)))$ for all $m\in \omega_0\inv(V)$.

\medskip

In order to prove Lemma \ref{bundlemap_eq_to_morphism} we first check
that $\omega^\star$ has the specified codomain, that is, that the
image under $\omega^\star$ of a smooth section of $B^*$ is again
smooth. Consider the pullback of $B$ under $\omega_0$, i.e.~the vector
bundle $\omega_0^*B\to M$. We have
$(\omega_0^*B)^*\simeq\omega_0^*B^*$ and the smoothness of
$\omega^\star(\beta)$ follows from the equality
$\omega^\star(\beta)=(\omega^!)^*(\omega_0^!\beta)$: for each $m\in
M$, and each $a_m\in A_m$, we have
\begin{align*}
  \langle((\omega^!)^*(\omega_0^!\beta))(m), a_m\rangle&=\langle(\omega_0^!\beta)(m),
  \omega^!(a_m)\rangle
  =\langle(m,\beta(\omega_0(m))),(m,\omega(a_m))\rangle\\
  &=\langle\beta(\omega_0(m)),
  \omega(a_m)\rangle=\langle\omega^\star(\beta)(m),a_m\rangle.
\end{align*}

The map $\omega^\star$ is obviously a morphism of modules over
$\omega_0^*$: for $\beta\in\Gamma(B^*)$ and $f\in C^\infty(N)$, we
easily find $\omega^\star(f\cdot \beta)=\omega_0^*f\cdot\omega^\star(\beta)$.

Next we need to show that a morphism $\mu^\star\colon
\Gamma(B^*)\to\Gamma(A^*)$ of modules over \linebreak $\mu_0^*\colon
C^\infty(N)\to C^\infty(M)$, for $\mu_0\colon M\to N$ smooth, induces
a morphism $A\to B$ of vector bundles over $\mu_0\colon M\to N$.
Choose $a_m$ in the fiber of $A$ over $m$ and define $\mu(a_m)\in
B_{\mu_0(m)}$ by
\[\langle \beta(\mu_0(m)), \mu(a_m)\rangle=\langle\mu^\star(\beta)(m), a_m\rangle 
\]
for all $\beta\in \Gamma(B^*)$.

The smoothness of $\mu$ is seen as follows: let $b_1,\ldots,b_n$ be
local basis fields for $B$ and let $\beta_1,\ldots,\beta_n$ be the
dual basis fields. Then for each $a_m\in A$, $\mu(a_m)$ can be written
$\sum_{i=1}^n\langle\mu(a_m),\beta_i(\mu_0(m))\rangle b_i(\mu_0(m))$.
Since each $\langle\mu(a_m),\beta_i(\mu_0(m))\rangle$ equals
$\ell_{\mu^\star(\beta_i)}(a_m)$, we find that locally,
$\mu=\sum_{i=1}^n\ell_{\mu^\star(\beta_i)}\cdot (b_i\circ\mu_0\circ
q_A)$.  To prove that $\mu$ is a vector bundle morphism, we need to
check that $\langle \beta(\mu_0(m)), \mu(a_m)\rangle$ only depends on
the value of $\beta$ at $\mu_0(m)$, or in other words, that if $\beta$
vanishes at $\mu_0(m)$, then $\langle \beta(\mu_0(m)),
\mu(a_m)\rangle=0$. Without loss of generality, assume that $\beta$
can be written as $f\cdot \beta'$ with $\beta'\in \Gamma(B^*)$ and
$f\in C^\infty(N)$ with $f(\mu_0(m))=0$.  Then $\langle
\beta(\mu_0(m)), \mu(a_m)\rangle=\langle
f(\mu_0(m))\mu^\star(\beta')(m), a_m\rangle=0$.  The morphism $\mu$ of
vector bundles clearly induces $\mu^\star$ on the sets of sections of
the duals, and vice-versa.

\def\cprime{$'$} \def\polhk#1{\setbox0=\hbox{#1}{\ooalign{\hidewidth
  \lower1.5ex\hbox{`}\hidewidth\crcr\unhbox0}}} \def\cprime{$'$}
  \def\cprime{$'$}

\bigskip
\noindent
\address{Mathematisches Institut,\\
Georg-August Universit\"at G\"ottingen.\\
\email{madeleine.jotz-lean@mathematik.uni-goettingen.de}}

\end{document}